\documentclass{amsart}[12pt,letter]

\usepackage{color,amssymb,enumerate}
\usepackage{mathabx}
\usepackage{thmtools}
\usepackage{thm-restate}

\newsavebox\tempbox
\let\svwidetilde\widetilde
\renewcommand\widetilde[1]{\sbox\tempbox{$#1$}\svwidetilde{\usebox{\tempbox}}}

\usepackage{mathtools}

\usepackage{tikz}
\usetikzlibrary{arrows}

\usepackage{comment}

\usepackage[small,nohug,heads=vee]{diagrams}
\diagramstyle[labelstyle=\scriptstyle]

\usepackage{hyperref}
\hypersetup{colorlinks=true}

\usepackage[all]{hypcap}
\usepackage[alphabetic]{amsrefs}
\numberwithin{figure}{section}
\numberwithin{equation}{section}

\newtheorem{theorem}{Theorem}[section]
\newtheorem{corollary}[theorem]{Corollary}
\newtheorem{lemma}[theorem]{Lemma}
\newtheorem{proposition}[theorem]{Proposition}

\theoremstyle{definition}
\newtheorem{definition}[theorem]{Definition}
\newtheorem{definition*}{Definition}
\theoremstyle{remark}
\newtheorem{remark}[theorem]{Remark}

\newcommand{\defin}[1]{\textit{#1}}

\newcommand{\Ind}{\operatorname{Ind}}
\newcommand{\codim}{\operatorname{codim}}

\newcommand{\R}{\mathbb{R}}
\newcommand{\C}{\mathbb{C}}
\newcommand{\Q}{\mathbb{Q}}
\newcommand{\CP}{\mathbb{CP}}

\newcommand{\Z}{\mathbb{Z}}

\newcommand{\A}{\mathcal{A}}

\newcommand{\pb}{{}^{*}}
\renewcommand{\d}{{\,d}}

\newcommand{\M}{\mathcal{M}}
\newcommand{\U}{\mathcal{U}}
\newcommand{\V}{\mathcal{V}}
\newcommand{\MM}{{\M }}
\newcommand{\e}{\operatorname{e}}

\newcommand{\loc}{\text{loc}}

\newcommand{\Hom}{{\operatorname{Hom}}}

\newcommand{\pt}{\text{pt}}

\newcommand{\CZ}{\operatorname{CZ}}
\newcommand{\SM}{\operatorname{SM}}

\newcommand{\tensor}{\otimes}

\DeclareMathOperator{\Th}{Th}             %
\DeclareMathOperator{\Crit}{Crit}             %
\DeclareMathOperator{\ind}{ind}             %
\DeclareMathOperator{\Aut}{Aut}             %
\DeclareMathOperator{\coker}{coker}             %
\DeclareMathOperator{\Sp}{Sp}                %

\newcommand{\J}{\mathcal J}

\newcommand{\ev}{\operatorname{ev}}

\newcommand{\wc}[1]{\widecheck {#1}}            %
\newcommand{\wh}[1]{\widehat {#1}}              %

\newcommand{\Nabla}{\widetilde{\nabla}}

\begin{document}

\title{Morse--Bott split symplectic homology}
\date{\today}
\author{Lu\'is Diogo}
\author{Samuel T. Lisi}

\begin{abstract}

We introduce a chain complex associated to a Liouville domain $(\overline{W}, d\lambda)$ whose boundary
$Y$ admits a Boothby--Wang contact form (i.e.~is a prequantization space). The
differential counts Floer cylinders with cascades in the completion $W$ of
$\overline{W}$, in the spirit of 
Morse--Bott homology \citelist{\cite{BourgeoisThesis}
\cite{Frauenfelder_Arnold_Givental_conjecture} \cite{BOSymplecticHomology}}.
The homology of this complex is the symplectic homology of $W$
\cite{DiogoLisiComplements}.

Let $X$ be obtained from $\overline{W}$ by collapsing the boundary $Y$ along
Reeb orbits, giving a codimension 2 symplectic submanifold $\Sigma$.
Under monotonicity assumptions on $X$ and $\Sigma$, we show that for
generic data, the differential in our chain complex counts elements of moduli spaces of cascades 
that are transverse. Furthermore, by some index estimates, we show that very few
combinatorial types of cascades can appear in the differential.

\end{abstract}

\maketitle
\tableofcontents

\section{Introduction and statement of main results}

In this paper we define Morse--Bott \textit{split} symplectic homology theory for
Liouville manifolds $W$ of finite-type whose boundary $Y = \partial W$ is a
prequantization space. 
This is inspired by the construction of Bourgeois and Oancea for positive symplectic homology, $SH^+$ 
\cite{BOExactSequence}. 
Our main result is that we obtain transversality for all the
relevant moduli spaces and thus have a well-defined theory. 
This is obtained by
means of generic choice of the geometric data, as opposed to using abstract perturbations. 
This is an important preliminary step in the computation of this chain complex 
in \cite{DiogoLisiComplements}.
Furthermore, in \cite{DiogoLisiComplements}, we justify that the homology of
this complex is indeed the symplectic homology of $W$.

    The main purpose of this definition of a split version of symplectic homology is
    to enable computations in certain examples. For instance, we 
    expect to be able to compute symplectic homology 
    for completions $W$ of $X\setminus \Sigma$, where
   both $X$ and $\Sigma$ are smooth projective complete intersections. Many of these
    examples fit in our framework and additionally enough is known about their
Gromov--Witten invariants for the computation to be possible.

In the following, we consider a $2n$--dimensional Liouville 
domain $(\overline{W}, d\lambda)$ with $\partial \overline{W} = Y$. 
Denoting by $\alpha = \lambda|_Y$ the contact form on the boundary induced by the 
Liouville form $\lambda$, we require that $\alpha$ is
a Boothby--Wang contact form, i.e.~its Reeb vector field induces a free $S^1$
action.  Such a contact manifold is also called a prequantization space.
We let $\Sigma^{2n-2}$ denote the quotient by the Reeb vector field, and we
note that $d\alpha$ descends to a symplectic form $\omega_\Sigma$ on the
quotient. 
It follows that there exists a closed symplectic manifold $(X,
\omega)$ for which $\Sigma$ is a codimension 2 symplectic submanifold,
Poincar\'e dual to a multiple of $\omega$, with
the property that $(X \setminus \Sigma, \omega)$ is symplectomorphic to
$(\overline{W} \setminus Y, d\lambda)$. 
See Proposition \ref{prop:X} for
a more detailed description of this, or see
\cite{Giroux_remarks_Donaldson}*{Proposition 5}. 
We can think of $X$ as the symplectic cut of $\overline W$ along $Y$. 

We let $(W, d\lambda)$ be the completion of $\overline{W}$, obtained by
attaching a cylindrical end $\R^+ \times Y$, and take a Hamiltonian $H \colon
\R\times Y \to \R$ with a growth condition (see Definition \ref{def:J_shaped}
for details). This Hamiltonian will have Morse--Bott families of 1-periodic
orbits, and we then use the formalism of cascades 
(similar in style to the Morse--Bott theories of \citelist{\cite{BourgeoisThesis} \cite{Frauenfelder_Arnold_Givental_conjecture} \cite{Hutchings_Nelson_S1_Morse_Bott}}) 
in order to construct a
Morse--Bott Floer homology associated to $H$, but we also consider
configurations that interact between $\R \times Y$ and $W$, as in \cite{BOExactSequence}.

In Section \ref{S:chain complex} we introduce the chain complex for split
symplectic homology, and in Section \ref{S:split moduli} we describe the moduli spaces that are relevant for the
differential.  The chain complex will be generated by critical points of
auxiliary Morse functions associated to our Morse--Bott manifolds of orbits.
The differential will be obtained from moduli spaces of Floer cylinders with
cascades
together with asymptotic boundary conditions given in terms of the auxiliary Morse functions.

We then prove three main results.  
The first is a transversality theorem for ``simple cascades''. 
These are elements of the relevant moduli spaces that are also somewhere
injective when projected to $\Sigma$. This result holds without any
monotonicity assumptions.
A more precise formulation is provided in Proposition \ref{cascades form manifold}:
\begin{theorem} \label{thm:simple Floer are transverse} 
    Simple Floer cylinders with cascades in $W$ and $\R \times Y$ are transverse for a
    co-meagre set of compatible, Reeb-invariant and cylindrical almost complex structures on $(W,d\lambda)$.
\end{theorem}

Theorem \ref{thm:simple Floer are transverse} 
builds on a transversality theorem for moduli spaces of spheres
in $X$ and in $\Sigma$, which may be of independent interest. 
This construction is somewhat analogous to the strings of pearls that Biran
and Cornea study in the Lagrangian case \cite{BiranCornea}. Our result builds essentially on results from \cite{McDuffSalamon}. 
A more precise formulation is given in Proposition \ref{necklaces are
regular}:
\begin{theorem}
    For a generic compatible almost complex structure, 
    moduli spaces of somewhere injective spheres in $\Sigma$ and of somewhere injective spheres in $X$ with order-of-contact constraints at $\Sigma$, connected by gradient trajectories, are transverse assuming no two spheres have the same image.
\end{theorem}

The third result builds on these two to describe the moduli spaces that are relevant 
under suitable monotonicity assumptions on $X$ and $\Sigma$. 
In particular, the differential is computed from only four types of simple
cascades. See Propositions \ref{P:Floer_in_R_times_Y} and \ref{P:Floer_also_in_W} for more details. 
\begin{theorem}
    Assume that $(X, \omega)$ is spherically monotone with monotonicity
    constant $\tau_X$ and assume that $\Sigma \subset X$ is Poincar\'e dual to
    $K \omega$ with $\tau_X > K > 0$.

    Then, the split Floer homology of $W$ is well-defined, and does not depend
    on the choice of Hamiltonian or of compatible, Reeb--invariant and cylindrical almost complex
    structure on $W$ (in a co-meagre set of such almost complex structures).
    
    The only moduli spaces of split Floer cylinders with cascades that count in the
    differential are the following: 
    \begin{enumerate}[(1)]
        \item[(0)] Morse trajectories in $Y$ or in $W$;
        \item Floer cylinders in $\R \times Y$, projecting to non-trivial
            spheres in $\Sigma$, and with asymptotics constrained by
            descending/ascending manifolds of critical points in $Y$;
        \item holomorphic planes in $W$ that converge to a generic Reeb orbit
            in $Y$;
        \item holomorphic planes in $W$ constrained to have a marked point in
            the descending manifold of a Morse function in $W$ and whose
            asymptotic limit is constrained by the auxiliary Morse function on
            the manifold of orbits in $Y$.
    \end{enumerate}
\end{theorem}
We remark that Cases (2) and (3) are non-trivial
cascades, but their components in $\R \times Y$ lie in fibres of $\R
\times Y \to \Sigma$.
This is formulated more precisely in Propositions \ref{P:Floer_in_R_times_Y}
and \ref{P:Floer_also_in_W}. See Figures \ref{case_1_fig}, \ref{case_2_fig} and 
\ref{case_3_fig} for a depiction of Cases (1)-(3). 

The paper concludes with a discussion of orientations of the moduli spaces of cascades contributing 
to the differential, in Section \ref{sec:orientations}. 

\subsection*{Acknowledgements}
We would like to thank Yasha Eliashberg, Paul Biran and Dusa McDuff for many
helpful conversations, guidance and ideas. We also thank Jean-Yves Welschinger
and Felix Schm\"aschke for helping us to understand coherent
orientations better. We also thank Fr\'ed\'eric Bourgeois, Joel Fish, Richard Siefring
and Chris Wendl for helpful suggestions. Finally, we thank the referee for their
careful reading of our paper and for many helpful suggestions for improvement.

S.L.~was partially supported by the ERC Starting Grant of Fr\'ed\'eric Bourgeois
StG-239781-ContactMath and also by Vincent Colin's ERC Grant geodycon.

L.D.~thanks Stanford University, ETH Z\"urich, Columbia University and Uppsala University 
for excellent working conditions. L.D.~was partially supported by the Knut and Alice Wallenberg Foundation.  

\section{Set--up}

\label{S:setup}

We now provide details of the classes of Liouville manifolds for which we prove
transversality in split Floer homology.

We begin by summarizing some constructions from \cite{DiogoLisiComplements}, specifically 

\begin{proposition}[\cite{DiogoLisiComplements}*{Lemma 2.2}] \label{prop:X}
    Let $(\overline{W}, d\lambda)$ be a Liouville domain
    with boundary $Y = \partial \overline{W}$. Assume that 
    $\alpha \coloneqq \lambda|_Y$ has a Reeb
    vector field generating a free $S^1$ action.

    Then, if $\Sigma$ is the quotient of $Y$ by the $S^1$ action, and
    $\omega_\Sigma$ is the symplectic form induced from $d\alpha$, there exists a
    symplectic manifold $(X, \omega)$ with $\Sigma \subset X$ so that
    $\omega|_{\Sigma} = \omega_\Sigma$, with the following properties:
    \begin{enumerate}[(i)]
        \item $(X\setminus \Sigma, \omega)$ is symplectomorphic to
            $(\overline{W} \setminus \partial{\overline W}, d\lambda)$;
        \item $[\Sigma] \in H_{2n-2}(X; \Q)$ is Poincar\'e dual to 
            $[K \omega] \in H^2(X; \Q)$  for some $K > 0$;
        \item if $N\Sigma$ denotes the symplectic normal bundle to $\Sigma$ in
            $X$, equipped with a Hermitian structure (and hence symplectic
            structure),  there exist
            a neighbourhood $\U$ of the $0$-section in
            $N\Sigma$ and a symplectic embedding $\varphi \colon \U \to X$.
            By shrinking $\U$ as necessary, we may arrange that $\varphi$ extends
            to an embedding of $\overline \U$.
    \end{enumerate}
    \qedhere
\end{proposition}

\begin{definition}
Let $X, \omega, \Sigma, \omega_\Sigma$ and $N\Sigma$ be as in Proposition
\ref{prop:X}. 

Fix a Hermitian line bundle structure 
on the symplectic normal bundle $\pi \colon N\Sigma \to \Sigma$.
A Hermitian connection on $N\Sigma$ can be encoded in terms of a
connection 1-form $\Theta \in \Omega^1(N\Sigma \setminus \Sigma)$ 
with the property that $d\Theta = -K \pi^*d\omega_\Sigma$. 
Fix such a Hermitian connection 1-form $\Theta$. 

Since $N\Sigma$ is a Hermitian line bundle, we have the action of
$U(1)$ on the bundle by rotation in the $\C$--fibres. The infinitesimal
generator of this action is a vector field on which $\Theta$ evaluates to $1$.

Let $\rho \colon N\Sigma
\to [0, +\infty)$ denote the norm in $N\Sigma$ measured with respect to
the Hermitian metric.  Then, for each $x \in N\Sigma \setminus \Sigma$,
let $\xi_x = \left( \ker \Theta_x \cap \ker d\rho \right) \subset T_x
N\Sigma$.  We may then extend the distribution $\xi_x$ smoothly to the $0$-section by
definining $\xi_x = T_x \Sigma$ if $x \in \Sigma$.

Notice that this gives a splitting $T_x N\Sigma = \ker d\pi \oplus \xi_x$ 
and $d\pi|_{\xi} \colon \xi_x \to T_{\pi(x)} \Sigma$ is a symplectic
isomorphism. 

Any almost complex structure $J_\Sigma$ on $\Sigma$ then can be lifted to an
almost complex structure on $N\Sigma$ by taking it to be the linearization 
of the bundle complex structure on $\ker d\pi$ and to be the pull-back 
of $J_\Sigma$
to $\xi_x$
by the isomorphism $d\pi \colon \xi_x \to T_{\pi(x)} \Sigma$.

We refer to any almost complex structure obtained in this way as a
\defin{bundle almost complex structure} on $N\Sigma$. 
\end{definition}

Define the open set $\V = X \setminus \overline{ \varphi (\U)}$. 
We will later perturb our almost complex structures in $\V$.

\begin{proposition}[\cite{DiogoLisiComplements}*{Lemma 2.6}]
\label{prop:Xacs}

Let $\overline{W}$, $Y$, $\lambda$, $\alpha$, $X$, $\Sigma$ and $\varphi
\colon \U \to X$ be as in Proposition \ref{prop:X}.

Then, there exists a diffeomorphism $\psi \colon W \to X \setminus \Sigma$
with the following properties:
\begin{enumerate}[(i)]
    \item if $J_X$ in a compatible almost complex structure on $X$ that restricts to
        $\varphi(\U)$ as the push-forward by $\varphi$ of a bundle almost complex
        structure, then $\psi^*J_X$ is a compatible almost complex structure
        on $W$ that is cylindrical and Reeb--invariant on $W \setminus \overline{W}$;

    \item if $J_W$ is a compatible almost complex structure on $W$, cylindrical and Reeb invariant on $W \setminus \overline{W}$, then the pushforward $\psi_* J_W$ extends to an almost complex structure $J_X$ on $X$ and
    $J_\Sigma \coloneqq J_X|_\Sigma$ is also given by restricting $J_W$ to a
        parallel copy of $\{ c \} \times Y$ for some $c > 0$, and taking the
        quotient by the Reeb $S^1$ action.
    \end{enumerate}
    Furthermore, $\psi$ may be taken to be radial, in the sense that for 
        $(r,y) \in [r_0, +\infty) \times Y$ in the cylindrical end of $W$, 
        $r_0$ sufficiently large, then $\psi(r, y) = (\rho(r), y) 
        \in N\Sigma \setminus \Sigma$. 
        
    \qedhere
\end{proposition}
Note that the diffeomorphism $\psi$ is not symplectic. 

In Section \ref{sec:index_inequalities}, we impose additional
conditions of monotonicity. These will also be relevant for the grading given
in Definition \ref{def:grading}.
\begin{definition} \label{def:monotone triple}
    Given a symplectic manifold $(X^{2n},\omega)$ and a
    codimension-2 symplectic submanifold $\Sigma^{2n-2} \subset X$, say that 
    $(X, \Sigma, \omega)$ is a \defin{monotone triple} if 
\begin{enumerate}[(i)]
    \item $X$ is spherically monotone: there exists a constant $\tau_X > 0$ so
        that for each spherical homology class
        $A \in H_2^S(X)$, $\omega(A) = \tau_X \langle c_1(TX), A \rangle$;
    \item $[\Sigma] \in H_{2n-2}(X; \Q)$ is Poincar\'e
        dual to $[K\omega]$ for some $K > 0$ with $\tau_X > K$. \label{enum_monotone_degree}
\end{enumerate}
In this case, we write $\omega_\Sigma = \omega|_\Sigma$.
    \end{definition}
Observe that Condition \eqref{enum_monotone_degree} is automatically verified
if $[\Sigma]$ is Poincar\'e dual to $[K \omega]$ and 
is non--trivially spherically monotone. In that case, 
$\tau_X - K$ is then
the monotonicity constant of $(\Sigma, \omega_\Sigma)$, which we denote by
$\tau_\Sigma$.

We then obtain the following useful characterization of $J_W$--holomorphic
planes in $W$ (see also 
    \cite{HoferKriener}):
\begin{lemma}[\cite{DiogoLisiComplements}*{Lemma 2.6}]\label{planes = spheres}
    Under the diffeomorphism of Proposition \ref{prop:Xacs}, 
    finite energy $J_W$--holomorphic planes in $W$ correspond to
    $J_X$--holomorphic spheres in $X$ with a single intersection with
    $\Sigma$. The order of contact gives the multiplicity of the Reeb
    orbit to which the plane converges.
\end{lemma}
\begin{proof}
    Under the map $\psi$, a finite energy $J_W$--plane in $W$ gives a 
    $J_X$--plane in $X \setminus \Sigma$. 
    By the fact that $\psi$ is radial, the restriction of $\psi^*\lambda$ to
    the cylindrical end $[r_0, +\infty) \times Y$ takes the form
    $g(r)\alpha$ for some function $g$ with $g' > 0$ and $\lim_{r\to \infty}
    g(r) = 1$.  It follows then that the integral of $\psi^*d\lambda$ over the
    plane is dominated by its Hofer energy as given in
    \cite{SFTcompactness}*{Section 5.3}.  Hence, the image of the plane under
    the map is a plane in $X \setminus \Sigma$ with finite $\omega$--area.

    It follows then that the singularity at $\infty$ is removable by Gromov's removal of
    singularities
    theorem, and thus the plane admits an extension to a $J_X$--holomorphic sphere. 
    The order of contact follows from considering the winding around
    $\Sigma$ of a loop near the puncture.
\end{proof}

We now formalize the class of almost complex structures that 
will be relevant for this paper.

\begin{definition} \label{def:admissible J_W}
    An \defin{admissible} almost complex structure $J_X$ on $X$ is compatible with
$\omega$ and its restriction to $\varphi(\U)$ is the push-forward by $\varphi$
of a bundle almost complex structure on $N\Sigma$. 

An almost complex structure $J_W$ on $(W, d\lambda)$ is \defin{admissible}
if $J_W = \psi^*J_X$ for an admissible $J_X$. 
In particular, such almost complex structures are
cylindrical and Reeb-invariant on $W \setminus \overline{W}$.

A compatible almost complex structure $J_Y$ on the symplectization
$\R \times Y$ is \defin{admissible} if $J_Y$ is cylindrical and
Reeb--invariant.

\end{definition}

In the following, we will identify $W$ with $X \setminus \Sigma$ by means of
the diffeomorphism $\psi$ and identify the corresponding almost complex
structures. By an abuse of notation, we will both write $\pi_\Sigma \colon Y \to \Sigma$ to
denote the quotient map that collapses the Reeb fibres, and 
$\pi_\Sigma \colon \R \times Y \to \Sigma$ to denote the composition of this
projection with the projection to $Y$.

\begin{definition} \label{def:J_W}
  Denote the space of almost complex structures in $\Sigma$ that are
  compatible with $\omega_\Sigma$ by $\J_\Sigma$.
  
  Let $\J_Y$ denote the space of admissible almost complex structures on $\R
  \times Y$. Then, the projection $d\pi_\Sigma$ induces a diffeomorphism between $\J_Y$
  and $\J_\Sigma$.

  Let $\J_W$ denote the space of admissible almost complex structures on $W$.
  By Proposition \ref{prop:Xacs}, for any $J_W \in \J_W$, we obtain an almost
  complex structure $J_\Sigma$. 

  Denote this map by $P \colon \J_W \to \J_\Sigma$. This map is surjective and
  open by Proposition \ref{prop:Xacs}. 
  For any given $J_\Sigma \in \J_\Sigma$,
  $P^{-1}(J_\Sigma)$ consist of almost complex structures on $W$ that differ in
  $\overline{W}$, or equivalently, can be identified (using $\psi$)
  with almost complex structures on $X$ that differ in $\V = X \setminus
  \varphi(\U)$.
\end{definition}

\section{The chain complex} \label{S:chain complex}

We will now describe the chain complex for the split symplectic homology
associated to $W$.

\begin{definition} \label{def:J_shaped}
Let $h \colon (0, +\infty) \to \R$ be a smooth function with the following properties:
\begin{enumerate}[(i)]
\item $h(\rho) = 0$ for $\rho \le 2$;
\item $h'(\rho) > 0$ for $\rho > 2$;
\item $h'(\rho) \to +\infty$ as $\rho \to \infty$;
\item $h''(\rho) > 0$ for $\rho > 2$; 
\end{enumerate}

An \defin{admissible Hamiltonian} $H \colon \R \times Y \to \R$ is given by 
$H(r, y) = h(\e^r)$.
Since this Hamiltonian $H(r, y) = 0$ for all $r \le \ln 2$, we will also take $H
= 0$ on $W$ when considering cascades with components in $W$. 
See Figure \ref{F:Jshaped}.

\end{definition}

\begin{figure}[h]
\begin{tikzpicture}[domain=0:6]
    \draw[->] (-0.2,0) -- (3.1,0) node[right] {$\e^r = \rho$};
    \draw[->] (0,-1.7) -- (0,2.7) node[above] {$h(\e^r)$};%
    \draw[line width=1.5pt, color=black] (0,0) -- (1,0);
    \draw[line width=1.5pt,domain=1:2.7,smooth,variable=\x,black] plot ({\x},{(\x-1)*(\x-1)});
    \draw[line width=0.5pt, color=black] (0,-1.25) -- (3,1.75);
    \draw [dashed, color=black] (1.5,-0.1) -- (1.5,0.25);
    \node at (1.65,-.3) {$b_k$};
    \node at (-.65,-1.25) {$-\A(\gamma_k)$};
\end{tikzpicture}
\caption{An admissible Hamiltonian and the graphical procedure for computing the action of a periodic orbit}
\label{F:Jshaped}
\end{figure}
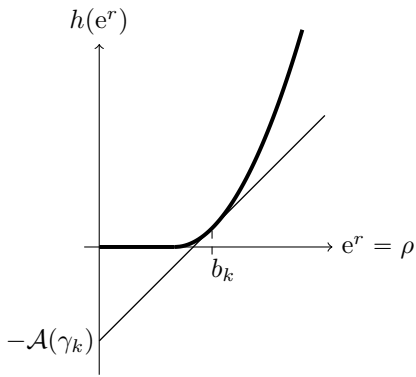

Since $\omega = d(\e^r \alpha)$ on $\R_+\times Y$,  the Hamiltonian vector
field associated to $H$ is $h'(\e^r)R$, where $R$ is the Reeb vector field
associated to $\alpha$. 
The fibres of $Y \to \Sigma$ are periodic Reeb orbits. Their minimal periods
are $T_0 \coloneqq \int_{\pi_\Sigma^{-1}(p)} \alpha$, for $p\in \Sigma$.
The 1-periodic orbits of $H$ are thus of two types:
\begin{enumerate}
    \item constant orbits: one for each point in $W$ and at each point in
        $(-\infty, \log 2] \times Y \subset \R \times Y$;
 \item non-constant orbits: for each $k\in \Z_+$, there is a $Y$-family of 1-periodic $X_H$-orbits, contained in the level set $Y_k
 \coloneqq \{b_k\}\times Y$, for the unique $b_k>\log 2$ such that $h'(\e^{b_k}) = kT_0$. Each point in $Y_k$ is the starting point of one such orbit.
\end{enumerate}

\begin{remark}
Notice that these Hamiltonians are Morse--Bott 
non-degenerate \textit{except} at $\{\log 2\} \times Y$. 
These orbits will not play a role because they can never arise as boundaries
of relevant moduli spaces (see \cite{DiogoLisiComplements}*{Lemma 4.8}). 

Recall that a family of periodic Hamiltonian orbits for a time-dependent
Hamiltonian vector field is said to be Morse--Bott non-degenerate
if the connected components of the space of parametrized 1-periodic orbits form manifolds, and the
tangent space of the family of orbits at a point is given by the
eigenspace of $1$ for the corresponding Poincar\'e return map.
(Morse non-degeneracy requires the return map not to have 1 as an
eigenvalue and hence such periodic orbits must be isolated.)

\end{remark}

We also fix some auxiliary data, consisting of Morse functions and
vector fields.  Fix throughout a Morse function $f_\Sigma \colon
\Sigma \to \R$ and a gradient-like vector field $Z_\Sigma\in\mathfrak X(\Sigma)$, which means that 
$\frac{1}{c} |df_\Sigma|^2 \le df_\Sigma(Z_\Sigma)\le c |df_\Sigma|^2$ for some constant $c>0$.
Denote the time-$t$ flow
of $Z_\Sigma$ by $\varphi^t_{Z_\Sigma}$. 
Given $p\in \Crit (f_\Sigma)$, its stable and unstable manifolds (or ascending and descending manifolds, respectively) are
\begin{equation}\label{(un)stable}
W^s_{\Sigma}(p) := \left\{ q\in \Sigma | \lim_{t\to \infty} \varphi_{Z_\Sigma}^{-t}(q) = p \right\}
 \text{,} \,
W^u_{\Sigma}(p) := \left\{ q\in \Sigma | \lim_{t\to -\infty} \varphi_{Z_\Sigma}^{-t}(q) = p \right\}.
\end{equation}
Notice the sign of time in the flow, so that these are the stable/unstable
manifolds for the flow of the negative gradient. 
We further require that $(f_\Sigma,Z_\Sigma)$ be a Morse--Smale pair, i.e.~that all stable and unstable
manifolds of $Z_\Sigma$ intersect transversally.

The contact distribution $\xi$ defines an Ehresmann connection on
the circle bundle $S^1 \hookrightarrow Y\to \Sigma$. Denote the
horizontal lift of $Z_\Sigma$ by $\pi_\Sigma^*Z_\Sigma\in \mathfrak
X(Y)$.  We fix a Morse function $f_Y \colon Y \to \R$ and a
gradient-like vector field $Z_Y\in \mathfrak X(Y)$ such that
$(f_Y,Z_Y)$ is a Morse--Smale pair and the vector field
$Z_Y-\pi_\Sigma^*Z_\Sigma$ is vertical (tangent to the $S^1$-fibres).
Under these assumptions, flow lines of $Z_Y$ project under $\pi_\Sigma$
to flow lines of $Z_\Sigma$.

Observe that critical points of $f_Y$ must lie in the fibres above
the critical points of $f_\Sigma$ (and these are zeros of $Z_Y$ and
$Z_\Sigma$ respectively).  For notational simplicity, we suppose
that $f_Y$ has two critical points in each fibre.  In the following,
given a critical point for $f_\Sigma$, $p \in \Sigma$, we denote
the two critical points in the fibre above $p$ by $\widehat p$ and
$\widecheck p$, the fibrewise maximum and fibrewise minimum of
$f_Y$, respectively.

We will denote by $M(p)$ the Morse index of a critical point $p \in
\Sigma$ of $f_\Sigma$, and by $\tilde M (\tilde p) = M(p) + i(\tilde p)$
the Morse index of the critical point $\tilde p = \widehat p$ or $\tilde
p = \widecheck p$ of $f_Y$. The fibrewise index has $i(\widehat p) = 1$
and $i(\widecheck p) = 0$.

Fix also a Morse function $f_W$ and a gradient-like vector field $Z_W$ on $W$, such
that $(f_W,Z_W)$ is a Morse--Smale pair and $Z_W$ restricted to $[0,\infty)\times Y$ is the constant
vector field $\partial_r$, where $r$ is the coordinate function on the
first factor. 
We denote also by $(f_W,Z_W)$ the Morse--Smale pair that is induced on $X\setminus\Sigma$ by the diffeomorphism 
in Lemma \ref{planes = spheres}.
Denote by $M(x)$ the Morse index of $x\in \Crit(f_W)$ with respect to $f_W$.

We now define the Morse--Bott symplectic chain complex of $W$ and
$H$. Recall that for every $k>0$, each point in $Y_k \coloneqq \{ b_k \} \times Y \subset \R^+ \times Y$ is the starting point of a 
1-periodic orbit of $X_H$, which covers $k$ times its underlying Reeb orbit. 

For each critical
point $\tilde p = \widehat p$ or $\tilde p = \widecheck p$ of $f_Y$,
there is a generator corresponding to the pair $(k, \tilde p)$. We will
denote this generator by $\tilde p_k$.  The complex is then given by:
\begin{equation} \label{eqn:chain complex}
SC_*(W,H) = \left(\bigoplus_{k>0} \,
        \bigoplus_{p\in \text{Crit}(f_\Sigma)}\Z
            \langle \widecheck p_k, \widehat p_k \rangle\right) 
        \oplus \left(\bigoplus_{x\in \text{Crit}(f_W)}\Z\langle x \rangle\right) 
\end{equation}

\begin{definition} \label{def:symplectic action}
The \defin{Hamiltonian action} of a loop $\gamma \colon S^1 \to \R \times Y$ is 
$$\A( \gamma ) = \int \gamma^*( \lambda - H dt).$$ 
\end{definition}

In particular, $\A(\gamma ) = 0$ for any constant orbit $\gamma$, and for any orbit $\gamma_k \in Y_k$ we have 
\begin{equation} \label{E:HamiltonianAction}
\A( \gamma_k) = \e^{b_k} h'(\e^{b_k}) - h( \e^{b_k})>0,
\end{equation}
where $b_k$ is the unique solution to the equation 
$h'(\e^{b_k}) = kT_0$, as above.  
The action of $\gamma_k$ is the negative of the
$y$-intercept of the tangent line to the graph of $h$ at $\e^{b_k}$. 
See
Figure \ref{F:Jshaped}. The convexity of $h$ implies that
$\A(\gamma_k)$ is monotone increasing in $k$.

\subsection{Gradings} \label{sec:gradings}
We will now define the gradings of the generators. 
For this, we will assume that $(X, \Sigma, \omega)$ is a monotone triple as in
Definition \ref{def:monotone triple}.

\begin{definition} \label{def:grading}
For a critical point $\widetilde p$ of $f_Y$, and a multiplicity $k$, we define
\begin{equation} \label{eqn:grading Y}
| \widetilde p_k | = \tilde M( \widetilde p) + 1-n + 2\frac{\tau_X - K}{K} k \in \R, 
\end{equation}
where we recall that $\tau_X$ is the monotonicity constant of $X$ and $c_1(N\Sigma) = [K \omega_\Sigma]$.

For a critical point $x$ of $f_W$, we define
\begin{equation} \label{eqn:grading W}
|x| = n - M(x).
\end{equation}

Finally, for convenience, we introduce an index similar to the SFT grading
for the Reeb orbits to which a split Floer cylinder converges at augmentation
punctures. 
If $\gamma$ is such a Reeb orbit, it is a $k$-fold cover of a fibre of $Y \to
\Sigma$ for some $k$. We then define its index to be:
\begin{equation} \label{E:grading_Reeb_orbits}
    | \gamma |_0 
    = -2 + 2\left ( \frac{ \tau_X - K}{K} \right ) k .
\end{equation}
\end{definition}

The justifications of these gradings comes from analyzing the Conley--Zehnder
indices of the 1-periodic Hamiltonian orbits. 
These are defined for Morse non-degenerate Hamiltonian/Reeb
orbits, using a trivialization of $TW$ or of $T(\R \times Y)$ over the orbit. 
See Definition \ref{D:CZ} for the Morse--Bott analogue, and also
\citelist{
\cite{AbreuMacarini_dynamically_convex_elliptic}*{Section 3}
\cite{Gutt_generalized_CZ}
}. 
The first key observation is that the Conley--Zehnder index of an orbit only
depends on the trivialization of the complex line bundle $L \coloneqq \Lambda^{n}_\C TW$ 
over the orbit.

For a constant orbit, we may take a constant trivialization, and applying
Definition \ref{D:CZ}, we obtain the Conley--Zehnder index of the constant 
orbit to be $-n+ (2n-M(x)) = n - M(x)$.

A non-constant orbit $\gamma$ in $\R \times Y$ projects to a point in $\Sigma$.
From this, we may take a ``constant'' trivialization of $\gamma^*\xi$ by taking
the horizontal lift of a constant trivialization of $T_{\pi(\gamma)}\Sigma$. 
Then, by considering the linearized Hamiltonian flow in the vertical direction,
we obtain the Conley--Zehnder index of the corresponding generator
$\widetilde{p}_k$ to be 
\begin{equation} \label{eqn:constant trivialization CZ index}
    \CZ_0(\widetilde p_k) = \tilde M( \widetilde p) + 1 - n.
\end{equation}
The Conley--Zehnder index is computed by using the splitting of $T(\R \times
Y) = (\R \oplus \R R) \oplus \xi$. The contribution to the index is given by
$i(\widetilde p)$ in the vertical $\R \oplus \R R$ direction, and by $M(p) + 1-n$
in the horizontal $\xi$ direction. Notice that this index does not explicitly
depend on the covering multiplicity $k$ of the orbit.

Notice also that $Y$ may be capped off by the normal disk bundle over $\Sigma$,
and each orbit bounds a disk fibre. The trivialization induced by the fibre
differs from the constant trivialization only in a $k$-fold winding in the $\R
\partial_r \oplus \R R$ direction. The resulting  Conley--Zehnder index of
$\widetilde{p}_k$ for this trivialization induced by the disk fibre 
is then $\tilde M( \widetilde p) +1 - n - 2k$. We refer to this trivialization
as the \defin{normal bundle trivialization}.

Now, 
suppose that $\gamma_k$ is the $k$--fold cover a simple Reeb orbit $\gamma$,
and suppose it is contractible in $W$. 
Denote by $\dot B$ a disk in $W$ whose boundary is $\gamma_k$. 
As we pointed out, $\gamma_k$ is also the boundary
of a $k$-fold cover of a fibre of the normal bundle to $\Sigma$ in $X$.
This cover of a fibre can be
concatenated with $\dot B$ to produce a spherical homology class $B\in
H_2^S(X)$ such that the intersection $B\bullet \Sigma = k>0$.  Conversely, note that any
$B\in H_2^S(X)$ such that $B\bullet \Sigma = k$ gives rise to a disk
$\dot B$ bounding $\gamma_k$.  The complex line bundle $L|_{\dot B}$
is trivial, since $\dot B$ is a disk.  This induces a trivialization of
$L$ over $\gamma_k$, which can be identified with a trivialization of
$L^{\otimes k}$ over $\gamma$. We refer to this as the \defin{trivialization induced
by $\dot B$}.

The relative winding of the trivialization of $L$ over $\gamma_k$ induced by $\dot B$ and the normal
bundle trivialization considered above is given by 
$\langle c_1(L), B \rangle$,
since this represents the obstruction to extending the trivialization
of $L$ over $\dot B$ to all of $B$.  Recall that $c_1(L) = c_1(TX)$.

Putting this together, we obtain that the Conley--Zehnder index of the orbit 
with respect to the trivialization induced by the disk $\dot B$ is given by 
\begin{equation} \label{eqn:CZcap}
\CZ_H^W( \widetilde p_k) = \tilde M( \widetilde p) +1 - n - 2k + 2 \langle
c_1(TX), B \rangle.
\end{equation}
Finally, we obtain the grading from Equation \eqref{eqn:grading Y} by using the
spherical monotonicity of $X$ and the fact that 
$k = B \bullet [\Sigma] = K \omega(B)$.
\begin{align*}
    \CZ_H^W(\widetilde p_k) &= \tilde M( \widetilde p) +1 - n + 2 (\tau_X - K)\,\omega(B)\\
                            &= \tilde M( \widetilde p) +1 - n + 2 \left (
\frac{\tau_X - K}{K} \right ) k.
\end{align*}
Note that this expression does not depend on the choice of spherical class $B$.

This formula holds when $k \in \omega( \pi_2(X) )$, and we extend it as a
fractional grading for all $k\in \Z$. 
(This corresponds to 
the fractional SFT grading from \cite{EGH}*{Section 2.9.1}.)

\

Finally, we compare our gradings with those described by
Seidel \cite{SeidelBiasedView} and generalized by McLean
\cite{McLean_discrepancy} (the latter considers Reeb orbits, but there is an
analogous construction for Hamiltonian orbits) in the case that 
$c_1(TW) \in H_2(W;\Z)$ is torsion, so $Nc_1(TW) = 0$ for a suitable choice of $N>0$.
Note that in our setting, this holds if $X$ is monotone (and not just
spherically monotone) and $\Sigma$ is Poincar\'e dual to a multiple of $\omega$.

First, we describe the Seidel--McLean approach. 
Recall that $L = \Lambda^n TW$.
We choose a global trivialization of $L^{\otimes N}$ over $W$. This exists
since $c_1(L)$ is $N$-torsion in $H^2(W; \Z)$. 
Then, for any orbit $\gamma$, we consider the (complex) rank $nN$ bundle
over $\gamma$ given by $\gamma^*TW \oplus \dots \oplus \gamma^*TW$. Denote
this by $\gamma^*( TW^N)$. Notice that the determinant of this bundle is
precisely $\gamma^*(L^{\otimes N})$ and thus has a trivialization already
chosen. 
We now choose any trivialization of $TW^N$ whose determinant matches 
the trivialization of $\gamma^*(L^{\otimes N})$. 
For any such trivialization, 
the linearized flow $d\phi_t \oplus d\phi_t \oplus \dots \oplus
d\phi_t \colon TW^N \to TW^N$ gives a path of symplectic matrices.
This gives a a Conley--Zehnder index, which we denote by
$\CZ_{L^{\otimes N}}(\gamma)$.
Notice that this Conley--Zehnder index 
depends only on the trivialization of
$L^{\otimes N}$ and not on the further trivialization of $TW^N$.  
The Seidel--McLean grading is
then defined to be
\[
    \SM(\gamma) = \frac{1}{N} \CZ_{L^{\otimes N}}( \gamma).
\]

We now observe some immediate consequences of this construction.
First of all, given a null-homologous orbit in $W$, a capping
surface induces a trivialization of $L$ over the
orbit, unique up to homotopy, namely a trivialization that extends across the
surface. 
This implies that there is a homotopically unique trivialization of 
$L^{\otimes N}$ over the orbit, hence the Seidel--McLean
grading matches the Conley--Zehnder index induced by the trivialization
coming from the capping surface, for null-homologous orbits.
(Notice that because $c_1(L)$ is torsion, it
doesn't matter which surface we use.) 
Thus, if $\gamma$ is an orbit such that $\gamma_k$ 
bounds a disk $\dot B$, 
(as discussed above, see notably Equation \eqref{eqn:CZcap}), 
\[
    \CZ_H^W(\widetilde p_k) = \SM(\widetilde p_k).
\]

We now consider an orbit $\gamma_m$ that is an $m$-fold cover of a simple
orbit $\gamma$. We may trivialize $TW$ over $\gamma$ by the constant trivialization
discussed previously. 
This induces a trivialization of $L$ over $\gamma$. By taking the $N$-fold
tensor power, we obtain a trivialization of $L^{\otimes N}$ over $\gamma$.
This has some winding $w \in \Z$ relative to the reference 
trivialization of $L^{\otimes N}$ over all of $W$ (used above to define the Seidel--McLean grading). 
Under iteration of the orbit, we obtain a relative winding of $mw$ between the
constant trivialization and the reference trivialization.
From this, we obtain 
\[
    N \CZ_0( \widehat p_m) = \CZ_{L^{\otimes N}}(\gamma_m) + 2mw
\]
Now, it follows from this that
\begin{align*}
    \CZ_H^W(\widetilde p_m) &= \CZ_0(\widetilde p_m) +2 \left ( \frac{\tau_X -K}{K} \right ) m   
    \\ 
    &= \SM(\widetilde p_m) + 2 \left ( \frac{\tau_X -K}{K} + \frac{w}{N} \right ) m .
\end{align*}
As previously observed, $\CZ_H^W$ and $\SM$ are equal whenever $m$ is in the image of
$c_1(TX) \colon \pi_2(X) \to \Z$, so it follows that
\[
    \CZ_H^W(\widetilde p_m) = \SM(\widetilde p_m)
\]
for all $m$, as claimed.

The index
\eqref{E:grading_Reeb_orbits} of the Reeb orbit 
was introduced for convenience in writing formulas for expected dimensions of
moduli spaces. It comes from similar considerations for the
Conley--Zehnder index of the Reeb orbit, together with the $n-3$ shift coming
from the grading of SFT. In particular, the Fredholm index for an unparametrized
holomorphic plane in $W$ asymptotic to the closed Reeb orbit $\gamma_0$ 
(free to move in its Morse--Bott family) will be given by $|\gamma_0|$.

\begin{remark}
Even though the idea of a fractional grading may seem unnatural at
first, it can be thought of as a way of keeping track of some
information about the homotopy classes of the Hamiltonian orbits.

Indeed, there can only be a Floer cylinder connecting two Hamiltonian
orbits if the difference of their degrees is an integer. Hence, one
could write the symplectic homology as a direct sum indexed by the
fractional parts of the degrees.  Alternatively, one could also
decompose it as a direct sum over homotopy classes of Hamiltonian
orbits, as done for instance in \cite{BOSymplecticHomology}.
\end{remark}

\section{Split symplectic homology moduli spaces} \label{S:split moduli}

In this section, we describe the moduli spaces of cascades that contribute to the differential in the Morse--Bott split symplectic homology of $W$.

We also define auxiliary moduli spaces of spherical ``chains of pearls'' in $\Sigma$ and in $X$. (These are familiar objects, reminiscent of ones considered in the literature for Floer homology of compact symplectic manifolds 
\citelist{%
\cite{ Biran Cornea} 
\cite{ OhChainsOfPearls} 
\cite{PSS}%
}.)

\subsection{Split Floer cylinders with cascades}
\label{cascades after splitting}

We now identify the moduli spaces of \defin{split Floer cylinders with cascades} 
we use to define the differential on the chain complex \eqref{eqn:chain complex}. 

First, we define the basic building blocks: split Floer cylinders. We consider
two types of basic split Floer cylinders: one connecting two non-constant
1-periodic Hamiltonian orbits and one that connects a non-constant 1-periodic
orbit to a constant one (in $W$).

Notice that we may identify a 1-periodic orbit of $H$ with its starting point,
and in this way, we have an identification between $Y_{k}$ and the set of
(parametrized) 1-periodic orbits of $H$ that have covering multiplicity $k$ over
the underlying simple periodic orbit.

\begin{definition}
\label{stretched Y to Y}
Let $x_\pm \in Y_{k_\pm}$ be $1$-periodic orbits of $X_H$ in $\R \times Y$.
A {\em split Floer cylinder} from $x_-$ to $x_+$ consists of a map 
$\tilde v=(b,v) \colon\R \times S^1\setminus \Gamma \to \R \times Y$, 
where $\Gamma = \{z_1,\ldots,z_k\} \subset \R \times S^1$ is a (possibly empty) finite subset, 
together with equivalence classes $[U_i]$ of $J_W$-holomorphic planes $U_i: \C \to W$ for 
each $z_i\in \Gamma$, such that 
\begin{itemize}
\item $\tilde v$ satisfies Floer's equation
\begin{equation}
\partial_s \tilde v + J_Y(\partial_t \tilde v - X_H(\tilde v)) = 0;
\label{Floer eq Y}
\end{equation}
\item $\lim_{s \to \pm \infty} \tilde v(s,.) = x_{\pm}$;
\item if $\Gamma \neq \emptyset$, then, for conformal parametrizations 
$\varphi_i: (-\infty,0] \times S^1 \to \R \times S^1\setminus\{z_1, \ldots, z_k\}$ of neighborhoods 
of the $z_i$, $\lim_{s\to -\infty} \tilde v(\varphi_i(s,.)) = (-\infty,\gamma_i(.))$, where the 
$\gamma_i$ are periodic Reeb orbits in $Y$;
\item for each Reeb orbit $\gamma_i$ above, $U_i: \C \to W$ is asymptotic to $(+\infty,\gamma_i)$. 
We consider $U_i$ up to the action of $\Aut(\C)$. 
\end{itemize} 
Call $\tilde v$ the \defin{upper level} of the split Floer cylinder. 
\end{definition}

See Figures \ref{case_1_fig} and \ref{case_2_fig} for an illustration (ignore the horizontal segments in 
the figures, which represent gradient flow lines). 

\begin{definition}
\label{stretched Y to W}
Let $x_+\in Y_{k_+}$ for some $k_+$, 
and let $x_-\in W$. A {\em split Floer cylinder} from $x_-$ to $x_+$ consists
of $\tilde v_1=(b,v) \colon \R \times S^1\setminus \Gamma \to \R \times Y$ (where
$\Gamma = \{z_1,\ldots,z_k\} \subset \R \times S^1$ is a (possibly empty) finite subset), 
$\tilde v_0 \colon \R\times S^1 \to W$, and equivalence classes $[U_i]$ of $J_W$-holomorphic 
planes $U_i: \C \to W$ for each $z_i\in \Gamma$, such that 
\begin{itemize}
\item $\tilde v_1$ solves equation (\ref{Floer eq Y});
\item $\tilde v_0$ is $J_W$--holomorphic;
\item $\lim_{s \to +\infty} \tilde v_1(s,.) = x_+$;
\item $\lim_{s \to -\infty} \tilde v_1(s,.) = (-\infty,\gamma(\cdot))$, for some Reeb orbit $\gamma$ in $Y$;
\item $\lim_{s \to +\infty} \tilde v_0(s,.) = (+\infty,\gamma(\cdot))$, where $\gamma$
    is the same Reeb orbit; 
\item $\lim_{s \to -\infty} \tilde v_0(s,.) = x_-$;
\item if $\Gamma \neq \emptyset$, then, for conformal parametrizations 
$\varphi_i: (-\infty,0] \times S^1 \to \R \times S^1\setminus\{z_1, \ldots, z_k\}$ of neighborhoods 
of the $z_i$, $\lim_{s\to -\infty} \tilde v(\varphi_i(s,.)) = (-\infty, \gamma_i(.))$, where the $\gamma_i$ are periodic Reeb orbits in $Y$;
\item for each Reeb orbit $\gamma_i$ above, $U_i: \C \to W$ is asymptotic to $(+\infty,\gamma_i)$. We consider $U_i$ up to the action of $\Aut(\C)$. 
\end{itemize}
Call $\tilde v_1$ the {\em upper level} of this split Floer cylinder.  
\end{definition} 

See Figure \ref{case_3_fig} for an illustration. 

For the upper levels of split Floer cylinders
in $\R \times Y$, we introduce a suitable form of energy, a hybrid between
the standard Floer energy and the Hofer energy used in symplectizations.
Recall that the Hofer energy of a punctured pseudoholomorphic curve $\tilde u$ in the symplectization of $Y$ with contact form $\alpha$ is given by:
\[
\sup \{  \int \tilde u \pb d(\psi \alpha) \, | \, \psi\colon \R \rightarrow [0,1] \text{ smooth and nondecreasing} \}.
\]
In a symplectic manifold either compact or convex at infinity, the standard Floer energy of a cylinder $\tilde v \colon \R \times S^1 \to W$ is given by 
\[
\int \tilde v \pb \omega - \tilde v\pb dH \wedge dt.
\]
In our situation, however, the target manifold is $\R \times Y$, which
has a concave end. We therefore need to combine these two types of energy.

\begin{definition} \label{D:hybridEnergy}
Consider a Hamiltonian $H \colon \R \times Y \rightarrow \R$ so that $dH$ has 
support in $[R, \infty) \times Y$, for some $R \in \R$.

Let $\vartheta_R$ be the set of all non-decreasing smooth functions 
$\psi \colon \R \to [0, \infty)$ such that $\psi(r) = \e^r$ for $r \ge R$.

The \textit{hybrid energy} $E_R$ of  
$\tilde v \colon \R \times S^1 \setminus \Gamma  \rightarrow \R \times Y$
solving Floer's equation \eqref{Floer eq Y} is then given by:
\begin{equation} \label{E:hybridEnergy}
E_R( \tilde v ) = 
\sup_{\psi \in \vartheta} 
\int_{\R \times S^1} \tilde v \pb \left ( d( \psi \alpha ) - dH \wedge dt \right).
\end{equation}
\end{definition}

Notice that this is equivalent to partitioning $\R \times S^1 \setminus \Gamma = S_0 \cup
S_1$, so that $S_0 = \tilde v^{-1} ( [R, +\infty) \times Y )$ and
$S_1 = S \setminus S_0$.  Then, $\tilde v$ has finite 
hybrid energy if and only if $\tilde v|_{S_0}$ has finite Floer energy
and $\tilde v|_{S_1}$ has finite Hofer energy.  

Equivalently, $\tilde v$ has finite hybrid
energy if and only if the punctures $\{\pm\infty\}\cup \Gamma$ can be partitioned into 
$\Gamma_F$ and $\Gamma_C$ (with $+\infty\in \Gamma_F$, $\Gamma \subset \Gamma_C$ and $-\infty$ either in $\Gamma_F$ or in $\Gamma_C$), such that in a neighbourhood of each puncture in
$\Gamma_F$, the map $\tilde v$ is asymptotic to a Hamiltonian trajectory
and in a neighbourhood of each puncture in $\Gamma_C$, the map is proper 
and negatively asymptotic to an orbit cylinder in $\R \times Y$.  
This follows from a variation on the arguments in 
\cite{SFTcompactness}*{Proposition 5.13, Lemma 5.15}.
We will use the following notation to denote the
Hamiltonian orbits to which such a punctured cylinder $\tilde v$ is asymptotic:
\begin{align*}
    \tilde v(+\infty, t) &= \lim_{s \to \infty} \tilde v(s, t) \\
    \tilde v(-\infty, t) &= \lim_{s \to -\infty} \tilde v(s, t) \quad \text{ if
    } \quad -\infty \in \Gamma_F.
\end{align*}
Instead, if $-\infty \in \Gamma_C$, we will write $v(-\infty, t) = \lim_{s \to
-\infty} v(s, t)$ for the Reeb orbit in $Y$ that the curve converges to. Notice
that since the cylinder is parametrized, the asymptotic limit is parametrized as
well.  
Since there is an ambiguity of the $S^1$ parametrization of the Reeb orbits to
which $v$ is asymptotic at punctures $z \in \Gamma$, we will avoid using the
analogous notation at punctures in $\Gamma$.

We now define a {\em split Floer cylinder with cascades} between
two generators of the chain complex \eqref{eqn:chain complex}. 
\begin{definition} \label{def:FloerCylinderWithCascades}
Fix $N \ge 0$. 
Let $S_0, S_1, \dots, S_{N}$ be a collection of connected spaces of orbits, 
with $S_0 = Y_{k_0}$ or $S_0 = W$, and $S_i = Y_{k_i}$ for $1 \le i \le N$. 
Let $(f_i, Z_i)$, $i=0, \dots, N$ be the pair of Morse function and 
gradient-like vector field of $f_i = f_Y$, $Z_i = Z_Y$ if $S_i = Y_{k_i}$ 
and $f_i = -f_W$, $Z_i=-Z_W$ if $S_i = W$.

Let $x$ be a critical point of $f_0$ and $y$ a critical point of $f_N$ 
(so $x$ and $y$ are generators of the chain complex (\ref{eqn:chain complex})). 

A {\em Floer cylinder with $0$ cascades} ($N = 0$), from $x$ to $y$, consists of a 
positive gradient trajectory $\nu \colon \R \to S_0$, such that $\nu(-\infty) = x$, $\nu(+\infty) = y$ and $\dot \nu =  Z_0(\nu)$.

A {\em Floer cylinder with $N$ cascades}, $N \ge 1$, from $x$ to $y$, consists of the following data:
\begin{enumerate}
\renewcommand{\theenumi}{\roman{enumi}}
\item $N-1$ length parameters $l_i > 0, i=1, \dots, N-1$;
\item two half-infinite gradient trajectories, $\nu_0 \colon (-\infty, 0] \to
    S_0$ and $\nu_{N} \colon [0, +\infty) \to S_{N}$ with $\nu_0(-\infty) = x$, 
    $\nu_N(+\infty) = y$ and $\dot \nu_i = Z_i(\nu_i)$ for $i=0$ or $N$;
\item $N-1$ gradient trajectories $\nu_i$ defined on intervals of length $l_i$,
    $\nu_i \colon [0, l_i] \to S_i$ for $i=1, \dots, N-1$ such that $\dot \nu_i
    = Z_i(\nu_i)$;
\item $N$ non-trivial split Floer cylinders from $\nu_{i-1}(l_{i-1}) \in S_{i-1}$ to
    $\nu_{i}(0) \in S_i$, where we take $l_0 = 0$.
\end{enumerate}	
In the case of a Floer cylinder with $N \ge 1$ cascades, we refer to the 
non-trivial Floer cylinders $\tilde v_i$ as {\em sublevels}. 
Notice that if $S_0$ (and thus all $S_i$) is of the form $Y_k$, all the split
Floer cylinders are as in Definition \ref{stretched Y to Y}. If $S_0 = W$, then
the bottom-most level is a split Floer cylinder as in Definition \ref{stretched
Y to W}.
\end{definition}

See Figure \ref{fig:Split Floer cylinder with cascades} for a schematic
illustration.

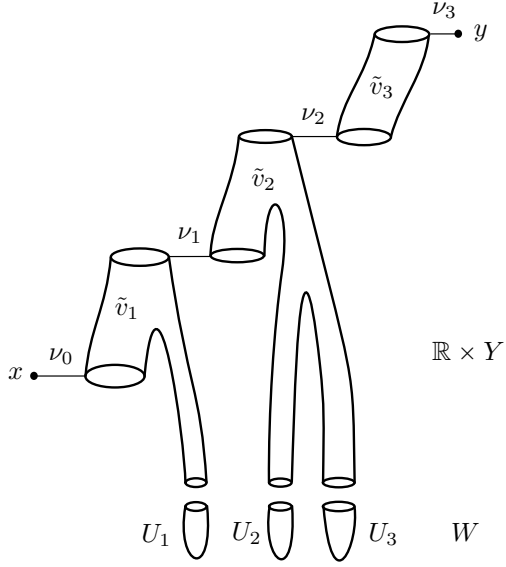
\begin{figure}
  \begin{center}
    \def\svgwidth{0.5\textwidth}

\begingroup
  \makeatletter
  \providecommand\color[2][]{%
    \errmessage{(Inkscape) Color is used for the text in Inkscape, but the package 'color.sty' is not loaded}
    \renewcommand\color[2][]{}%
  }
  \providecommand\transparent[1]{%
    \errmessage{(Inkscape) Transparency is used (non-zero) for the text in Inkscape, but the package 'transparent.sty' is not loaded}
    \renewcommand\transparent[1]{}%
  }
  \providecommand\rotatebox[2]{#2}
  \ifx\svgwidth\undefined
    \setlength{\unitlength}{318.76461123pt}
  \else
    \setlength{\unitlength}{\svgwidth}
  \fi
  \global\let\svgwidth\undefined
  \makeatother
  \begin{picture}(1,1.18489294)%
    \put(0,0){\includegraphics[width=\unitlength]{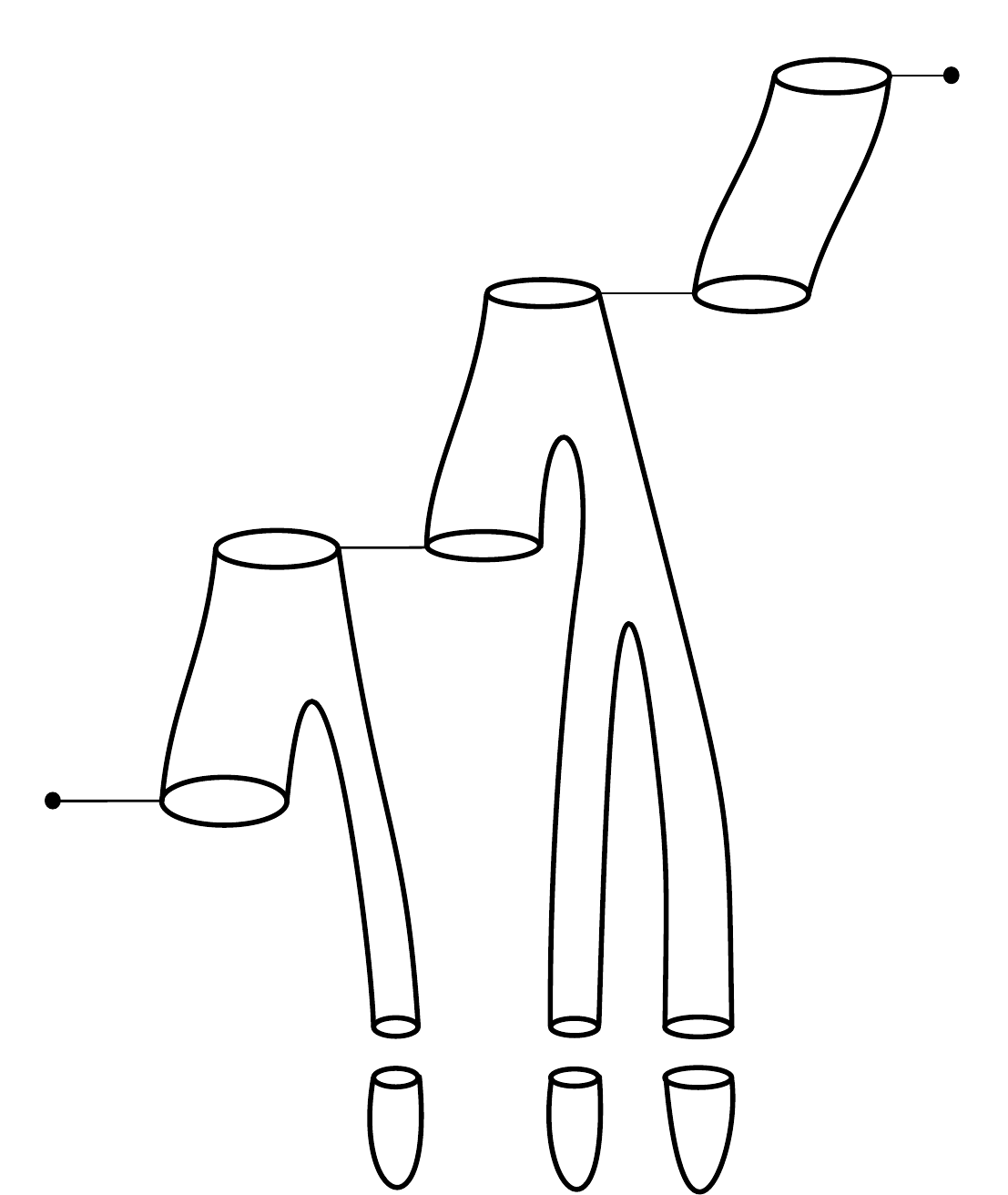}}%
    \put(-0.00166659,0.38440617){\color[rgb]{0,0,0}\makebox(0,0)[lb]{\smash{$x$}}}%
    \put(0.97661872,1.10473558){\color[rgb]{0,0,0}\makebox(0,0)[lb]{\smash{$y$}}}%
    \put(0.2222196,0.51825114){\color[rgb]{0,0,0}\makebox(0,0)[lb]{\smash{$\tilde v_1$}}}%
    \put(0.50937797,0.79080819){\color[rgb]{0,0,0}\makebox(0,0)[lb]{\smash{$\tilde v_2$}}}%
    \put(0.75759958,0.98305829){\color[rgb]{0,0,0}\makebox(0,0)[lb]{\smash{$\tilde v_3$}}}%
    \put(0.08350749,0.42334281){\color[rgb]{0,0,0}\makebox(0,0)[lb]{\smash{$\nu_0$}}}%
    \put(0.35363105,0.67886517){\color[rgb]{0,0,0}\makebox(0,0)[lb]{\smash{$\nu_1$}}}%
    \put(0.61402046,0.92952031){\color[rgb]{0,0,0}\makebox(0,0)[lb]{\smash{$\nu_2$}}}%
    \put(0.88901099,1.15097293){\color[rgb]{0,0,0}\makebox(0,0)[lb]{\smash{$\nu_3$}}}%
    \put(0.27819111,0.04127623){\color[rgb]{0,0,0}\makebox(0,0)[lb]{\smash{$U_1$}}}%
    \put(0.4680077,0.04614319){\color[rgb]{0,0,0}\makebox(0,0)[lb]{\smash{$U_2$}}}%
    \put(0.75516602,0.04370979){\color[rgb]{0,0,0}\makebox(0,0)[lb]{\smash{$U_3$}}}%
    \put(0.88901099,0.42334281){\color[rgb]{0,0,0}\makebox(0,0)[lb]{\smash{$\R\times Y$}}}%
    \put(0.92901099,0.044334281){\color[rgb]{0,0,0}\makebox(0,0)[lb]{\smash{$W$}}}%
  \end{picture}%
\endgroup
   \end{center}
  \caption{A split Floer cylinder with 3 cascades.} 
   \label{fig:Split Floer cylinder with cascades}
\end{figure}

\begin{definition}
    We refer to the punctures $\Gamma$ appearing in Definitions
    \ref{stretched Y to Y} and  \ref{stretched Y to W} as
    \defin{augmentation punctures}.  The corresponding $J_W$ holomorphic
    planes, $U_i \colon \C \to W$ are referred to as \defin{augmentation
    planes}.  This terminology is by analogy to linearized contact
    homology, where rigid planes of this type give an (algebraic) augmentation
    of the full contact homology differential. 
\end{definition} 

\begin{remark}

Notice that the hybrid energy of each sublevel must be non-negative. 
Since we require that the sublevels are non-trivial, 
it follows that any such cascade with
collections of orbits $S_i = Y_{k_i}$, $i=1, \dots, N$ and $S_0 = Y_{k_0}$, or,
if $S_0 = W$, with $k_0 = 0$, we must have that the sequence of multiplicities
is monotone increasing $k_0 < k_1 < \dots < k_N$.

By a standard SFT-type compactness argument, the
Floer--Gromov--Hofer compactification of a moduli space of split Floer cylinders with cascades
will consist of several possible configurations. The length parameters can go to
$0$ or to $\infty$ (in the latter case, corresponding to a Morse-type breaking
of the gradient trajectory). The split Floer cylinders can break at Hamiltonian
orbits, thus increasing the number of cascades but with a length parameter of
$0$. 
They can also split off a holomorphic building with levels in $\R \times
Y$ and in $W$. We will see that this latter degeneration can't occur in low dimensional moduli
spaces, at least under our monotonicity assumptions.  
For energy reasons, these Floer cylinders with cascades will
not break at constant Hamiltonian trajectories in $(-\infty, \log 2] \times Y$ (see \cite{DiogoLisiComplements}*{Lemma 4.8}).

\end{remark}

We now define the split Floer differential $\partial$ on the chain complex (\ref{eqn:chain complex}).
Given generators $x, y$, denote by
$$
\M_{H,N}(x,y;J_W)
$$
the space of split Floer cylinders with $N$ cascades from $x$ to
$y$ (with negative end at $x$ and positive end at $y$).

For $N \ge 1$, this moduli space $\M_{H,N}(x,y; J_W)$ has an $\R^{N}$ action 
by domain automorphisms corresponding to $\R$-translation 
of the domain cylinders $\R \times S^1$. 
When $N = 0$, this moduli space is of gradient trajectories, and also admits an
$\R$ reparametrization action.

When $|x| = |y| -1$, these moduli spaces will be rigid modulo these actions.  
See Remark \ref{rem:degree_minus_one}.

We now define
\begin{equation} \label{eqn:differential}
    \partial \, y = \sum_{|x|=|y| - 1} \# \left( \M_{H, 0}(x,y; J_W) / \R \right ) x
            + \sum_{|x| = |y| - 1} \sum_{N = 1}^\infty \# \left ( \M_{H, N}(x,y; J_W / \R^N
            \right ) x.
\end{equation}

We call $\partial$ the {\em split} Floer differential on \eqref{eqn:chain complex}. 

\section{Transversality for the Floer and holomorphic moduli spaces}
\label{S:transversality}

In this section, we will build the transversality theory needed 
for the Floer cylinders with cascades that appear in the split Floer differential 
as in Equation \eqref{eqn:differential}.
In the process, we will also discuss transversality for pseudoholomorphic
curves in $X$ and in $\Sigma$, which will be necessary for the proof of our
main result.

\subsection{Statements of transversality results}

Before stating the main result of this section, we will introduce some
definitions allowing us to relate transversality for split Floer cylinders
with cascades to transversality problems for spheres in $\Sigma$ and in $X$
with various constraints.

\begin{lemma} \label{L:projectsToSphere}
Let $\tilde v \colon \R \times S^1 \setminus \Gamma \to \R \times Y$ be a
    finite hybrid energy Floer cylinder in $\R \times Y$ (as in Definition \ref{D:hybridEnergy}), converging
    to a Hamiltonian orbit in the manifolds $Y_+$  at $+\infty$ and 
    converging at $-\infty$ either to a Hamiltonian orbit in the manifold $Y_-$ 
    or to a Reeb orbit at $\{-\infty\} \times Y$,
    and with finitely many punctures at $\Gamma \subset \R \times S^1$ converging to Reeb orbits
    in $\{-\infty\} \times Y$.  
    Then, the projection $\pi_\Sigma\circ \tilde v$ extends to a smooth $J_\Sigma$-holomorphic
sphere $\pi_\Sigma \circ \tilde v \colon \CP^1 \to \Sigma$.

\end{lemma}
\begin{proof}
The projection $\pi_\Sigma \circ \tilde v$ is $J_\Sigma$-holomorphic on $\R\times S^1$, 
since $H$ is admissible (as in Definition \ref{def:J_shaped}). 
 The result now follows from Gromov's removal of singularities theorem
together with the finiteness of the energy of $\pi_\Sigma \circ \tilde v$.
\end{proof}

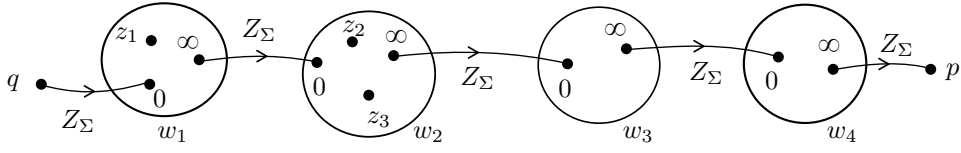
\begin{figure}
  \begin{center}
    \def\svgwidth{1\textwidth}

\begingroup
  \makeatletter
  \providecommand\color[2][]{%
    \errmessage{(Inkscape) Color is used for the text in Inkscape, but the package 'color.sty' is not loaded}
    \renewcommand\color[2][]{}%
  }
  \providecommand\transparent[1]{%
    \errmessage{(Inkscape) Transparency is used (non-zero) for the text in Inkscape, but the package 'transparent.sty' is not loaded}
    \renewcommand\transparent[1]{}%
  }
  \providecommand\rotatebox[2]{#2}
  \ifx\svgwidth\undefined
    \setlength{\unitlength}{561.72030987pt}
  \else
    \setlength{\unitlength}{\svgwidth}
  \fi
  \global\let\svgwidth\undefined
  \makeatother
  \begin{picture}(1,0.14241963)%
    \put(0,0){\includegraphics[width=\unitlength]{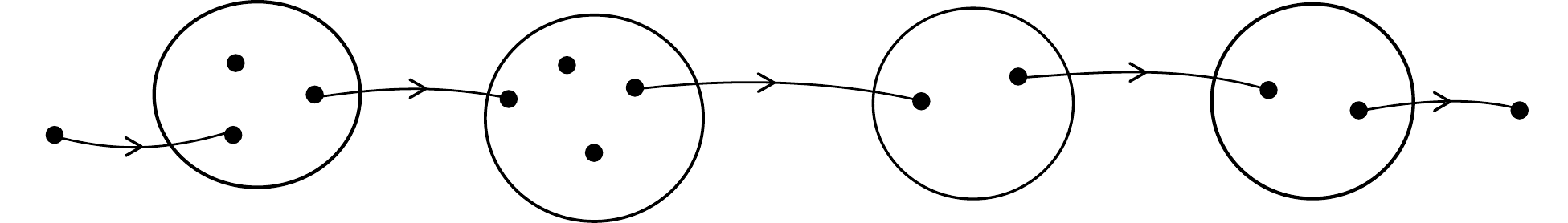}}%
    \put(0.1770181,0.093){\color[rgb]{0,0,0}\makebox(0,0)[lb]{\smash{${\infty}$}}}%
    \put(0.32,0.04529838){\color[rgb]{0,0,0}\makebox(0,0)[lb]{\smash{$0$}}}%
    \put(0.1570181,0.00){\color[rgb]{0,0,0}\makebox(0,0)[lb]{\smash{${w_1}$}}}%
    \put(0.39562311,0.0999673){\color[rgb]{0,0,0}\makebox(0,0)[lb]{\smash{$\infty$}}}%
    \put(0.5783324,0.04242108){\color[rgb]{0,0,0}\makebox(0,0)[lb]{\smash{$0$}}}%
    \put(0.42562311,0.00){\color[rgb]{0,0,0}\makebox(0,0)[lb]{\smash{$w_2$}}}%
    \put(0.62587865,0.1071606){\color[rgb]{0,0,0}\makebox(0,0)[lb]{\smash{$\infty$}}}%
    \put(0.79556945,0.05105301){\color[rgb]{0,0,0}\makebox(0,0)[lb]{\smash{$0$}}}%
    \put(0.64587865,0.00){\color[rgb]{0,0,0}\makebox(0,0)[lb]{\smash{$w_3$}}}%
    \put(0.85030883,0.0898967){\color[rgb]{0,0,0}\makebox(0,0)[lb]{\smash{$\infty$}}}%
    \put(0.15256096,0.03235053){\color[rgb]{0,0,0}\makebox(0,0)[lb]{\smash{$0$}}}%
    \put(0.86030883,0.00){\color[rgb]{0,0,0}\makebox(0,0)[lb]{\smash{$w_4$}}}%
    \put(0.11227862,0.10428326){\color[rgb]{0,0,0}\makebox(0,0)[lb]{\smash{$z_1$}}}%
    \put(0.35109544,0.11075554){\color[rgb]{0,0,0}\makebox(0,0)[lb]{\smash{$z_2$}}}%
    \put(0.37699123,0.01508662){\color[rgb]{0,0,0}\makebox(0,0)[lb]{\smash{$z_3$}}}%
    \put(-0.00137517,0.05393031){\color[rgb]{0,0,0}\makebox(0,0)[lb]{\smash{$q$}}}%
    \put(0.98554246,0.06543958){\color[rgb]{0,0,0}\makebox(0,0)[lb]{\smash{$p$}}}%
    \put(0.91554246,0.09043958){\color[rgb]{0,0,0}\makebox(0,0)[lb]{\smash{$Z_\Sigma$}}}%
    \put(0.71554246,0.06043958){\color[rgb]{0,0,0}\makebox(0,0)[lb]{\smash{$Z_\Sigma$}}}%
    \put(0.4754246,0.05543958){\color[rgb]{0,0,0}\makebox(0,0)[lb]{\smash{$Z_\Sigma$}}}%
    \put(0.2454246,0.10543958){\color[rgb]{0,0,0}\makebox(0,0)[lb]{\smash{$Z_\Sigma$}}}%
    \put(0.0554246,0.010543958){\color[rgb]{0,0,0}\makebox(0,0)[lb]{\smash{$Z_\Sigma$}}}%
  \end{picture}%
\endgroup
   \end{center}
  \caption{A chain of 4 pearls from $q$ to $p$ with 3 marked points.} 
   \label{fig:linear tree marked pts}
\end{figure}

In order to describe the projection to $\Sigma$ of the levels of a split Floer
cylinder with $N$ cascades that map to $\R \times Y$, we introduce the
following:
\begin{definition} \label{D:chain of pearls}
A \defin{chain of pearls} from $q$ to $p$, where $p$ and $q$ are critical points of $f_\Sigma$, consists of the following:
\begin{itemize}
\item $N \ge 0$ parametrized $J_\Sigma$-holomorphic spheres $w_i$ in $\Sigma$
    with two distinguished marked points at $0$ and $\infty$ and a possibly
    empty collection of additional marked points $z_1, \dots, z_k$ on the union
    of the $N$ domains (distinct from $0$ or $\infty$ in each of the $N$
    spherical domains); 
    the spheres are either non-constant or contain at least one additional marked point; 
\item if $N=0$, an infinite positive flow trajectory of $Z_\Sigma$ from $q$ to $p$; if $N\geq 1$, 
    a half-infinite trajectory of $Z_\Sigma$ from $q$ to $w_1(0)$, 
    a half-infinite trajectory of $Z_\Sigma$ from $w_N(\infty)$ to $p$;
\item if $N \ge 1$, positive length parameters $l_i, i=1, \dots, N-1$, 
    so that $\varphi^{l_i}_{Z_\Sigma}( w_i(\infty) ) = w_{i+1}(0)$, $i=1,
    \dots, N-1$.
\end{itemize}
\end{definition}

See Figure \ref{fig:linear tree marked pts}. If such a chain of
pearls is the projection to $\Sigma$ of the components in $\R\times
Y$ is a split Floer cylinder, then the additional marked points in
the pseudoholomorphic spheres correspond to augmentation punctures in the Floer
cylinders, where they converge to cylinders over Reeb orbits that
are capped by planes in $W$.

Notice that the geometric configuration of two spheres touching at a critical point of $f_\Sigma$ 
admits an interpretation as a chain of pearls in $\Sigma$, since the critical
point is the image of any positive length flow line with that initial condition.

\begin{definition} \label{D:chain of pearls w sphere in X}
A \defin{chain of pearls with a sphere in $X$} from $x$ to $p$, where 
$x$ is a critical point of $f_W$ and $p$ is a critical point of $f_\Sigma$, consists of the following:
\begin{itemize}
\item $N \ge 1$ parametrized $J_\Sigma$-holomorphic spheres $w_i$ in $\Sigma$ with two distinguished marked points at $0$ and $\infty$ and a possibly empty collection of additional marked points $z_1, \dots, z_k$ on the union of the $N$ domains (distinct from $0$ or $\infty$);
\item a parametrized non-constant $J_X$-holomorphic sphere $v$ in $X$;
\item a half-infinite trajectory of $Z_\Sigma$ from 
    $w_N(\infty)$ to $p$, a half-infinite trajectory of $-Z_X$ from $x$ to $v(0)$ 
    (where $Z_X$ is the push-forward of $Z_W$ by the inverse of the map from Lemma
    \ref{planes = spheres});
\item positive length trajectories of $Z_\Sigma$ from $w_{i}(\infty)$ to $w_{i+1}(0)$ for $i=1, \dots, N-1$;
\item the sphere in $X$ touches the first sphere in $\Sigma$: $w_1(0) =
    v(+\infty)$;
\item the spheres $w_2, \dots, w_N$ satisfy the stability condition that they are either non-constant
or contain at least one of the additional marked points ($v$ is automatically
non-constant and $w_1$ is allowed to be constant and unstable).
\end{itemize}
\end{definition}

\begin{definition} \label{D:aug chain of pearls}
An \defin{augmented chain of pearls} [or an \defin{augmented chain of pearls with
a sphere in $X$}] is a chain of pearls [or chain of pearls with a
sphere in $X$] together with $k$ equivalence classes $[U_i]$ of $J_X$-holomorphic spheres $U_i \colon \CP^1 \to
X$, $i=1, \dots, k$, with the following additional properties:
\begin{itemize}
    \item for each $z \in \CP^1$, $U_i(z) \in \Sigma$ if and only if $z =
        \infty$;
    \item if the puncture $z_i$ is in the domain of the holomorphic sphere
        $w_{j_i} \colon \CP^1 \to \Sigma$, then $w_{j_i}(z_i) = U_i(\infty)$;
    \item each $U_i$ is considered up to the action of $\Aut(\CP^1,\infty) = \Aut(\C)$, that is, 
        as an unparametrized sphere. 
\end{itemize}
\end{definition}

From Lemma \ref{L:projectsToSphere} 
and Lemma \ref{planes = spheres} and the fact that the trajectories of $Z_Y$ cover
trajectories of $Z_\Sigma$, it follows that a Floer cylinder with $N$
cascades projects to a chain of pearls or a chain of pearls with a sphere
in $X$. Additionally, again by Lemma \ref{planes = spheres}, 
if any of the sublevels have augmentation planes,
then those correspond to spheres in $X$ passing through $\Sigma$ at the
images of the corresponding marked points in the chain of pearls.

Observe that we allow the sphere $w_1$ to be unstable in the definition
of a chain of pearls in $\Sigma$ with a sphere in $X$.  The case in
which $w_1$ is a constant curve without marked points corresponds to the
situation in which the corresponding Floer cascade contains a non-trivial
Floer cylinder $\tilde v_1$ contained in a single fibre of $\R \times
Y \to \Sigma$, and has the asymptotic limits $\tilde v_1(+\infty, t)$
on a Hamiltonian orbit and $\tilde v_1(-\infty, t)$ on a closed Reeb
orbit in $\{ -\infty \} \times Y$.
The Floer cylinder $\tilde v_1$ in $\R\times Y$ is non-trivial and hence
stable, whereas the
corresponding sphere $w_1$ in $\Sigma$ is unstable. Since we do not
quotient by automorphisms (yet), this does not pose a problem.
(see Figure \ref{case_3_fig} and Proposition \ref{P:Floer_also_in_W} below, where this
situation is analysed.)

\begin{definition} \label{D:simple chain pearls}

A chain of pearls in $\Sigma$ is \textit{simple} if each sphere is either simple 
(i.e. not multiply covered, \cite{McDuffSalamon}*{Section 2.5}) or is constant,
and if the image of no sphere is contained in the image of another.
If the chain of pearls has a sphere $v$ in $X$, we require $v$ to be somewhere injective 
(but the first sphere in $\Sigma$ is allowed to be constant, with image contained in the image of $v$).

An augmented chain of pearls is simple if the chain of pearls is simple and
the augmentation spheres in 
$X$ are somewhere
injective, none has image contained in the fixed open neighbourhood $\varphi(\U)$ 
and no sphere in $X$ has image contained in the image of
another sphere in $X$.
\end{definition}

\begin{remark}
    Recall that a chain of pearls with a sphere in $X$ has a distinguished
    sphere $v$ in $X$ for which $v(0)$ is on the descending manifold of a critical
    point $x$ of $f_W$. By the construction of $f_W$, this forces the image of $v$ to intersect the complement
    of the tubular neighbourhood of $\Sigma$. As we revisit in 
    Remark \ref{rem:augmentation planes not near Sigma},
    Fredholm index considerations related to monotonicity will force 
    the augmentation planes/spheres to leave the tubular neighbourhood.
\end{remark}

\begin{remark}
    Notice that our condition on a simple chain of pearls is slightly different
    than the condition imposed in \cite{McDuffSalamon}*{Section 6.1}, with
    regard to constant spheres. For a chain of pearls to be simple by our
    definition, constant spheres may not be contained in another sphere,
    constant or not. In \cite{McDuffSalamon}, there is no such 
    condition on constant spheres. 
\end{remark}

\begin{definition} \label{D:simple cascade}
Given a finite hybrid energy Floer cylinder with $N$ cascades, we obtain an
augmented chain of pearls (possibly with a sphere in $X$) by the following
construction:
\begin{enumerate}
    \item cylinders in $\R \times Y$ are projected to $\Sigma$: by 
        Lemma \ref{L:projectsToSphere} these form holomorphic spheres in
        $\Sigma$;
    \item planes in $W$ are interpreted as spheres in $X$ by Lemma
        \ref{planes = spheres};
    \item flow lines of the gradient-like vector field $Z_Y$ are projected to
        flow lines of $Z_\Sigma$.
\end{enumerate}
We refer to this augmented chain of pearls in $\Sigma$ (possibly with a sphere
in $X$) as the projection of the Floer cylinder with $N$ cascades.

A finite hybrid energy Floer cylinder with $N$ cascades is \defin{simple}
if the projected chain of pearls is simple.
\end{definition}

Given generators $x, y$ of the chain complex (\ref{eqn:chain complex}), denote by
$$
\M_{H,N}^*(x,y;J_W)
$$
the space of {\em simple} split Floer cylinders with $N$ cascades from $x$ to $y$. 
Recall that if $x$ or $y$ is in $\R \times Y$, the corresponding generator is described by a critical point $\widetilde p$ of $f_Y$
(which can be either $\wc p$ or $\wh p$), together with a multiplicity $k$. If instead, $x$ or $y$ is in $W$, it
corresponds to a critical point of $f_W$.     Note that when both $x, y$ are generators in $\R \times Y$, 
    the moduli space $\M_{H,N}^*(x,y; J_W)$ will depend on $J_W$ only insofar as 
    augmentation planes appear, otherwise it depends only on $J_Y$.

\begin{proposition} \label{cascades form manifold}
    There exists a residual set $\mathcal J^{reg}_W \subset \mathcal J_W$ of
    almost complex structures such that for each $J_W \in \mathcal J^{reg}_W$,
    $\M_{H,N}^*(x,y;J_W)$ is a manifold. 

 If $N=0$, and thus $x, y$ are generators in $\R \times Y$, 
 then  $x = \widetilde q_k$, $y = \widetilde p_k$ for the same multiplicity $k$, and 
     \[
         \dim_\R \M_{H,0}^*(\widetilde q_k,\widetilde p_k;J_W) = |\widetilde p_k| - |\widetilde q_k|.
     \]

If $N \ge 1$, and $x, y$ are generators in $ \R \times Y$, then $x = \widetilde q_{k_-}$, 
$y = \widetilde p_{k_+}$ and 
\[
    \dim_\R \M_{H,N}^*(\widetilde q_{k_-},\widetilde p_{k_+};J_W) = |\widetilde
p_{k_+}| - |\widetilde q_{k_-}| + N-1 .
\]

Finally, if $x \in W$ and $y \in \R \times Y$, then $x\in \Crit(f_W)$, $y = \widetilde p_k$  and 
\[
    \dim_\R \M_{H,N}^*(x,\widetilde p_k;J_W) = |\widetilde p_k| - |x| + N. 
\]

Furthermore, the image $P(\mathcal J^{reg}_W) \subset \mathcal J_\Sigma$ (recall Definition \ref{def:J_W}) 
is residual and consists of almost complex structures that are regular for simple pseudoholomorphic spheres in $\Sigma$.
\end{proposition}

    The two different formulas involving $N$ reflect the fact that $N$ counts the
    number of cylinders in $\R \times Y$. In the case of a Floer cascade that
    descends to $W$, there are therefore $N+1$ cylinders in the cascade.

\begin{remark} \label{rem:degree_minus_one}
    These index formulas justify that the moduli spaces are rigid (modulo
    their $\R$, $\R^{N}$ and $\R^{N+1}$ actions) when the index difference is
    $1$, which then justifies the definition of the differential given in
    Equation \eqref{eqn:differential}.
    Indeed,
    observe that the case $N=0$ corresponds to a pure Morse configuration and
    doesn't depend on any almost complex structure. We count rigid flow lines modulo the 
    $\R$ action, and thus require $|y| - |x| = 1$. For generators $x, y$ in $\R \times Y$,
    we consider these $N$ cylinders modulo the $\R$ action on each one, giving an $\R^N$ action.
    From this, a rigid configuration has $|y|-|x|+N-1 = N$. 
    For the case with $x \in W$, we have $N+1$ cylinders in the Floer cascade, so we have a 
    rigid configuration modulo the $\R^{N+1}$ action when $N+1 = |y|-|x|+N$.
\end{remark}

The split Floer differential $\partial$, introduced in Equation
\eqref{eqn:differential}
was defined by counting elements in $\M_{H,N}(x,y;J_W)$. 
We will see in Propositions \ref{P:Floer_in_R_times_Y} and
\ref{P:Floer_also_in_W}
that our monotonicity assumptions 
imply that this is equivalent to counting simple configurations in $\M_{H,N}^*(x,y;J_W)$.

The rest of this section will be devoted to the proof of 
Proposition \ref{cascades form manifold}. It will proceed in the following
steps:
\begin{itemize}
    \item Section \ref{S:fredholm_for_cascades} describes 
        the Fredholm set-up for Floer cylinders with cascades. On a first reading, it can be skipped and used as a reference. 
        In Section \ref{S:sobolev_morse_bott}, we discuss the necessary function spaces and linear
        theory for the Morse--Bott problems. Then, 
         Section \ref{S:linearization at Floer} splits the linearization of the
        Floer operator in such a way as to split the transversality problem into two problems.
        The first is a Cauchy--Riemann operator acting on sections of a complex line
        bundle, and it is transverse for topological reasons (automatic
        transversality).  The second is a transversality problem 
        for a Cauchy--Riemann operator in $\Sigma$.
    \item Section \ref{S:transverse pearls} adapts the transversality arguments
        from \cite{McDuffSalamon} in order to obtain transversality for chains
        of pearls in $\Sigma$. 
    \item Section \ref{S:transverse contact} shows transversality for 
        the components of the cascades contained in $W$. This problem is
        translated into the equivalent problem of obtaining transversality for
        spheres in $X$ with order of contact conditions at $\Sigma$, together
        with evaluation maps. The main technical point is an extension of the
        transversality results from \cite{CieliebakMohnkeTransversality}.
    \item Finally, Section \ref{S:proof} 
        uses the splitting from Section \ref{S:linearization at Floer} 
        to lift the transversality results in $\Sigma$ to obtain transversality
        for Floer cylinders with cascades, and to finish the proof of Proposition \ref{cascades form manifold}. 
\end{itemize}

\subsection{A Fredholm theory for Floer cylinders with cascades}
\label{S:fredholm_for_cascades}

\subsubsection{ A Fredholm theory for Morse--Bott asymptotics}
\label{S:sobolev_morse_bott}

In this section, we collect some facts about Cauchy--Riemann operators on
Hermitian vector bundles over punctured Riemann surfaces, specifically in the
context of \textit{degenerate} asymptotic operators. 
These facts can mostly be found in the literature, but not in a unified way. 
The main reference for these results is
\cite{SchwarzThesis}. Additional references include \citelist{
    \cite{HWZ3}
    \cite{SchwarzThesis}
    \cite{WendlAutomatic} \cite{ACHplanar}*{Sections 2.1--2.3} \cite{BourgeoisMohnke}}. 

We begin by introducing some Sobolev spaces of sections of appropriate bundles.
Let $\Gamma \subset \R \times S^1$ be a finite set of punctures and 
denote $\R \times S^1 \setminus \Gamma$ by $\dot S$. 
Write $\Gamma_+ = \{+\infty\}$ and $\Gamma_- = \{-\infty\} \cup \Gamma$. 
Consider, for each puncture $z\in \Gamma$, exponential cylindrical 
polar coordinates of the form $(-\infty, -1] \times S^1 \to \R \times S^1
\setminus \Gamma  \colon \rho + i \eta \mapsto z_0 + \epsilon \e^{2\pi(\rho + i
    \eta)}$. Choose $\epsilon > 0$ sufficiently small these are embeddings and 
    that the image of these embeddings for any two different punctures are disjoint.

Let $E \to \dot S$ be a (complex) rank $n$ Hermitian vector bundle over
$\dot S$ together with a preferred set of trivializations in a small
neighbourhood of $\Gamma \cup \{\pm \infty\}$.  
While the bundle $E$ over $\dot S$ is trivial if there is at least one puncture, this is no
longer the case once we specify these preferred trivializations near
$\Gamma \cup \{\pm\infty\}$.  
We therefore associate a first Chern number
to this bundle relative to the asymptotic trivializations.  There are
several equivalent definitions. One approach is to consider the complex
determinant bundle $\Lambda_\C^n E$. The trivialization of $E$ at infinity
gives a trivialization of this determinant bundle at infinity, and we can
now count zeros of a generic section of $\Lambda_\C^n E$ that is constant
(with respect to the prescribed trivializations) near the punctures. We
denote this Chern number by $c_1(E)$, but emphasize that it depends on
the choice of these trivializations near the punctures.

Since we cannot specify where an augmentation puncture appears when
we stretch the neck on a Floer cylinder, we should have the punctures
in $\Gamma$ free to move on the domain $\R \times S^1$.  
This creates a
problem when we try to linearize the Floer operator in a family of domains
where the positions of the punctures are not fixed.  
We will instead
consider a $2\#\Gamma$ parameter family of almost complex structures
on $\R \times S^1$, but fix the location of the punctures.
Specify a fixed collection $\Gamma$ of punctures on $\R\times S^1$
and, for any other collection of augmentation punctures, choose an
isotopy with compact support from the new punctures to the fixed ones.
We take the push-forward of the standard complex structure in $\R\times
S^1$ by the final map of the isotopy, to produce a family of complex
structures on $\R \times S^1$, which can be assumed standard near $\Gamma$
and outside of a compact set. 

For each $z \in \Gamma$, let $\beta_{z} \colon \dot S \to [0, +\infty)$ be a function supported in a small
neighbourhood of $z$, with $\beta_{z}(\rho  , \eta) = -\rho$ near the
puncture (where $(\rho, \eta)$ are cylindrical polar coordinates near $z$, as above). Similarly, let $\beta_+ \colon \R \times S^1 \to [0, +\infty)$ be  
supported in a region where $s$ is sufficiently large and $\beta_+(s,t) = s$ for $s$ large enough.  Let $\beta_- \colon \R \times S^1 \to [0, +\infty)$ have
support near $-\infty$, and $\beta_-(s,t) = -s$ for $s$ sufficiently small.

In many situations, it will be convenient to consider the function 
\begin{equation} \label{beta}
    \beta \coloneqq \sum_{z \in \Gamma} \beta_{z} + \beta_- + \beta_+.
\end{equation}

Finally, on the punctured cylinder $\dot S$, we take the measure induced by an
area form on $\dot S$ that has the form $ds \wedge dt$ for $|s|$ large and that
has the form $d\rho \wedge d\eta$ in the cylinder polar coordinates near each
puncture in $\Gamma$. Notice that pairing this with the domain complex
structure induces a metric on $\dot S$ for which the vector field
$\partial_\eta$, defined 
near a puncture in $\Gamma$ by the exponential cylindrical polar coordinates, 
has norm comparable to $1$.

Given a vector of weights $\mathbf \delta \colon \Gamma \cup \{ \pm\infty \} \to \R$, we define $W^{1,p,\mathbf \delta}(\dot S, E)$ to be the space of sections $u$ of $E$ for which 
\[
    u \e^{\sum \delta_z \beta_z + \delta_- \beta_- + \delta_+ \beta_+} 
\in W^{1,p}(\dot S, E) \]
(with respect to the measure and metric described above).
Note that these sections decay exponentially fast at the punctures where $\delta
>0$ and are allowed to have exponential growth at punctures where $\delta < 0$.
We can similarly define $L^{p, \mathbf{\delta}}(\dot S, E)$.
While these definitions involve making various choices, the resulting metrics
are strongly equivalent. In practice, we'll typically take $p > 2$ to obtain
continuity of the sections. By a similar construction, we may define $W^{m,p,
\mathbf{\delta}}$ as well.

We will say that a differential operator $D \colon \Gamma(E) \to
\Gamma(\Lambda^{0,1} T^*\dot S \tensor E)$ is a Cauchy--Riemann operator if it is
a real linear Cauchy--Riemann operator \cite{McDuffSalamon}*{Definition
C.1.5} such that, near $\pm \infty$, it takes the form:
\begin{equation}
(D \sigma) \frac{\partial}{\partial s} = \frac{\partial}{\partial s} \sigma + J(s,t) \frac{\partial}{\partial t} \sigma + A(s,t) \sigma
\label{CR}
\end{equation}
where $J(s,t)$ is a smooth function on $\R^\pm \times S^1$ with values
in almost complex structures on $\C^n$ compatible with the standard
symplectic form, and $A(s,t)$ takes values in real matrices on $\R^{2n}
\cong \C^n$. We further impose that these functions converge uniformly as $s \to
\pm \infty$, $J(s,t) \to J_z(t)$ and $A(s,t) \to A_z(t)$, where $A_z(t)$
is a loop of self-adjoint matrices. We impose the same conditions near
punctures $z\in \Gamma$, using the local coordinates $(\rho,\eta)$
instead of $(s,t)$ in \eqref{CR}.

    A Cauchy--Riemann operator $D \colon \Gamma(E) \to \Gamma( \Lambda^{0,1}
    T^*\dot S \tensor E)$ acting on smooth sections induces an
    operator on various Sobolev spaces of sections. Of particular
    importance for us will be that for any vector of weights $\delta \colon
    \Gamma \cup \{ \pm \infty \} \to \R$, $D$ induces an operator 
    \[
        D \colon W^{1,p,\delta}(\dot S, E) \to L^{p, \delta}(\dot S,
    \Lambda^{0,1}T^*\dot S \tensor E) \]
    as well as operators $D \colon W^{k,p, \delta} \to W^{k-1, p, \delta}$.
    In the following, we will not emphasize the distinction and refer to
    the operators $W^{1,p,\delta} \to L^{p,\delta}$ as Cauchy--Riemann
    operators. 
    For generic weight vectors, these operators will be
    Fredholm as formulated precisely in Theorem \ref{T:RiemannRoch} below.

Associated to a Cauchy--Riemann operator $D$, we obtain
\defin{asymptotic operators} at each puncture in $\Gamma
\cup\{\pm\infty\}$ by $\mathbf{A}_z \coloneqq - J_z(t) \frac{d}{dt} -
A_z(t)$. This is a densely defined unbounded self-adjoint operator on
$L^2(S^1, \R^{2n})$. Let $\sigma(\mathbf{A}_z) \subset \R$ denote its
spectrum. This will consist of a discrete set of eigenvalues. If an
asymptotic operator $\mathbf{A}_z$ does not have $0$ in its spectrum,
we say the asymptotic operator is non-degenerate. If all the asymptotic
operators are non-degenerate, we say $D$ itself is non-degenerate.

Note that we obtain a path of symplectic matrices associated to the
asymptotic operator $\mathbf{A}_z$ by finding the fundamental matrix
$\Phi$ to the ODE $\frac{d}{dt} x = J_z(t) A_z(t) x$. The asymptotic
operator is non-degenerate if and only if the time-1 flow of the ODE
does not have $1$ in the spectrum. 
We will consider a description of the
Conley--Zehnder index in terms of properties of the asymptotic operator
itself \cite{HWZ2}*{Lemmas 3.4, 3.5, 3.6, 3.9}.

\begin{remark}
An asymptotic operator induces a path of symplectic matrices, and this
identification (understood correctly) is a homotopy equivalence. 
This will allow us to 
associate a Conley--Zehnder index to a periodic orbit of a Hamiltonian vector field, 
given a trivialization of the tangent bundle along the orbit. In order to do so, 
we take the linearized flow map, which defines a path $\Phi\colon [0,1] \to \Sp(2n)$ 
with respect to the fixed trivialization. 
If we fix a path of almost complex structures, this path of symplectic matrices
satisfies an ODE as in the previous paragraph, which in turn specifes an
asymptotic operator. The Conley--Zehnder index of the Hamiltonian orbit is by
definition the Conley--Zehnder index of this asymptotic operator. This is
homotopic to the asymptotic operator coming from the linearized Floer operator.
\end{remark}

\begin{proposition} \label{P:CZ winding}

Suppose $\mathbf{A}_z$ is non-degenerate and $E$ is a rank 1 vector
bundle. 

If $u \colon S^1 \to \C$ is an eigenfunction of $\mathbf{A}_z$ corresponding to
the eigenvalue $\lambda$, it must be nowhere vanishing. The winding number of
$u$ is then the degree of the map $\frac{u}{|u|} \colon S^1 \to S^1$. 
Any two eigenfunctions corresponding to the same eigenvalue $\lambda$ have the same
winding number. This is then referred to as the winding number of the
eigenvalue, and is denoted by $w(\lambda)$. 

The function $w \colon \sigma( \mathbf{A}_z) \to \Z$ 
is non-decreasing in $\lambda$ and is surjective. If $\lambda_\pm$
are eigenvalues so that $\lambda_- < 0 < \lambda_+$ and there are no
eigenvalues in the interval $(\lambda_-, \lambda_+)$, then \[ \CZ(
\mathbf{A}_z ) = w(\lambda_-) + w(\lambda_+).  \]

\qedhere

\end{proposition}

This formulation will be the most useful for our
calculations. Furthermore, in the case of a higher rank bundle, we use
the axiomatic description, see for instance \cite{HWZ2}*{Theorem 3.1}
to observe that $\CZ( \mathbf{A}_z )$ is invariant under deformations
for which $0$ is never in the spectrum, and that if the operator can
be decomposed as the direct sum of operators, then the Conley--Zehnder index
is additive.

The following computation is useful at several points in the paper. It
can often be combined with Proposition \ref{P:CZ winding} to compute
Conley--Zehnder indices of interest.
\begin{lemma}\label{L:spectrum}
Given a constant $C\geq 0$, the spectrum $\sigma(A_C)$ of the operator
$$
\mathbf{A}_C:= - i \frac{d}{dt} - %
\begin{pmatrix} C & 0 \\
0 & 0  
\end{pmatrix} %
\colon W^{1,p}(S^1,\C) \to L^p(S^1,\C)
$$
is the set
 $$\left \{\frac{1}{2}\left(- C - \sqrt{ C^2 + 16 \pi^2 k^2}\right) \, | \, k\in \Z \right\} \cup \left \{\frac{1}{2}\left(- C + \sqrt{ C^2 + 16 \pi^2 k^2}\right) \, | \, k\in \Z \right\}.$$
If $\lambda$ is an eigenvalue associated to $k\in \Z$, then the winding number of the corresponding eigenvector is $|k|$ if $\lambda\geq 0$ and $-|k|$ if $\lambda\leq 0$. If $C=0$, then all eigenvalues have multiplicity 2 (see Table \ref{evals 0}). If $C>0$, then the same is true except for the eigenvalues $-C$ and $0$, corresponding to $k=0$ above, both of which have multiplicity 1 (see Table \ref{evals C}).  
 
In particular, the $\sigma(\mathbf{A}_0) = 2\pi\Z$ and the winding number of $2\pi k$ is $k$. 
\end{lemma}
\begin{proof}
An eigenvector $v : S^1 \to \C$ of $\mathbf{A}_C$ with eigenvalue $\lambda$ solves the equation 
 $$
 - \begin{pmatrix} 0 & -1 \\
1 & 0  
\end{pmatrix}  \dot v - 
\begin{pmatrix} C & 0 \\
0 & 0 
\end{pmatrix} v = \lambda v \Leftrightarrow \dot v = \begin{pmatrix} 0 & -\lambda \\
 C + \lambda & 0
\end{pmatrix} v.
 $$
Computing the eigenvalues of the matrix on the right, and requiring that they be of the form $2\pi i k$, $k\in \Z$ (since  $v(t+1)= v(t)$), yields the result. 
\end{proof}

\begin{table}[ht]
\begin{tabular}{c|c|c|c|c|c|c|c}
eigenvalues & \ldots & $-4\pi$ & $-2\pi$ & $0$ & $2\pi$ & $4\pi$ & \ldots\\
\hline
multiplicities & \ldots & 2 & 2 & 2 & 2 & 2 &\ldots  \\
\hline
winding numbers & \ldots & $-2$ & $-1$ & 0 & 1 & 2 & \ldots \\
\end{tabular}
\caption{Eigenvalues of $\mathbf{A}_0$}
  \label{evals 0}
\end{table}

\begin{table}[ht]
\begin{tabular}{c|c|c|c|c|c|c}
eigenvalues & \ldots & $\frac{1}{2}\left(- C - \sqrt{ C^2 + 16 \pi^2}\right)$ & $-C$ & $0$ & $\frac{1}{2}\left(- C + \sqrt{ C^2 + 16 \pi^2}\right)$ & \ldots\\
\hline
multiplicities & \ldots & 2 & 1 & 1 & 2 &\ldots  \\
\hline
winding \#s & \ldots & $-1$ & 0 & 0 & 1 & \ldots \\
\end{tabular}
\caption{Eigenvalues of $\mathbf{A}_C$ in increasing order, if $C>0$}
  \label{evals C}
\end{table}

\begin{corollary} \label{C:CZ computation}
Take $C\geq 0$ and $\delta>0$ such that $[-\delta,\delta] \cap \sigma(\mathbf{A}_C) = \{0\}$. Then
$$
\CZ(\mathbf{A}+ \delta) = \begin{cases}
                     0 & \text{ if } C > 0 \\
                     -1 & \text{ if } C = 0 
                    \end{cases}
\qquad \text{and} \qquad
\CZ(\mathbf{A}_C - \delta) = 1. 
$$

For any $n\geq 0$, taking 
$$
-i\frac{d}{dt}\colon W^{1,p}(S^1,\C^n) \to L^p(S^1,\C^n),
$$
we have 
$$
\CZ\left(-i\frac{d}{dt} \pm \delta\right) = \mp n.
$$
\end{corollary}
\begin{proof}
 The case $n=1$ follows from Proposition \ref{P:CZ winding} and Lemma \ref{L:spectrum}. 
 The case of general $n$ uses the additivity of $\CZ$ under direct sums. 
\end{proof}

\begin{definition} \label{def:delta_perturbed}
A key observation for our computations of Fredholm indices (as noted, for instance, in \cite{HWZ3}) is that a Cauchy--Riemann operator 
\[
D \colon W^{1,p,\mathbf{\delta}}(\dot S, E) \to L^{p, \mathbf{\delta}}(\dot S,  \Lambda^{0,1} T^*\dot S \tensor E)
\]
with asymptotic operators $\mathbf{A}_z$ 
is conjugate to the Cauchy--Riemann operator
\begin{align*}
D^\mathbf{\delta} &\colon  W^{1,p}(\dot S, E) \to L^{p}(\dot S,  \Lambda^{0,1} T^*\dot S \tensor E) \\
&D^\mathbf{\delta} =  \e^{ \sum \delta_z \beta_z(s) }D \e^{ -\sum \delta_z \beta_z(s) }.
\end{align*}
This has asymptotic operators $\mathbf{A}^\mathbf{\delta}_z = \mathbf{A}_z \pm
\delta_z$,
which are non-degenerate and 
 where the sign is positive at positive punctures and negative at
negative punctures. We refer to these as the {\em $\delta$-perturbed
asymptotic operators}.

Notice that the operator $D^\mathbf{\delta}$ depends on the choice of cut-off
functions $\beta_z, z \in \Gamma \cup \{ \pm \infty \}$. A different choice of
cut-off function will give an operator that differs only by a compact operator.
This is thus of secondary importance for what we discuss here.
\end{definition}

    Note that with the sign conventions that we have chosen, 
    a positive weight $\delta_z > 0$ always corresponds to the constraint of
    exponential decay at the puncture. A negative weight $\delta_z < 0$ always corresponds to
    allowing exponential growth.

\begin{theorem} \label{T:RiemannRoch}
    Let $\mathbf{\delta} \colon \Gamma \cup \{ \pm \infty \} \to \R$ such that 
    $-\delta_z \notin \sigma( \mathbf{A}_z )$ for positive punctures $z \in
    \Gamma_+$ and such that $+\delta_z \notin \sigma( \mathbf{A}_z )$ for
    negative punctures $z \in \Gamma_-$.  

Then, the Cauchy--Riemann operator 
\[
D \colon W^{1,p,\mathbf{\delta}}(\dot S, E) \to L^{p, \mathbf{\delta}}(\dot S,  \Lambda^{0,1} T^*\dot S \tensor E)
\]
with asymptotic operators $\mathbf{A}_z$, $z \in \Gamma \cup \{ \pm \infty \}$
is Fredholm and its index is given by
\[
\Ind(D, \mathbf{\delta}) = n \chi_{\dot S} + 2c_1(E) + \sum_{z \in \Gamma_+} \CZ( \mathbf{A}_z + \delta_z) - \sum_{z \in \Gamma_-} \CZ( \mathbf{A}_z - \delta_z ).
\]
\end{theorem}

This observation about the conjugation of the weighted operator to the
non-degenerate case, combined with Riemann--Roch for punctured domains 
(see for instance, \citelist{
    \cite{SchwarzThesis}*{Theorem 3.3.11} 
    \cite{HWZ3}*{Theorem 2.8}
    \cite{WendlLecturesSFT}*{Theorem 5.4}
}) gives the following.

Now, a useful fact for us is a description of how the Conley--Zehnder index changes as a weight crosses the spectrum of the operator:
\begin{lemma} \label{L:crossingCZ}
Let $\delta > 0$ with $[-\delta, +\delta] \cap \sigma( \mathbf{A}_z ) = \{ 0 \}$. Then, 
\[
\CZ( \mathbf{A}_z - \delta ) - \CZ( \mathbf{A}_z + \delta) = \dim( \ker
\mathbf{A}_z ).
\]
\qedhere
\end{lemma}

For a proof (using the spectral flow idea of \cite{Robbin_Salamon_spectral}), 
see for instance \cite{WendlThesis}*{Proposition 4.5.22}. 

\

To obtain a result that is useful for our moduli spaces of cascades asymptotic to Morse--Bott families of orbits, we consider the following modification of our function spaces.

To each puncture, we associate a subspace of the kernel of the
corresponding asymptotic operator, which we denote by $V_z, z\in \Gamma$,
$V_-, V_+$ and write $\mathbf V$ for this collection. Then, for each
puncture $z \in \Gamma$ and also $\pm \infty$, we associate a smooth
bump function $\mu_z$, $\mu_\pm$, supported near and identically $1$
even nearer to its puncture. We then define
\begin{equation} \label{eqn:weightedSobolevMovingPunctures}
    \begin{split}
W^{1, p, \mathbf \delta}_{\mathbf V} (\dot S,E)
        = &\{ u \in W^{1,p}_\loc (\dot S,E)\, 
		| \, \exists \, c_z \in V_z, z \in \Gamma, c_- \in V_-, c_+ \in V_+ \\
    &\qquad \text{such that } u - \sum c_z \mu_z - c_- \mu_- - c_+ \mu_+ 
		\in W^{1,p, \mathbf \delta}(\dot S,E) \}.%
\end{split}
\end{equation}

We remark that we are using the asymptotic cylindrical coordinates near
$\Gamma$ and the asymptotic trivialization of $E$ in order to define the local
sections $c_z \mu_z$.

In this paper, we are primarily concerned with Cauchy--Riemann operators defined on $\dot S = \R \times S^1$ and on 
$\dot S = \R \times S^1 \setminus \{ P \}$ (a cylinder with one additional negative puncture). 
In the case of $\R \times S^1$, we will write $\mathbf V = (V_-; V_+)$, and in the 
case of $\R \times S^1 \setminus \{ P \}$, we will write $\mathbf V = (V_-, V_P; V_+)$.
(The negative punctures are enumerated first, and separated from the positive 
puncture by a semicolon.)

Observe that since the vector spaces $V_z$ are in the kernel of the
corresponding asymptotic operators, for any choice of $\mathbf{V}$ and 
any vector of weights $\mathbf{\delta}$, we have that the 
Cauchy--Riemann operator $D$ can be extended to 
\begin{equation*}
D \colon W^{1,p,
    \mathbf{\delta}}_{\mathbf V}(\dot S, E) \to 
    L^{p, \mathbf{\delta}}(\dot S, \Lambda^{0,1} T^*\dot S \otimes E).
\end{equation*}

Let $\dim V_z$ denote the dimension of the vector space $V_z$ and let $\codim V_z = \dim \left ( \ker \mathbf{A}_z  / V_z \right )$.
Combining Theorem \ref{T:RiemannRoch} with Lemma \ref{L:crossingCZ}, we have:
\begin{theorem} \label{T:RiemannRochMB}
Let $\delta > 0$ be sufficiently small that 
for $z \in \Gamma_+$,
$[-\delta, 0) \cap \sigma( \mathbf{A}_z) = \emptyset$ and such that for $z \in
\Gamma_-$, $(0, \delta] \cap \sigma(\mathbf{A}_z) = \emptyset$.

For each $z \in \Gamma$, fix the subspace $V_z \subset \ker \mathbf{A}_z$. 

Then, the operator
\[
D \colon W^{1,p, {\delta}}_{\mathbf V}(\dot S, E) \to L^{p, \mathbf{\delta}}(\dot S, \Lambda^{0,1} T^*\dot S \otimes E)
\]
is Fredholm, and its Fredholm index is given by
\[
\Ind(D) = n \chi_{\dot S} + 2c_1(E) + \sum_{z \in \Gamma_+} \big( \CZ( \mathbf{A}_z + \delta) + \dim(V_z)\big)
- \sum_{z \in \Gamma_-} \big(\CZ( \mathbf{A}_z + \delta ) + \codim(V_z)\big).
\]
\end{theorem}

In applications where there are Morse--Bott manifolds of orbits,
we will typically take $V_z$ to be the tangent space to the descending
manifold of a critical point $p_z$ on the manifold of orbits at a
positive puncture, and $V_z$ will be the tangent space to ascending
manifold of a critical point $p_z$ at a negative puncture. In either
case, the contribution to $\Ind(D)$ of $\dim V_z$ or of $\codim
V_z$ will be the Morse index of the appropriate critical point.
This motivates the following definition.

\begin{definition} \label{D:CZ}
    Let $\delta > 0$ be sufficiently small.
If $p_z$ is a critical point of an auxiliary Morse function on the manifold of orbits associated to $z$, then the Conley--Zehnder index of the pair $(\mathbf{A}_z, p_z)$ is 
$$\CZ(\mathbf{A}_z, p_z) = \CZ( \mathbf{A}_z + \delta) + M(p_z),$$ 
where $M(p_z)$ is the Morse index of $p_z$. 
\end{definition}

In this case, we can write the Fredholm index as
\[
\Ind(D) = n \chi_{\dot S} + 2c_1(E) + \sum_{z \in \Gamma_+} \CZ( \mathbf{A}_z,p_z)
- \sum_{z \in \Gamma_-} \CZ( \mathbf{A}_z,p_z).
\]

\

We conclude with a lemma that is particularly useful when applying the automatic transversality result \cite{WendlAutomatic}*{Proposition 4.22}. 
The lemma states that the Fredholm index of an operator with a small negative weight at a puncture is the same as that of the corresponding operator with a small positive weight at that puncture, if the puncture is decorated with the kernel of the corresponding asymptotic operator. 
The former indices are used in \cite{WendlAutomatic}*{Proposition 4.22}, whereas
the latter can be computed using Theorem \ref{T:RiemannRochMB}. Additionally,
the latter arises naturally in the linearization of the non-linear problem.

We first learned this result from Wendl \cite{WendlThesis}. We give a
proof of this formulation since it is slightly stronger than what we
have found in the literature (and is still not as strong as can be proved.)

\begin{lemma} \label{L:exponentialGrowth}
Let $D$ be a Cauchy--Riemann operator. Fix a puncture $z_0 \in \Gamma \cup \{ \pm \infty \}$. 

Let $\mathbf{\delta}$ and $\mathbf{\delta'}$ be vectors of sufficiently small weights so that the differential operator induces a Fredholm operator
 on $W^{1,p,\mathbf{\delta}}$
and on $W^{1,p,\mathbf{\delta'}}$, and $\delta_{z_0} > 0 $ and $\delta_{z_0}' < 0$,
the interval $[\delta_{z_0}', \delta_{z_0}] \cap \sigma(\mathbf{A}_{z_0}) = \{ 0 \}$,
and for each $z \in \Gamma \cup \{ \pm \infty \}$ with $z \ne z_0$, the weights $\delta_z = \delta'_z$.

Let $\mathbf{V}$ be the trivial vector space at each puncture other than $z_0$ and let $V_{z_0}$ be the kernel of the asymptotic operator $\mathbf{A}_{z_0}$ at $z_0$.

Then, the induced operators
\begin{align*}
D_\delta &\colon W^{1,p, \mathbf \delta}_{\mathbf{V}}( \dot S, E)  \to L^{p, \mathbf{\delta}}(\dot S, \Lambda^{0,1} T^*\dot S \otimes E)\\
D_{\delta'} &\colon W^{1,p, \mathbf \delta'}( \dot S, E)  \to L^{p, \mathbf{\delta'}}(\dot S, \Lambda^{0,1} T^*\dot S \otimes E)
\end{align*}
have the same Fredholm index and their kernels and cokernels are isomorphic.
\end{lemma}
\begin{proof}

The main idea of the lemma is contained in \cite{WendlThesis}*{Proposition 4.5.22}, which contains a proof of the equality of Fredholm indices. See also the very closely related \cite{WendlSuperRigid}*{Proposition 3.15}.

Note that $W^{1,p, \mathbf \delta}_{\mathbf{V}}( \dot S, E) $ is a subspace of $W^{1,p, \mathbf \delta'}(\dot S, E)$, and thus the kernel of $D_\delta$ is contained in the kernel of $D_{\delta'}$. 

Now, by a linear version of the analysis done in \citelist{\cite{HWZ3} \cite{SiefringAsymptotics}}, any element of the kernel of $D_{\delta'}$ 
converges exponentially fast at $z_0$ to an eigenfunction of the asymptotic operator, with exponential rate governed by the eigenvalue (in this case 0). Therefore, any element of the kernel of $D_{\delta'}$ must converge exponentially fast to an element of the kernel of the asymptotic operator at $z_0$. 
Hence, the kernel of $D_{\delta'}$ is contained in the kernel of $D_{\delta}$. 

We conclude that the kernels of the two operators may be identified. Since their Fredholm indices are the same, their cokernels are also isomorphic.
\end{proof}

\subsubsection{The linearization at a Floer solution} \label{S:linearization at Floer}

The first step in the proof of Proposition \ref{cascades form manifold} is to
set up the appropriate Fredholm problem.
Given a Floer solution $\tilde v \colon \R \times S^1 \setminus \Gamma \to \R
\times Y$, we consider exponentially weighted Sobolev spaces of sections
of the pull-back bundle $\tilde v^*T( \R \times Y)$ since the asymptotic limits
are (Morse--Bott) degenerate.
For $\delta > 0$, we 
denote by $W^{1,p,\delta}(\R \times S^1 \setminus \Gamma, v^*T(\R \times Y))$
the space of sections that decay exponentially like $\e^{-\delta |s|}$ near the
punctures (also in cylindrical coordinates near the punctures $\Gamma$),
as in the previous section.

We similarly define $W^{m,p,\delta}$ sections with exponential decay/growth.
The following results will not depend on $m$ except in the case of jet
conditions considered in Section \ref{S:transverse contact}, where $m$ will
need to be sufficiently large that the order of contact condition can be
defined.

In order to consider a parametric family of punctured cylinders in which the
asymptotic limits move in their Morse--Bott families, 
we let $\mathbf V$ be a collection of vector spaces, associating to each
puncture $z \in \Gamma \cup \{ \pm \infty \}$ 
a vector subspace $V_z$ of the tangent space to the corresponding Morse--Bott family of
orbits.
For $\delta >0$, we then consider the space of
sections $W^{1,p,\delta}_{\mathbf V}( \R \times S^1 \setminus \Gamma, v^*T(\R
\times Y))$ that converge exponentially at each puncture $z$ to a vector
in the corresponding vector space $V_z$.

\begin{remark}
In this paper, we will not always be careful to specify how small $\delta$
has to be. It is worth pointing out that there is no value of $\delta$
that works for all moduli spaces.  The reason is that we need $|\delta|$
to be smaller than the absolute value of all eigenvalues in the spectra
of the relevant linearized operators.  
Lemma \ref{L:spectrum} computes the spectrum of a number of these relevant
asymptotic operators, and as we see in 
Table \ref{evals C}, the 
smallest positive eigenvalue $\frac{1}{2}\left(-C + \sqrt{C^2 + 16
\pi^2} \right)$ becomes arbitrarily small as $C \to \infty$. As will become
clear from Lemma \ref{L:linearizationFloer} and Equation \eqref{E:DC}, the
relevant value for $C$ here is $h''(e^{b_k}) e^{b_k}$, which can become
arbitrarily large as the multiplicity $k \to \infty$.  
Since the relevant moduli spaces in the differential involve connecting orbits
of bounded multiplicities, for any given moduli space, we may choose $\delta$
sufficiently small.
\label{delta small}
\end{remark}

We now adapt an observation first used in \citelist{
\cite{DragnevTransversality} \cite{BourgeoisHomotopyContact}}, 
to show that the linearization of the Floer operator is upper triangular with
respect to the splitting of $T(\R \times Y)$ as $(\R \oplus \R R) \oplus \xi$.
We then describe the non-zero blocks in this upper triangular presentation of the
operator. The two diagonal terms are of special importance: one will be a
Cauchy--Riemann operator acting on sections of a complex line bundle, and
the other can be identified with the linearization of the Cauchy--Riemann
operator for spheres in $\Sigma$. 

We now explain this construction in more detail.
Let $\tilde v \colon \R \times S^1 \setminus \Gamma \to \R \times Y$ be a Floer
solution with punctures $\Gamma$. The Hamiltonian need not be admissible, but
needs to be radial (i.e.~depending only on $r$, the symplectization variable).
The almost complex structure $J_Y$ is assumed to be admissible.
We consider three possible cases for the asymptotics of such a curve.

In the first case, $\tilde v$ is asymptotic to a closed Hamiltonian orbit at 
$\tilde v(+\infty, t)$, to a closed Hamiltonian orbit at $\tilde v(-\infty,
t)$, and with negative ends converging to Reeb orbits at the punctures in
$\Gamma$. The second case has $\tilde v$ asymptotic to a closed
Hamiltonian orbit at $\tilde v(+\infty, t)$, but with negative ends converging
to Reeb orbits in $\{-\infty \} \times Y$ at $\{ -\infty \} \cup \Gamma$.
These two cases correspond to an upper level of a split Floer cylinder as in
Definitions \ref{stretched Y to Y} and \ref{stretched Y to W}, respectively.

The third case we consider is most directly applicable to studying holomorphic
curves in $\R \times Y$:
$\tilde v$ has a positive cylindrical end at $+\infty$ converging to a Reeb orbit 
in $\{+\infty \} \times Y$, and has
negative cylindrical ends at the punctures $\{ -\infty \} \cup \Gamma$.
For such a curve, we may assume that $H$ is identically $0$, and thus this
example includes $J_Y$--holomorphic curves.  
This is of independent interest, and is useful in
\cite{DiogoLisiComplements}. Part of this was sketched in \cite{EGH}*{Section 2.9.2}.

Let $w = \pi_\Sigma \circ \tilde v \colon \CP^1 \to \Sigma$ be the smooth extension of the projection
of $\tilde v$ to the divisor (as given by Lemma \ref{planes = spheres}).
The linearized projection $d\pi_\Sigma$ induces an isomorphism of complex vector bundles
\[
\tilde v^*\big(T (\R \times Y) \big) \cong ( \R \oplus \R R) \oplus w^*T\Sigma.
\]
To see this, note that for each point
$p \in Y$, $d\pi_\Sigma$ induces a symplectic isomorphism $(\xi_p, d\alpha) \cong
(T_{\pi_\Sigma(p)} \Sigma, K\omega_\Sigma)$. By the Reeb invariance of the almost
complex structure (and thus $S^1$-invariance under rotation in the fibre), 
this then gives a complex vector bundle isomorphism.

Let $\mathbf{V}$ associate to each puncture $z \in \Gamma \cup \{ \pm \infty \}$
the tangent space to $Y$ if the corresponding limit of $\tilde v$ is a
closed Hamiltonian orbit and the tangent space to $\R \times Y$ if the
corresponding limit of $\tilde v$ is a closed Reeb orbit. 
As will be clearer shortly, this is associating to each puncture the entirety
of the kernel of the corresponding asymptotic operator.

Let
\begin{equation} \label{E:FloerLinearization}
D_{\tilde v} \colon W^{1,p,\delta}_\mathbf{V} ( \tilde v^* T(\R \times Y) )\to L^{p,\delta}(
\Hom^{0,1}(T(\R \times S^1 \setminus \Gamma), \tilde v^*T(\R \times Y) ))
\end{equation}
be the linearization of the nonlinear Floer operator at the solution $\tilde v$,
for $\delta > 0$ sufficiently small. The vector spaces $\mathbf{V}$ correspond to
allowing the asymptotic limits to move in their Morse--Bott families. 
We have then a linearized evaluation map at the punctures with values in
$\oplus_{z \in \{ \pm \infty \} \cup \Gamma} V_z$.  
Let 
\[
D^\Sigma_w \colon W^{1,p}( w^*T \Sigma) \to L^{p}( \Hom^{0,1}(T\CP^1, w^*T\Sigma))
\]
be the linearized Cauchy--Riemann operator in $\Sigma$ at the holomorphic sphere $w$.
We also have the linearized Cauchy--Riemann operator
$\dot D^\Sigma_w$ at
the holomorphic cylinder $s+it \mapsto w( \e^{2\pi(s+it)} ) = \pi_\Sigma( \tilde
v(s,t) )$. Then, %
$(\pi_\Sigma\circ\tilde v)^*T\Sigma = w^*T\Sigma|_{\R \times S^1 \setminus \Gamma}$
is a Hermitian vector bundle over $\R \times S^1 \setminus \Gamma$.
Let $\mathbf{V_\Sigma}$ be the kernels of the asymptotic operators of 
$\dot D^\Sigma_w$ 
at each of the punctures, $\pm \infty$ and $\Gamma$.
(These are explicitly given by 
$V_\Sigma(- \infty) = T_{w(0)} \Sigma$, $V_\Sigma(+\infty) = T_{w(\infty)}
\Sigma$, $V_\Sigma(z) = T_{w(z)} \Sigma$ for each marked point $z \in \Gamma$.)
We consider this operator acting on the space of sections
\[
    \dot D^\Sigma_w 
       \colon W^{1,p, \delta}_{\mathbf{V_\Sigma}}(w^*T\Sigma|_{\R \times S^1 \setminus \Gamma}) 
       \to L^{p, \delta}( 
       \Hom^{0,1}(T(\R \times S^1 \setminus \Gamma), w^*T\Sigma|_{\R \times S^1 \setminus
                \Gamma}
                )
                ).
\]
The operator $D^\Sigma_w$ is Fredholm independently of the weight, but $\dot
D^\Sigma_w$ is only Fredholm when the weight $\delta$ is not an integer multiple of
$2\pi$. Furthermore, by combining \cite{WendlSuperRigid}*{Proposition 3.15}
with Lemma \ref{L:exponentialGrowth}, for $0 <
\delta < 2\pi$, these operators have the same Fredholm index and their kernels
and cokernels are isomorphic by the map induced by restricting a section of $w^*T\Sigma$ to
the punctured cylinder. 

Finally, define $D^\C_{\tilde v}$ by
\begin{equation} \label{E:DC}
\begin{aligned}
    D^\C_{\tilde v} \colon W^{1,p,\delta}_{\mathbf{V_0}}(\R \times S^1 \setminus \Gamma, \C) &\to L^{p,\delta}(
\Hom^{0,1}(T (\R \times S^1 \setminus \Gamma), \C) )\\
	 (D^\C_{\tilde{v}} F)(\partial_s) &= F_s + iF_t  + \begin{pmatrix} h''(\e^b) \e^b & 0 \\0 & 0 \end{pmatrix}F
\end{aligned}
\end{equation}
where $\mathbf{V_0}$ associates 
the vector space $i\R$ to the punctures at which $\tilde v$ converges to a
closed Hamiltonian orbit 
and associates the vector space $\C$ at punctures at which $\tilde v$ converges
to a closed Reeb orbit.
Notice that again these are chosen so that they precisely give the kernels of the
corresponding asymptotic operators of $D^\C_{\tilde v}$.

\begin{lemma} \label{L:linearizationFloer}
The isomorphism $\tilde v^*T( \R \times Y) \cong (\R \oplus \R R) \oplus w^*T\Sigma$
induces a decomposition:
\[
D_{\tilde v} = \begin{pmatrix} 
	D^\C_{ \tilde v} & M \\
	0 & \dot D^\Sigma_{ w}
	\end{pmatrix}
\]
where $M$ is a multiplication operator 
that evaluates on $\partial_s$ to a 
fibrewise 
linear map $M \colon w^*T\Sigma \to \R \oplus \R R$, decaying at the
punctures. (In particular, $M$ is
compact.) Furthermore, if $w = \pi_\Sigma \circ \tilde v$ is non-constant, 
then $M$ is pointwise surjective except at finitely many points.
 \end{lemma}

 \begin{proof}
In our setting, the nonlinear Floer operator takes the form of the left-hand
side of the equation:
\begin{equation*}
    d \tilde v + J_Y(\tilde v) d\tilde v \circ i - h'(\e^r) R \otimes dt + h'(\e^r) \partial_r \otimes ds = 0.
\end{equation*}

Write $\tilde v = (b,v) \colon \R\times S^1 \setminus \Gamma \to \R \times Y$. If we apply $dr$ to the previous equation, and use the fact that $dr \circ J_Y = - \alpha$, we get:
\[
db - v^*\alpha \circ i + h'(\e^b) ds = 0.
\]

Denoting by $\pi_\xi \colon TY \to \xi$ the projection along the Reeb vector field, we get
\begin{equation} \label{apply dr}
\pi_\xi d\tilde v + J_Y(\tilde v) \pi_\xi d\tilde v \circ i = 0.
\end{equation}

Let $g$ be the metric on $\R\times Y$ given by 
$g = dr^2 + \alpha^2 + d\alpha( \cdot, J_Y \cdot)$. 
This metric is $J_Y$-invariant. 
Let $\Nabla$ be the Levi-Civita connection for $g$. Let $\nabla$
be the Levi-Civita connection on $T \Sigma$ for the metric $\omega_\Sigma(\cdot,
J_\Sigma \cdot)$.

Then, it follows that the linearization $D_{\tilde v}$ applied to a section $\zeta$ of $\tilde v^* T(\R \times Y)$ satisfies
\begin{equation} \label{E:linearizedFloerOperator}
\begin{aligned}
D_{\tilde v} \zeta \left(\partial_s \right) 
	&= \Nabla_s\zeta + J_Y(\tilde v) \Nabla_t \zeta + \big(\Nabla_\zeta J_Y(\tilde v)\big) \partial_t \tilde v - \Nabla_\zeta (J_Y X_H)(\tilde v)\\
	&= \Nabla_s \zeta + J_Y(\tilde v) \Nabla_t \zeta + \big( \Nabla_\zeta J_Y(\tilde v) \big ) \partial_t \tilde v + \Nabla_\zeta( h'(\e^r) \partial_r )_{r = b}.
\end{aligned}
\end{equation}

Notice that $\Nabla \partial_r = 0$ since $g$ is a product metric. We have then
$$\Nabla_\zeta( h'(\e^r) \partial_r)|_{r=b} = h''(\e^b) \e^b dr(\zeta) \partial_r.$$

Observe also that for any vector field $V$ in $T\Sigma$, there is a 
unique horizontal
lift $\tilde V$ to $Y$ with the property $\alpha( \tilde V) =0$.
For any two vector fields $V$ and $W$ in $T\Sigma$,
since $d\alpha( \tilde V, \tilde W) = K\omega_\Sigma(V, W)$,
we have the following 
\[
    [\tilde V, \tilde W] = \widetilde{[V, W]} - K\omega_\Sigma(V, W) R.
\]

From this, it follows that the Levi-Civita connection $\Nabla$ satisfies the following
identities:
\begin{align*}
\Nabla_{\tilde V} \tilde W &= \widetilde{ \nabla_V W} - \frac{K}{2} \omega_\Sigma( V, W) R \\
\Nabla_{R}R &= 0 \\
\Nabla_{R} \tilde V &= - \frac{1}{2} J_Y \tilde V. 
\end{align*}

A simple computation using the Reeb-flow invariance of $J_Y$ and the torsion-free
property of the connection gives
\[
    \Nabla_{\partial_r} J_Y = 0 = \Nabla_{R} J_Y.
\]

We will now compute $D_{\tilde v} \zeta \left (\partial s \right)$, first when
$\zeta = \zeta_1 \partial_r + \zeta_2 R = (\zeta_1 + i \zeta_2) \partial_r$, and
then when $\zeta$ is a section of $\tilde v^*\xi$.

For the first computation, it suffices to notice the following two identities
\begin{align*}
D_{\tilde v} \partial_r \left( \partial_s \right ) 
    &= h''(\e^b)\e^b \partial_r \\
D_{\tilde v} R \left( \partial_s \right ) 
    &= 0.
\end{align*}
It follows then from the Leibniz rule that we have
\[
D_{\tilde v} (\zeta_1 + i \zeta_2) \,\partial_r \left( \partial_s \right ) 
= 
\left (\zeta_s + i \zeta_t\right) + \left(h''(e^b)e^b \zeta_1 \right )\partial_r = D^\C_{\tilde v}(\zeta_1 + i \zeta_2)\,\partial_r(\partial_s).
\]

Now consider the case when $\zeta$ 
is a section of $\tilde v^*\xi$, and is thus the lift $\zeta = \tilde \eta$ of a section $\eta$ of $w^*T \Sigma$.
We compute 
\begin{align*}
\Nabla_s R &= \Nabla_{\pi_\xi v_s} R = -\frac{1}{2} J_Y \pi_\xi v_s \\
\Nabla_s \zeta 
  &= \widetilde{ \nabla_{w_s} \eta } - \frac{K}{2} \omega_\Sigma(w_s, \eta)R
      -\frac{1}{2} \alpha(v_s) J_Y \zeta,
\end{align*}
and similarly for $\Nabla_t$.
We then obtain the following covariant derivatives of $J_Y$, where $\tilde W$ is a
section of $\tilde v^*\xi$:
\begin{align*}
(\Nabla_{\zeta}J_Y) \partial_r &= \Nabla_\zeta R - J_Y \Nabla_\zeta \partial_r = -\frac{1}{2} J_Y \zeta \\
(\Nabla_\zeta J_Y)R &= - \Nabla_\zeta \partial_r - J_Y \Nabla_\zeta R = - \frac{1}{2} \zeta \\
(\Nabla_\zeta J_Y)\tilde W &= \Nabla_\zeta (J_Y \tilde W) - J_Y \Nabla_\zeta \tilde W \\
    &= \widetilde{ \nabla_{\eta} J_\Sigma W } 
    - \frac{K}{2} \omega_\Sigma(\eta, J_\Sigma W) R 
    - J_Y \left ( \widetilde{ \nabla_{\eta} W} 
    - \frac{K}{2} \omega_\Sigma(\eta, W)R \right ) \\
    &= \widetilde{ (\nabla_{\eta} J_\Sigma)W } - \frac{K}{2} \omega_\Sigma( \eta, J_\Sigma W) R -
    \frac{K}{2} \omega_\Sigma(\eta, W) \partial_r.
\end{align*}
It follows then
\begin{align*}
D_{\tilde v} \zeta \left( \partial_s \right ) 
    &= \Nabla_s \zeta + J_Y \Nabla_t \zeta + (\Nabla_\zeta J_Y) \tilde v_t \\
    \begin{split}
    &= 
    \widetilde{ \nabla_s \eta} - \frac{1}{2} \alpha(v_s) J_Y \zeta - \frac{K}{2} \omega_\Sigma(w_s, \eta) R 
    + J_Y \widetilde{ \nabla_t \eta } + \frac{1}{2} \alpha(v_t) \zeta +
    \frac{K}{2} \omega_\Sigma( w_t, \eta) \partial_r  \\
    &\qquad - \frac{1}{2} b_t J_Y \zeta - \frac{1}{2} \alpha(v_t) \zeta + \widetilde{ (\nabla_\eta J_\Sigma) w_t } - \frac{K}{2} \omega_\Sigma( \eta, J_\Sigma
    w_t) R - \frac{K}{2} \omega_\Sigma(\eta, w_t) \partial_r \\
\end{split} \\
&= \widetilde{ {\dot {D^\Sigma_w}} \eta  } + K\omega_\Sigma(w_t, \eta) \partial_r -
    K\omega_\Sigma(w_s, \eta) R.
\end{align*}
(Note that we use the fact that $\tilde v_s + J_Y \tilde v_t + h'(\e^b)
\partial_r = 0$ in the cancellations.)

Writing $\zeta = (\zeta_a,\zeta_b)$ under the isomorphism $\tilde v^*T( \R \times Y) \cong (\R \oplus \R R) \oplus w^*T\Sigma$,
we obtain the decomposition:
\[
D_{\tilde v} (\zeta_a, \zeta_b) (\partial_s) = \begin{pmatrix} 
	D_{aa} & D_{ab} \\
	D_{ba} & D_{bb}
	\end{pmatrix}
        \begin{pmatrix} \zeta_a \\ \zeta_b \end{pmatrix} (\partial_s).
\]

Our calculations now establish that 
$D_{aa} = D^\C_{\tilde v}$ and $D_{ba} = 0$, 
$D_{bb} = \dot D^\Sigma_w$, and 
$D_{ab} \zeta (\partial_s) = K\omega_\Sigma(w_t, \pi_\Sigma\zeta) \partial_r -  K\omega_\Sigma(w_s, \pi_\Sigma\zeta) R$. 
Observe that in particular, $D_{ab}$ is a pointwise linear map from $\tilde
v^*\xi|_p$ to $\R \partial_r \oplus \R R$. The map is surjective except at
critical points of the pseudoholomorphic map $w$, of which there are finitely many 
if $w$ is non-constant.
The decay claim follows since $w$ converges to a point, and thus its
derivatives decay exponentially fast.

\end{proof}

\begin{remark}
 Notice that for each puncture $z \in \{ \pm \infty \} \cup \Gamma$, 
if $\gamma(t)$ denotes the corresponding asymptotic Hamiltonian or Reeb orbit, 
the previous result allows us to identify $V_z$ with $T_{\gamma(0)} Y$ at a Hamiltonian orbit 
and with $\R \times T_{\gamma(0)} Y$ at a Reeb orbit.
\end{remark}

\begin{lemma} \label{L:verticalOperatorTransverse}
    Let $\tilde v \colon \R \times S^1 \setminus \Gamma \to \R \times Y$ be a
    finite hybrid energy Floer cylinder with punctures $\Gamma$.

    Then, the operator $D^\C_{\tilde v}$ defined in Equation \eqref{E:DC}
is Fredholm for $\delta > 0$ sufficiently small. 

The restriction \[
    D^\C_{\tilde v}|_{W^{1,p, \delta}} \colon W^{1,p,\delta}(\R \times S^1
\setminus \Gamma, \C) \to L^{p,\delta}(\Hom^{0,1}(T(\R \times S^1 \setminus
\Gamma), \C))\]
has Fredholm index 
$-1-2\#\Gamma$ if the positive puncture at $+\infty$ converges to a closed Hamiltonian
orbit and has Fredholm index $-2 -2\#\Gamma$ if the positive puncture converges
to Reeb orbit at ${+\infty} \times Y$.

If $\tilde v$ converges at both $\pm \infty$ to closed Hamiltonian orbits, then
\[
    D^\C_{\tilde v} \colon W^{1,p,\delta}_{\mathbf{V_0}}(\R \times S^1 \setminus \Gamma, \C) \to L^{p,\delta}(
\Hom^{0,1}(T (\R \times S^1 \setminus \Gamma), \C) )
\]
has Fredholm index $1$ and is surjective.

If, instead, $\tilde v$ converges at $+\infty$ to a closed Hamiltonian orbit,
and at $-\infty$ to a closed Reeb orbit in
$\{-\infty \} \times Y$, then 
$D^\C_{\tilde v}$ has Fredholm index $2$ and is
surjective.

Finally, if $\tilde v$ converges at $\pm \infty$ to closed Reeb orbits in $\{
\pm \infty \} \times Y$, then 
$D^\C_{\tilde v}$ has Fredholm index $2$ and is
surjective.

In all three cases, the kernel of $D^\C_{\tilde v}$ contains the
constant section $i$, which can be identified with the Reeb vector field.

\end{lemma}

\begin{proof}
We will apply the punctured Riemann--Roch Theorems \ref{T:RiemannRoch}
and \ref{T:RiemannRochMB}. 
For this, we need to compute the Conley--Zehnder indices of the
appropriately perturbed asymptotic operators.
We will first identify the (Morse--Bott degenerate) asymptotic operators at each of the 
punctures, and then apply Corollary \ref{C:CZ computation} to obtain the
Conley--Zehnder indices of the $\pm\delta$-perturbed operators.

Recall from Remark \ref{delta small} that we have $|\delta| > 0$ smaller than
the spectral gap for any of these punctures. 

In order to consider the operator $D^\C_{\tilde v} \colon W^{1,p,
\delta}_{\mathbf{V}_0} \to L^{p,\delta}$, it will be convenient to consider a
related operator with the same formula, but on the much larger space of
functions with exponential \textit{growth}. By a slight abuse of notation, we
will use the same name:
\begin{align*}
    &D^\C_{\tilde v} \colon W^{1,p, -\delta} \to L^{p, -\delta} \\
    &(D^\C_{\tilde{v}} F)(\partial_s) = F_s + iF_t  + \begin{pmatrix} h''(\e^b)
\e^b & 0 \\0 & 0 \end{pmatrix}F.
\end{align*}
Then, the kernel and cokernel of the operator acting on the spaces of sections
with exponential growth
can be identified with the kernel and cokernel of the operator acting on
$W^{1,p,\delta}_{\mathbf{V}_0}$, by Lemma \ref{L:exponentialGrowth}.

First, consider the case when $\tilde v$ converges to a closed Hamiltonian
orbit in $\{ b_\pm \} \times Y$ as $s \to \pm \infty$.
Then, the asymptotic operator associated to $D^\C_{\tilde v}$ at $\pm\infty$ is given by 
    \[
        \mathbf{A}_\pm = -i \frac{d}{dt} - \begin{pmatrix} h''(\e^{b_\pm}) \e^{b_\pm} & 0 \\0 & 0
        \end{pmatrix} . %
    \] 
In the case of $\delta$--exponential decay, the relevant asymptotic operators are
given by $\mathbf{A}_+ +\delta$ at the positive puncture $+\infty$ and by $\mathbf{A}_- - \delta$ at the
negative puncture $-\infty$. In the case of $\delta$--exponential growth, the
relevant asymptotic operators are $\mathbf{A}_+ - \delta$ and $\mathbf{A}_- + \delta$,
respectively.

For the case of exponential decay, Corollary \ref{C:CZ computation} then gives the Conley--Zehnder 
index of $0$ for $\mathbf{A}_+ + \delta$ and of $1$ for $\mathbf{A}_- - \delta$.

In the case of exponential growth, Corollary \ref{C:CZ computation} gives instead that the Conley--Zehnder 
index of $\mathbf{A}_+ - \delta$ is 1 and that of $\mathbf{A}_- + \delta$ is 0. 

    Associated to a Reeb puncture at $\pm \infty$ or at $P\in \Gamma$, we have the asymptotic operator
    \[
        -i \frac{d}{dt}.
    \]
    Writing $\tilde v = (b,v) \colon \R\times S^1 \setminus \Gamma \to \R\times Y$, we have $b\to -\infty$ at both types of negative punctures and $b \to +\infty$ at the positive puncture. 

    As above, in the case of exponential decay, the relevant
    asymptotic operators are $-i\frac{d}{dt} +\delta$ at a positive puncture
    and $-i\frac{d}{dt} -\delta$ at a negative puncture. Again, by Corollary
    \ref{C:CZ computation}, we obtain a Conley--Zehnder index of $-1$ at
    $+\infty$ and a Conley--Zehnder indices of $1$ at a negative puncture
    ($-\infty$ or $P \in \Gamma$).

    If, instead, we consider exponential growth, we obtain Conley--Zehnder
    indices of $+1$ at positive punctures and $-1$ at negative punctures.

Applying now the punctured Riemann--Roch theorem \ref{T:RiemannRoch}, and using the fact that the
    Euler characteristic of the punctured cylinder is $-\# \Gamma$, 
    we obtain that the Fredholm index of 
\[
    D^\C_{\tilde v}|_{W^{1,p, \delta}} \colon W^{1,p,\delta}(\R \times S^1
\setminus \Gamma, \C) \to L^{p,\delta}(\Hom^{0,1}(T(\R \times S^1 \setminus
\Gamma), \C))\]
is given by 
    \[
        - \# \Gamma  - c - 1 - \#\Gamma = -c -1 - 2\# \Gamma,
    \]
    where $c = 0$ if the positive puncture converges to a Hamiltonian orbit,
    and $c=1$ if the positive puncture converges to a Reeb orbit at $+\infty$,
    as claimed.

    The injectivity of $D^\C_{\tilde v}$ restricted to $W^{1,p,\delta}$ follows from automatic transversality, applying \cite{WendlAutomatic}*{Proposition 2.2}. 
The criterion involves the adjusted Chern number 
\cite{WendlAutomatic}*{Equations (2.4) and (2.5)}. 
In our situation, there are $1-c$ punctures with even Conley--Zehnder index. 
This adjusted Chern number then becomes
\begin{align*}
    c_1(E,l,\textbf{A}_\Gamma) 
        &= \frac{1}{2}( \Ind(D^C_{\tilde v}|_{W^{1,p,\delta}}) -2 +\# \Gamma_0 ) \\
        &= \frac{1}{2}( -c-1-2\#\Gamma-2+(1-c) ) = -\# \Gamma -1-c < 0.
\end{align*}
as necessary to apply \cite{WendlAutomatic}*{Proposition 2.2}.

Now, applying 
Theorem \ref{T:RiemannRochMB}, we compute that the Fredholm index 
of 
\[
    D^\C_{\tilde v} \colon W^{1,p,\delta}_{\mathbf{V_0}}(\R \times S^1 \setminus \Gamma, \C) \to L^{p,\delta}(
\Hom^{0,1}(T (\R \times S^1 \setminus \Gamma), \C) )
\]
is given by 
    \[
    \begin{split}
    - \# \Gamma + &1 - (-\# \Gamma) - 
    	\begin{cases} 0 & \text{if $\tilde v(-\infty)$ converges to a Hamiltonian orbit}\\
	 -1 & \text{if $\tilde v(-\infty)$ converges to a Reeb orbit}    
	 	 \end{cases}	 \\
	 &= 1 \text{ or } 2, \text{ depending on the negative end of $\tilde v$.}
	 \end{split}
    \]
   
   Furthermore, the fact that the curve has genus $0$ and one puncture
   with even Conley--Zehnder index precisely if $\lim_{s\to -\infty} \tilde v$ is a Hamiltonian orbit implies that 
$$
c_1(E,l,\textbf{A}_\Gamma) = \begin{cases}  \frac{1}{2}(1-2+1) = 0   & \text{if $\tilde v(-\infty)$ converges to a Hamiltonian orbit}\\
	 \frac{1}{2}(2-2) = 0  & \text{if $\tilde v(-\infty)$ converges to a Reeb orbit}    
	 	 \end{cases}	 \\ 
$$   
In either case, the adjusted Chern number is less than the Fredholm index.
Therefore, $D^\C_{\tilde v}$ satisfies the automatic transversality criterion and is thus surjective, as wanted.

It follows immediately from the expression for $D^\C_{\tilde v}$ that the 
constant $i$ is in the kernel. Recalling that  
$\C = \tilde v^*(\R \oplus {\R}R)$ in the splitting
given by Lemma \ref{L:linearizationFloer}, we then may identify this constant with 
the Reeb vector field $R$.

\end{proof}

To summarize the results of this section, 
by Lemma \ref{L:linearizationFloer}, a punctured Floer cylinder in $\R \times
S^1$ is regular if the operators $D^\C_{\tilde v}$ and $\dot D^\Sigma_w$ are surjective.
Surjectivity of the latter is equivalent to surjectivity of $D^\Sigma_w$.
Lemma \ref{L:verticalOperatorTransverse} gives the surjectivity of $D^\C_{\tilde
v}$. 
It thus remains to study transversality for $D^\Sigma_w$, specifically with
respect to the evaluation maps that will allow us to define the moduli spaces of
chains of pearls in $\Sigma$ (see Section \ref{S:transverse pearls}). 
Additionally, we need to consider transversality for moduli spaces of planes in
$W$ asymptotic to Reeb orbits in $Y$, or equivalently, the moduli spaces of 
spheres in $X$ with an order of contact condition at $\Sigma$ (see Section \ref{S:transverse contact}). 

\subsection{Transversality for chains of pearls in \texorpdfstring{$\Sigma$}{Sigma}} 

\label{S:transverse pearls}

In this section and the next, we show that for generic almost complex structure (in a sense
to be made precise), the moduli spaces of chains of pearls and moduli spaces of
chains of pearls with spheres in $X$ (possibly augmented as well) are
transverse. We begin with the definition of several moduli spaces that will be useful.  

\begin{definition} \label{D:moduliSpacesPearls}
Let $J_\Sigma \in \mathcal{J}_\Sigma$ be an almost complex structure
compatible with $\omega_\Sigma$. Given $p,q \in \Crit(f_\Sigma)$ and a
finite collection $A_1,\ldots,A_N\in H_2(\Sigma;\Z)$,
let $${\M^*_{k, \Sigma}}((A_1,\ldots,A_N); q,p;J_\Sigma)$$ denote the space of simple
chains of pearls in $\Sigma$ from $q$ to $p$ (see Definition \ref{D:simple
    chain pearls}), such that $(w_i)_*[\CP^1] = A_i$, with $k$ marked points.

Let $${\M^*_{k, \Sigma}}( (A_1, \ldots, A_N); J_\Sigma )$$ denote the moduli space of $N$
parametrized $J_\Sigma$-holomorphic spheres in $\Sigma$, representing the classes $A_i, i=1, \dots, N$,
with $k$ marked points, also satisfying the simplicity criterion of Definition
\ref{D:simple chain pearls}, i.e.~so each sphere is either somewhere injective
or constant, each constant sphere has at least one augmentation marked point, and no
sphere has image contained in the image of another.

For $J_W \in \J_W$, let $J_\Sigma = P(J_W)$ be the
corresponding almost complex structure in $\J_\Sigma$ and $J_X$ the
corresponding almost complex structure on $X$. Define 
\[
    {\M^*_{k, (X, \Sigma)}}( (B; A_1, \dots, A_N); x, p, J_W)
\]
to be the moduli
space of simple chains of pearls in $\Sigma$ with a sphere in $X$ (as in
Definitions \ref{D:chain of pearls w sphere in X} and \ref{D:simple chain pearls}),
where $x$ is a critical point of $f_W$ and $p$ is a critical point of
$f_\Sigma$, and representing the spherical homology classes $[w_i] =
A_i \in H_2(\Sigma;\Z)$, $i=1, \dots, N$ and $[v] = B \in H_2(X;\Z)
\setminus 0$. 
In the following, we will write \[
l = B \bullet \Sigma = K \omega(B)
\]
which is the order of contact of $v$ with $\Sigma$.

Let \[
    {\M^*_{k, (X, \Sigma)}}( (B; A_1, \dots, A_N); J_W)
\]
denote the moduli
space of $N$ parametrized $J_\Sigma$-holomorphic spheres in $\Sigma$, representing the
classes $A_i$, and of a $J_X$-holomorphic sphere in $X$ representing the class $B$
with order of contact $l = B \bullet \Sigma = K \omega(B)$, also satisfying the 
simplicity criterion of Definition \ref{D:simple chain pearls}, i.e.~so each
sphere in $\Sigma$ is either somewhere injective or constant (if constant, it
has at least one augmentation marked point), 
no image of a sphere in $\Sigma$ is contained in the image of
another and the image of the sphere in $X$ is not contained in the tubular
neighbourhood $\varphi(\U)$ of $\Sigma$. Furthermore, the spheres in $\Sigma$ have $k$ marked
points.

Let 
\[
    \M^*_{X}((B_1, B_2, \dots, B_k); J_W)
\]
denote the moduli space of
$k$ {\em unparametrized} $J_X$-holomorphic spheres in $X$, where each sphere is somewhere
injective, no image of a sphere is contained in the image of another sphere,
and 
so the image of each sphere is not contained in the tubular neighbourhood
$\varphi(\U)$ of
$\Sigma$, and  such that each sphere intersects $\Sigma$ 
only at $\infty \in \CP^1$ with order of contact $B_i \bullet \Sigma$. 
We can think of an unparametrized sphere as an equivalence class of parametrized spheres,
modulo the action of $\Aut(\CP^1,\infty) = \Aut(\C)$ on the domain. 

Finally, let 
\[
    \M^a_{k,\Sigma}((A_1, \dots, A_N), (B_1, \dots, B_k);q, p;J_W) 
\]
denote the moduli space of simple augmented chains of pearls in $\Sigma$ with
$k$ unparamentrized augmentation planes,
and let 
\[
    \M^a_{k,(X, \Sigma)}((B; A_1, \dots, A_N); (B_1, \dots, B_k); x, p; J_W)
\]
denote the moduli space of simple augmented chains of pearls with a sphere in $X$. 
(See 
Definitions \ref{D:aug chain of pearls} and \ref{D:simple chain pearls}.)
\end{definition}

In order to apply the Sard--Smale Theorem, we need to consider Banach spaces of
almost complex structures, so we let $\J^r_\Sigma, \J^r_W$ be the space of 
$C^r$--regular almost complex structures otherwise satisfying the conditions of
being in $\J_\Sigma$, $\J_W$. We impose $r \ge 2$ and in general will require
$r$ to be sufficiently large that the Sard--Smale theorem holds (this will
depend on the Fredholm indices associated to the collection of homology classes
and will also depend on the order of contact to $\Sigma$ for the spheres in $X$).

For each of these moduli spaces, we also consider the corresponding
universal moduli spaces as we vary the almost complex structure.
For instance, we denote by 
$\M_{k, \Sigma}^*((A_1, \dots, A_N), \J^r_\Sigma)$ the moduli space of pairs
$((w_i)_{i=1}^N, J_\Sigma)$ with $J_\Sigma \in \J^r_\Sigma$ and 
$(w_i)_{i=1}^N \in \M_{k, \Sigma}^*( (A_1, \dots, A_N), J_\Sigma)$.

The main goal of this section and of the next is to prove that these moduli spaces of simple
chains of pearls are transverse for generic almost complex structures. This is analogous to
\cite{McDuffSalamon}*{Theorem 6.2.6}, and indeed, the transversality theorem of
McDuff--Salamon will be a key ingredient of our proof. Their Theorem 6.2.6 is about transversality 
of the universal evaluation map to a specific submanifold $\Delta^E$ of the target, whereas our work in this 
section establishes transversality to some other submanifolds.
We will furthermore require an extension of the results from 
\cite{CieliebakMohnkeTransversality} (see Section \ref{S:transverse contact}), 
and an additional technical transversality
point needed to be able to consider the lifted problem in $\R \times Y$.  

\begin{proposition} \label{necklaces are regular}
There is a residual set 
$\mathcal J_W^{reg}\subset \J_W$ such that $\mathcal J_\Sigma^{reg} \coloneqq P( \mathcal J_W^{reg})$
is a residual set in $\mathcal J_\Sigma$ and such that 
 for all $J_\Sigma \in \mathcal J_\Sigma^{reg}$ and $J_W \in \mathcal J_W^{reg}$,
 $p \in \Crit(f_\Sigma)$, $q \in \Crit(f_\Sigma)$ and $x \in \Crit(f_W)$, 
the moduli spaces 
    ${\M^*_{k, \Sigma}}((A_1,\ldots,A_N); q,p;J_\Sigma)$,
    ${\M^*_{k, (X, \Sigma)}}((B; A_1, \dots, A_N); x, p, J_W)$,
    $\M^a_{k,\Sigma}((A_1, \dots, A_N), (B_1, \dots, B_k);q, p; J_W)$  { and }
    $\M^a_{k, (X, \Sigma)}((B; A_1, \dots, A_N); (B_1, \dots, B_k); x, p; J_W)$
are manifolds. Their dimensions are 
\begin{align*}
    \dim&~{\M^*_{k, \Sigma}}((A_1,\ldots,A_N); q,p;J_\Sigma) =~
        M(p) + \sum_{i=1}^N 2 \, \langle c_1(T\Sigma), A_i\rangle - M(q) + N-1 + 2k, \\
    \dim&~{\M^*_{k, (X, \Sigma)}}((B; A_1, \dots, A_N); x, p, J_W) \\
       =&~ M(p) + \sum_{i=1}^N 2 \, \langle c_1(T\Sigma), A_i \rangle + 2 
                (\langle c_1(TX), B \rangle - B \bullet \Sigma) + M(x) - 2(n-1)
                + N-1 + 2k \\
    \dim&~\M^a_{k,\Sigma}((A_1, \dots, A_N), (B_1, \dots, B_k);q, p; J_W) \\
    =&~ 
        M(p) + \sum_{i=1}^N 2 \, \langle c_1(T\Sigma), A_i\rangle - M(q) + N-1 + 
         \sum_{i=1}^k \left ( 2 \, \langle c_1(TX), B_i \rangle - 2 B_i
        \bullet \Sigma \right ), \\
    \dim&~\M^a_{k, (X, \Sigma)}((B; A_1, \dots, A_N); (B_1, \dots, B_k); x, p; J_W)\\
        =&~
        M(p) + \sum_{i=1}^N 2 \, \langle c_1(T\Sigma), A_i \rangle + 2 
                (\langle c_1(TX), B \rangle - B \bullet \Sigma) + M(x) - 2(n-1)
                +\\
                &+ N-1 + \sum_{i=1}^k \left ( 2 \, \langle c_1(TX),
            B_i \rangle - 2 B_i \bullet \Sigma \right ), 
\end{align*}
where $M(p)$ and $M(q)$ are the Morse indices of $p,q \in \Crit(f_\Sigma)$ and $M(x)$ is the Morse index of $x\in \Crit(f_W)$.
\end{proposition}

\begin{proposition}[\cite{McDuffSalamon}*{Proposition 6.2.7}]
    $\M_{k, \Sigma}^*((A_1, \dots, A_N); \J^r_\Sigma)$ is a Banach manifold. 
\end{proposition}

We will also make use of the following definition and proposition, the latter of  
which we prove in the next section.

\begin{definition} \label{def:universal_evaluation_map}
There is a \defin{universal evaluation map} 
\begin{align*}
    \ev_\Sigma \colon &\M_{k, \Sigma}^*((A_1, \dots, A_N); \J^r_\Sigma) \to \Sigma^{2N} \\
               &(w_1, \dots, w_N) \mapsto (w_1(0), w_1(\infty), w_2(0), w_2(\infty), \dots, w_N(\infty)).
\end{align*}
Similarly, we have
\begin{align*}
    \ev_{X,\Sigma} \colon &\M_{k, (X, \Sigma)}^*((B; A_1, \dots, A_N); \J_W) \to X \times \Sigma^{2N+1} \\
              &(v, w_1, \dots, w_N) \mapsto (v(0), v(\infty), w_1(0), w_1(\infty), w_2(0), \dots,
    w_N(\infty)),
\end{align*}
where $v$ is the holomorphic sphere in $X$ and the $w_i$ are the spheres in
$\Sigma$.

We have an evaluation map coming from simple collections of spheres in $X$:
\begin{align*}
    \ev^a_\Sigma \colon &\M^*_{X}((B_1, B_2, \dots, B_k); \J^r_W) \to \Sigma^{k} \\
                        &(v_1, \dots, v_k) \mapsto (v_1(\infty), v_2(\infty), \dots, v_k(\infty)).
\end{align*}
For spheres in $\Sigma$, we also obtain evaluation maps at the augmentation punctures  
\[
    \ev^a_\Sigma \colon \M_{k, \Sigma}^*((A_1, \dots, A_N); \J^r_\Sigma) \to
        \Sigma^{k} 
\]
and 
\[
    \ev^a_\Sigma \colon \M_{k, (X, \Sigma)}^*((B; A_1, \dots, A_N); \J^r_W) 
    \to \Sigma^{k}.
\]
We refer to these three maps denoted $\ev^a$ as \defin{augmentation evaluation
maps.}

\end{definition}

\begin{restatable}{proposition}{propJetConditionEvaluation}
\label{prop:jet_condition_evaluation}
Let $B_0, \dots, B_k$ be spherical classes in $H_2(X;\Z)$. Let $$r \ge \max_i B_i
\bullet \Sigma + 2.$$

The universal moduli space 
$\M^*_X( (B_1, \dots, B_k); \J^r_W)$ 
is a Banach manifold and the evaluation maps 
\begin{align*}
    \ev^a_\Sigma \colon &\M^*_X( (B_1, \dots, B_k); \J^r_W)
\to \Sigma^k: (f_1, f_2, \dots, f_k) \mapsto ( f_1(\infty), \dots, f_k(\infty)) \\
\ev_{X, \Sigma} \colon &\M^*_X( (B_0); \J^r_W) \to X \times \Sigma: f \mapsto (f(0), f(\infty))
\end{align*}
are submersions.
\end{restatable}

Recall that we have chosen a Morse function $f_\Sigma \colon \Sigma \to \R$
and a corresponding gradient-like vector field $Z_\Sigma$, such that
$(f_\Sigma,Z_\Sigma)$ is a Morse--Smale pair. The time-$t$ flow
of $Z_\Sigma$ is denoted by $\varphi^t_{Z_\Sigma}$ and the 
stable (ascending) $W^s_\Sigma(q)$ and unstable (descending) 
manifolds $W^u_\Sigma(p)$ were defined in Equation \eqref{(un)stable}. (Note
that these are the stable/unstable manifolds for the negative gradient flow.)

\begin{definition}
    The {\em{flow diagonal}} in $\Sigma \times \Sigma$ associated to the
    pair $(f_\Sigma, Z_\Sigma)$ is 
    \[
        \Delta_{f_\Sigma} \coloneqq \left \{ (x, y) \in \left ( \Sigma \setminus
            \Crit(f_\Sigma) \right )^2 \, | \, \exists t > 0 \text{ so } y =
        \varphi^t_{Z_\Sigma}(x) \right \}
    \]
 where $\Crit(f_\Sigma)$ is the set of critical points of $f_\Sigma$. 

\label{def:FlowDiagonal}
\end{definition}

We will now establish transversality of the evaluation maps to appropriate
products of stable/unstable manifolds, critical points, diagonals and flow diagonals. 
By \cite{McDuffSalamon}*{Proposition 6.2.8}, the key difficulty will be to deal
with constant spheres. For this, 
we will need the following lemma about evaluation maps 
intersecting with the flow diagonals. 

\begin{lemma} \label{L:flow_diagonal_0}
Suppose $f_0 \colon \M_0 \to \Sigma$ and $f_1 \colon \M_1 \to \Sigma$ are
submersions.

Then,
\begin{align*}
    F \colon \M_0 \times \M_1 &\to \Sigma^3 \\
    (m_0, m_1) &\mapsto (f_0(m_0), f_1(m_1), f_1(m_1))
\end{align*}
is transverse to $\Delta_{f_\Sigma} \times \{ p \}$ for each point $p \in
\Sigma$.
\end{lemma}

\begin{proof}
    Suppose $F(m_0, m_1) = (x, p, p) \in \Delta_{f_\Sigma} \times \{ p \}$.
    Then, there exists $t$ so that $
    \phi^t_{Z_\Sigma} ( x ) = \phi^t_{Z_\Sigma} ( f_0(m_0) ) = f_1(m_1)
    = p$. 

    Notice that 
    \[
        E \coloneqq \{ (d\phi_{Z_\Sigma}^{-t}(p) v, v) \, | \, v \in T_p \Sigma
        \} \subset T_{(x, p)}\Delta_{f_\Sigma}.
    \]
    For notational simplicity, we write $\Phi = d\phi_{Z_\Sigma}^{-t}(p)$.

    It follows then that 
    \begin{align*}
        &dF(m_0, m_1)\cdot T(\M_0 \times \M_1) + (E \oplus 0) \\
        &= \{ (df_0|_{m_0} v_0+\Phi w, df_1|_{m_1}v_1 + w, d{f_1}|_{m_1}v_1 ) \,
    | \, v_0 \in T_{m_0} \M_0, v_1 \in T_{m_1} \M_1, w \in T_p \Sigma \}  \\
     &= T\Sigma \oplus T\Sigma \oplus T\Sigma
    \end{align*}
    using the surjectivity of $df_0$, $df_1$.
    This then establishes the result, since $E \subset T_{(x, p)}\Delta_{f_\Sigma}.$
\end{proof}

From this, we now obtain the following:
\begin{lemma} \label{L:flow_diagonal}
Suppose $\M_0$ and $B$ are manifolds and there is a map
\[
\ev = (\ev_-,\ev_+) \colon \M_0 \to B \times \Sigma 
\]
that is transverse to $A\times pt$, for a submanifold $A$ of $B$ and for all points $pt \in \Sigma$. Suppose also that $\M_1$ is a manifold with a submersion $e\colon \M_1 \to \Sigma$.

Then the map
\begin{align*}
\hat \ev \colon \M_0 \times \M_1 &\to B\times \Sigma^3 \\
(m, n) &\mapsto (\ev_-(m), \ev_+(m), e(n), e(n))
\end{align*}
is transverse to $A \times \Delta_{f_\Sigma} \times \pt$, for all points $\pt \in \Sigma$.
\end{lemma}
\begin{proof}
    We apply the previous Lemma, using $f_0 = \ev_+$ and $f_1 = e$. 
    Then, $\hat \ev(m, n) = (\ev_-(m), F(m, n))$. The transversality to $A
    \times \Delta_{f_\Sigma} \times \pt$ follows by the transversality of $F$
    to $\Delta_{f_\Sigma} \times \pt$ together with the transversality of
    $\ev_-$ to $A$.
\end{proof}

\begin{lemma}
    \label{L:non_trivial_pearls_submersion}

Let $N \ge 1$, and let $A_1, \dots, A_N$ be spherical homology classes in
$\Sigma$ and let $B$ be a spherical homology class in $X$.

Suppose that  $S \subset \Sigma^{2N-2}$ is obtained by taking the product of
some number of copies of $\Delta_{f_\Sigma} \subset \Sigma^2$ and of the
complementary number of copies of 
$\{ (p, p) \, | \, p \in \Crit(f_\Sigma) \} \subset \Sigma^2$, 
in arbitrary order.
Let $\Delta \subset \Sigma^2$ denote the diagonal. 

Then 
if $\sum_{i=1}^N A_i \ne 0$, 
the universal evaluation map 
\[
    \ev_\Sigma 
        \colon \M_{k, \Sigma}^*((A_1, \dots, A_N); \J^r_\Sigma) \to \Sigma^{2N} 
\]
is transverse to the submanifold 
$\{x\} \times S \times \{ y \}$ for all $x, y \in \Sigma$.

If $B \ne 0$, the universal evaluation map 
\begin{equation*}
    \ev_{X, \Sigma} 
        \colon \M_{k, (X, \Sigma)}^*((B; A_1, \dots, A_N); \J^r_W) 
                \to  X \times \Sigma^{2N+1} 
\end{equation*}
is transverse to the submanifold
$\{x\} \times \Delta \times S \times\{y\}$ for any $x \in X$, $y \in \Sigma$.
\end{lemma}

\begin{proof}

    We consider the case of $\M_{k, \Sigma}^*$ in detail, since the argument
    is essentially the same for $\M_{k, (X, \Sigma)}^*$, though notationally
    more cumbersome.

    Suppose that $((v_1, \dots, v_N), J) \in \M_{k, \Sigma}^*((A_1, \dots,
    A_N); \J^r_\Sigma)$ is in the pre-image of $\{x\} \times S \times \{y\}$.
Write $S = S_1 \times S_2 \times \dots \times S_{N-1}$, where each
    $S_i \subset \Sigma^2$ is either the flow diagonal or the set of critical
    points. 

    Notice that the simplicity condition then requires that if some sphere
    $v_i$ is constant, $1 < i < N$, we must have that $S_{i-1}$ and $S_{i}$ are flow
    diagonals. If $v_1$ is constant, then $S_1$ is a flow diagonal and if
    $v_N$ is constant, $S_{N-1}$ is a flow diagonal.

    We will proceed by induction on $N$.  The case $N =1$ follows
    from \cite{McDuffSalamon}*{Proposition 3.4.2}.

    Now, for the inductive argument, we 
    suppose the result holds for any $S \subset \Sigma^{2(N-1)-2}$ of the
    form specified, and for any $k \ge 0$, for any collection of $N-1$ spherical
    classes, not all of which are zero.  

    Let now $A_1, \dots, A_N$ be spherical homology classes, not all of which are
    zero. 
Notice that each of these homology classes is represented by a
        $J_\Sigma$--holomorphic sphere, and thus has $\omega_\Sigma(A_i) \ge 0$
    for each $i$. 
In particular then, for such spherical classes, for any $1 \le a \le b \le N$,  $A_a, A_{a+1}\dots,
A_b$ are not all zero if and only if $\sum_{i=a}^b A_i \ne 0$.
    Then, at least one of $A_1, \dots, A_{N-1}$ or $A_2, \dots, A_N$ is a
    collection of spheres satisfying the hypotheses of the lemma. 
    For simplicity of notation,
    let us assume that $A_1+ \dots+ A_{N-1} \ne 0$. 
    Let $S_0 = S_1 \times S_2 \times \dots \times S_{N-2}$.
    Let $k=k_0+k_N$ where $k_N$ is
    the number of marked points we consider on the last sphere.
    By the induction hypothesis, we have that the evaluation map
    \[
                \M_{k_0, \Sigma}^*((A_1, \dots, A_{N-1}); \J^r_\Sigma) \to \Sigma^{2(N-1)} 
    \]
    is transverse to $\pt \times S_0 \times \pt$. Denote this map by
    $\ev_0$.

    Notice that $\M_{k, \Sigma}^*((A_1, \dots, A_N); \J^r_\Sigma) \subset 
    \M_{k_0, \Sigma}^*((A_1, \dots, A_{N-1}); \J^r_\Sigma) \times \M_{k_N,
    \Sigma}^*(A_N; \J^r_\Sigma)$. Let then 
    $\ev_N \colon \M_{k, \Sigma}^*((A_1, \dots, A_N); \J^r_\Sigma) \to
    \Sigma^{2}$ be the evaluation at $0$ and $\infty$ in the $N$-th sphere.
    We therefore have
    \[
        \ev_{\Sigma} \colon \M_{k, \Sigma}^*((A_1, \dots, A_N); \J^r_\Sigma) \to
        \Sigma^{2N} 
    \]
    given by $\ev_\Sigma = (\ev_0, \ev_N)$. 

    If $A_N \ne 0$, the result follows again from
    \cite{McDuffSalamon}*{Proposition 3.4.2}.

    If, instead, $A_N = 0$, we have from above that $S_{N-1} = \Delta_{f_\Sigma}$.
    Notice that the evaluation map of constant spheres on $\Sigma$ has image
    on the diagonal in $\Sigma \times \Sigma$. The result now follows by applying Lemma
    \ref{L:flow_diagonal}.  

    The case with a sphere in $X$ follows a nearly identical induction argument, 
    though the base case consists of a single sphere in $X$. The required
    submersion to $X \times \Sigma$ now follows from Proposition
    \ref{prop:jet_condition_evaluation}, and the induction proceeds as before.
\end{proof}

\begin{proposition}
Let $N \ge 0$.  
Suppose that  $S \subset \Sigma^{2N-2}$ is obtained by taking the product of
some number of copies of $\Delta_{f_\Sigma} \subset \Sigma^2$ and of the
complementary number of copies of 
$\{ (p, p) \, | \, p \in \Crit(f_\Sigma) \} \subset \Sigma^2$, 
in arbitrary order.

Let $\Delta \subset \Sigma \times \Sigma$ denote the diagonal and let
$\Delta_k$ denote the diagonal $\Sigma^k$ in $\Sigma^k \times \Sigma^k$.

Let $p, q$ be critical points of $f_\Sigma$ and let $x$ be a critical point of
$f_W$.

Then the universal evaluation maps together with augmentation evaluation maps
\begin{align*}
    \ev_\Sigma \times \ev^a_\Sigma \times \ev^a_\Sigma 
    &\colon \M_{k, \Sigma}^*((A_1, \dots, A_N); \J^r_\Sigma) 
    \times \M^*_X((B_1, \dots, B_k); \J^r_W) 
    \to \Sigma^{2N} \times \Sigma^k\times\Sigma^k  \\
    \ev_{X, \Sigma} \times \ev^a_\Sigma \times \ev^a_\Sigma
    &\colon \M_{k, (X, \Sigma)}^*((B; A_1, \dots, A_N); \J^r_W) 
            \times \M^*_X((B_1, \dots, B_k); \J^r_W) 
    \to  X \times \Sigma^{2N+1} \times \Sigma^k \times \Sigma^k
\end{align*}
are transverse to, respectively, 
\begin{align*}
    &W^{s}_\Sigma(q) \times S \times W^u_{\Sigma}(p) \times
    \Delta_k \\
&W^{u}_X(x) \times \Delta \times S \times W^u_\Sigma(p) \times 
    \Delta_k.
\end{align*}
  \label{ev Sigma submersion} 
\end{proposition}

\begin{proof}
We will consider only the first case, the second being analogous. 
    Notice first that by Proposition \ref{prop:jet_condition_evaluation}
    the augmentation evaluation map 
    \[
        \ev^a_\Sigma \colon \M^*_X((B_1, \dots, B_k); \J^r_W) \to \Sigma^k
    \]
    is a submersion. It suffices therefore to prove that
\begin{align*}
    \ev_\Sigma &\colon \M_{k, \Sigma}^*((A_1, \dots, A_N); \J^r_\Sigma) \to \Sigma^{2N} 
\end{align*}
is transverse to 
$
    W^{s}_\Sigma(q) \times S \times W^u_{\Sigma}(p) $.
    
    The proposition follows immediately if at least one of
    the $A_i, i=1, \dots, N$ is non-zero, or if we are considering the
    case of a chain of pearls with a sphere in $X$, by applying Lemma
    \ref{L:non_trivial_pearls_submersion}. 

    The only case then that must be examined is that of a chain of pearls entirely in
    $\Sigma$ with all spheres constant. In this case, the evaluation map from
    the moduli space $\M_{k, \Sigma}^*((0,0,\dots, 0), \J^r_W)$ factors through the evaluation map
    \[
        \{ (z_1, \dots, z_N ) \in \Sigma^N \, | \, z_i = z_j \implies i=j \}
        \times \J^r_W \to \Sigma^{2N}.
    \]

    Transversality follows from the Morse--Smale condition on the gradient-like vector field 
    $Z_\Sigma$. This gives that the intersection of $W^s_\Sigma(q)$ and
    $W^u_\Sigma(p)$ is transverse, and hence that the diagonal in $\Sigma \times
    \Sigma$ is transverse to $W^s_\Sigma(q) \times W^u_\Sigma(p)$, which is what we need when $N=1$.
    The case of $N \ge 2$ is similar, using the description of the tangent space to the flow diagonal at 
    $(x, y) \in \Delta_{f_\Sigma}$, such that $\varphi^t_{Z_\Sigma}(x) = y$ for some $t > 0$, as
    \[
    T_{(x,y)} \Delta_{f_\Sigma} = \{ (v, d\varphi^t_{Z_\Sigma}(x)v + c Z_\Sigma(y)) \, | \, v \in T_x \Sigma, c \in \R \} 
	\subset T_x \Sigma \oplus T_y \Sigma.
    \]
\end{proof}

Proposition \ref{ev Sigma submersion} can be combined with standard 
Sard--Smale arguments, the fact that $P\colon \J^r_W \to \J^r_\Sigma$ is an
open and surjective map and Taubes's method for passing to smooth almost
complex structures 
(see for instance \cite{McDuffSalamon}*{Theorem 6.2.6}) to give the following proposition:

\begin{proposition} \label{ev Sigma submersion bis}

There exist residual sets of almost complex structures 
$\J_W^{reg}\subset \J_W$ and $\J_\Sigma^{reg} = P( \J_W^{reg})$,
so that for fixed $J_W \in \J_W^{reg}$ and 
$J_\Sigma = P(J_W)$, 
the restrictions of 
the evaluation maps 
$\ev_\Sigma \times \ev^a_\Sigma \times \ev^a_\Sigma$ and $\ev_{X,\Sigma} \times \ev^a_\Sigma \times \ev^a_\Sigma$ 
to 
\begin{align*}
    &\M_{k, \Sigma}^*((A_1, \dots, A_N); J_\Sigma) \times \M_X^*((B_1, \dots, B_k); J_W) \quad \text{ and }\\
    &\M_{k, (X, \Sigma)}^*((B; A_1, \dots, A_N); J_W) \times \M_X^*((B_1, \dots, B_k); J_W),
\end{align*}
respectively,
are transverse to 
the submanifolds of Proposition \ref{ev Sigma submersion}.  

\end{proposition}

The transversality statement of the main result of this section, 
Proposition \ref{necklaces are regular}
now follows. The dimension formulas follow from usual index arguments,
combining Riemann--Roch with contributions from the constraints imposed by 
the evaluation maps.

\subsection{Transversality for spheres in \texorpdfstring{$X$}{X} with order of contact constraints in \texorpdfstring{$\Sigma$}{Sigma}} 
\label {S:transverse contact}

We will now consider transversality for a chain of pearls with a sphere in $X$.
We will extend the results from 
Section 6 in \cite{CieliebakMohnkeTransversality}. In that paper, Cieliebak and Mohnke prove that the moduli space of simple
curves not contained in $\Sigma$,
 with a condition on the order of contact with $\Sigma$,
 can be made transverse by a perturbation of the almost 
complex structure away from $\Sigma$. We will extend this result to show that additionally the 
evaluation map to $\Sigma$ at the point of contact can be made transverse. This can be useful, for instance, 
to define relative Gromov--Witten invariants with constraints on homology classes in $\Sigma$. 

Recall that $\Sigma$ is a symplectic divisor and $N\Sigma$ is its symplectic normal
bundle equipped with a Hermitian structure. 
Keeping in mind the discussion in Section \ref{S:setup}  (in particular the identification of 
$X\setminus \Sigma$ with $W$ in Lemma \ref{planes = spheres}), we will by an abuse of notation 
identify an almost complex struture on $W$ with the corresponding almost complex structure on $X$. 
We have fixed a symplectic neighbourhood
$\varphi \colon \U \to X$ where $\varphi \colon \overline{\U} \to X$ is an embedding. 
From Definition \ref{def:J_W}, we require that all $J_X \in \J_W$ 
have that $J_X$ is standard in the image $\varphi(\U) \subset X$ of this
neighbourhood.

Fix an almost complex structure $J_0 \in \J_W$. We may suppose that $P(J_0) \in
\J_\Sigma$ is an almost complex structure in the residual set $\J_\Sigma^{reg}$ given by
Proposition \ref{necklaces are regular}, though this isn't strictly speaking
necessary.

Let $\V \coloneqq X \setminus \varphi(\overline{\U})$. Following Cieliebak-Mohnke
\cite{CieliebakMohnkeTransversality}, let
$\J(\V)$ be the set of all almost complex structures on $X$ compatible with $\omega$
that are equal to $J_0$ on $\varphi(\U)$. Similarly, we will let $\J^r(\V)$ be
the compatible almost complex structures of $C^r$ regularity.

To define the order of contact, consider an almost complex 
structure $J_X \in \J_W$ and a $J_X$-holomorphic sphere $f \colon \CP^1 \to X$
with $f(0) \in \Sigma$, an isolated intersection.
Choose coordinates $s+it = z\in\C$ on the domain and local coordinates
near $f(0)\in \Sigma$ on the target, such that $f(0)\in \Sigma \subset X$
corresponds to $0\in \C^{n-1} = \C^{n-1}\times \{0\} \subset \C^{n-1}
\times \C$. Write $\pi_\C \colon \C^n\to \C$ for projection onto the last
coordinate (which is to be thought of as normal to $\Sigma$). Assume
also that $J_X(0) = i$.
Then, $f$ has contact of order $l$ at $0$ if the vector of all partial derivatives
of orders 1 through $l$ (denoted by $d^l f(0)$) has trivial projection to $\C$. 
We can write this condition as $d^l f(0) \in T_{f(0)}\Sigma$. 
We define then the order of contact at an arbitrary point in $\CP^1$ by 
precomposing with a M\"obius transformation.
(This is well-defined, by \cite{CieliebakMohnkeTransversality}*{Lemma 6.4}.)

Define the space of simple pseudoholomorphic maps into $X$ that have 
order of contact $l$ at $\infty$ to a point in $\Sigma$ to be
\[
\begin{split}
\M^*_{\infty, l, (X, \Sigma)}(\J_W) 
\coloneqq
    \{(f,J_X) \in\, & W^{m,p}(\CP^1, X) \times \J_W
                \, | \, 
                    \overline\partial_{J_X} f = 0, \\
                    &f(\infty) \in \Sigma, \,
                                            d^l f(\infty) \in T_{f(\infty)}\Sigma, \\
                    & f \text{ simple}, f^{-1} (\V) \neq \emptyset \} 
\end{split}
\]
where 
we require $m \ge l+2$.  
Note that our notation differs somewhat from the notation in
\cite{CieliebakMohnkeTransversality}. 

In this section, we need to have a higher regularity on our Sobolev spaces 
to make sense of the order of contact condition. For the remaining moduli
spaces, for simplicity of notation, we have taken $m=1$, where this is not a
problem. Notice that by elliptic regularity, the moduli spaces themselves are
manifolds of smooth maps, and are independent of the choice of $m$. This only affects the
classes of deformations we consider in setting up the Fredholm theory.

In this section, we will prove Proposition \ref{prop:jet_condition_evaluation}, which was stated and
used above:
\propJetConditionEvaluation*

Notice that it suffices to prove this when considering only pairs $(f, J_X)
\in 
\M^*_X( (B_0); \J^r_W)$ with the additional condition that $J_X \in \J^r(\V)$.

We also observe 
that if $l = B_0 \bullet \Sigma$, we have that 
$\M^*_X((B_0); \J_W) \subset \M^*_{\infty, k, (X, \Sigma)}(\J_W)$ for each $k
\le l$.
Furthermore, 
$\M^*_X((B_0); \J_W)$ is a connected component of  $\M^*_{\infty, l, (X, \Sigma)}(\J_W)$. 
This observation will enable us to obtain the result by inducting on $k$. 

The proposition will follow by a modification of the proof given in 
\cite{CieliebakMohnkeTransversality}*{Section 6}. Instead of reproducing their
proof, we indicate the necessary modifications. In order to be as consistent as possible
with their notation, we consider the point of contact with $\Sigma$ to be at $0$. 

Consider a $J_X$-holomorphic map $f \colon \CP^1\to X$ such that $f(0)\in
\Sigma$ with order of contact $l$. In the notation of
\cite{CieliebakMohnkeTransversality}, we are interested in the case
of only one component $Z=\Sigma$.
We will obtain transversality
of the evaluation map at $0$ by varying $J_X$ freely in the complement
of our chosen neighbourhood of the divisor, $\V = X \setminus
\varphi(\overline \U)$.

The linearized Cauchy--Riemann operator at $f$ with respect to a torsion-free connection is 
$$
(D_f \xi)(z)  = \nabla_s \xi(z) + J_X(f(z))\nabla_t\xi(z) + (\nabla_{\xi(z)} J_X(f(z))) \, f_t(z).
$$
At a coordinate chart around $z=0$, we can specialize to the standard Euclidean connection in $\R^{2n} = \C^n$ 
(which preserves $\C^{n-1}$ along $\C^{n-1}$), we get
$$
(D_f \xi)(z)  = \xi_s(z) + J_X(f(z))\xi_t(z) + A(z) \xi(z)
$$
where 
$$
A(z) \xi(z) = (D_{\xi(z)} J_X(f(z))) \, f_t(z)
$$
(see also page 317 in \cite{CieliebakMohnkeTransversality}). 

We need the following adaptation of Corollary 6.2 in \cite{CieliebakMohnkeTransversality}. 

\begin{lemma}
Suppose 
$(f, J_X) \in \M^*_{\infty, l, (X, \Sigma)}(\J^r_W)$ with $J_X \in \J^r(\V)$, $r \ge m$.

After choosing local coordinates, suppose $f(0) \in \Sigma$ and in coordinates
around $f(0)$, $\Sigma$ is mapped to $\C^{n-1}$ and is thus preserved by
$J_X$.

Denote the unit disk by $D^2$ and let $\xi \colon (D^2,0) \to (\C^n,0)$ be such that $D_f\xi = 0$. 
Given $0<k\leq l$, if $\xi(0)\in \C^{n-1}$, 
$d^{k-1}\xi(0) \in \C^{n-1}$ and $\frac{\partial^k \xi}{ \partial s^k}(0) \in \C^{n-1}$, then $d^{k}\xi(0)\in \C^{n-1}$. 
\label{bootstrap}
\end{lemma}
\begin{proof} We need to show that $\frac{\partial^k \xi}{\partial s^{k-i} \partial t^i}(0) \in \C^{n-1}$ for all 
$0\leq i \leq k$. It will be convenient to use multi-index notation for partial derivatives, and denote the previous 
expression by $D^{(k-i,i)}\xi(0)$. The case $i = 0$ is part of the hypotheses of the Lemma. 
For the induction step, note that $D_f\xi = 0$ combined with the product rule implies that 
$$
D^{(k-i,i)}\xi(z) = J_X(f(z)) \Big(D^{(k-i+1,i-1)} \xi(z) + \sum_{\alpha,\beta}
D^\alpha(J_X(f(z)) D^{\beta}\xi(z) + \sum_{\alpha',\beta'} D^{\alpha'}A(z)
D^{\beta'}\xi(z) \Big). %
$$
Here, $\alpha$ and $\beta$ are multi-indices such that $\alpha=(a_1,a_2)$ for $0\leq a_1\leq k-i, 0\leq a_2 \leq i-1, \alpha \neq (0,0)$ and $\alpha + \beta = (k-i,i)$. Similarly, $\alpha'$ and $\beta'$ are multi-indices such that $\alpha'=(a_1',a_2')$ for $0\leq a_1'\leq k-i, 0\leq a_2' \leq i-1$ and $\alpha' + \beta' = (k-i,i-1)$. 
The hypotheses of the Lemma and the induction hypothesis implie that the derivatives of $\xi$ on the right hand side take values in $T_{f(0)}\Sigma$. The fact that $J_X$ and $\nabla$ preserve $\C^{n-1}$ along $\C^{n-1}$, and that $d^l f(0)\in T_{f(0)}\Sigma$, implies the induction step. 
\end{proof}

We now prove the key property of the linearized evaluation map: 
\begin{proposition}
 For $m - 2/p>l$, $r \ge m$, the universal evaluation map 
 \begin{align*}
\ev_{X,\Sigma}
\colon  \M^*_{\infty, l, (X, \Sigma)}(\J^r_W)
&\to \Sigma \\
(f,J_X) & \mapsto f(0)
 \end{align*}
is a submersion. 
\end{proposition}
\begin{proof}
    We show that for every $0 \le k \le l$, and $(f, J_X) 
    \in \M^*_{\infty, k, (X, \Sigma)}(\J^r(\V))$,
 \begin{align*}
     (d\ev_{X,\Sigma})_{(f,J_X)}\colon T_{(f,J_X)} \M^*_{\infty, k, (X, \Sigma)}(\J^r(\V))
  &\to T_{f(0)}\Sigma \\
  (\xi,Y) &\mapsto \xi(0)
 \end{align*}
is surjective. By Lemma 6.5 in \cite{CieliebakMohnkeTransversality}, 
\begin{align*}
T_{(f,J_X)}
\M^*_{\infty, k, (X, \Sigma)}(\J^r(\V))
= \{(\xi,Y) \in &T_f W^{m,p}(\CP^1, X)\times T_{J_X}\J^r(\V) \, | \, \\
            &D_f\xi + \frac{1}{2}Y(f) \circ df \circ j = 0,\\
            &\xi(0) \in T_{f(0)}\Sigma, d^k \xi(0) \in T_{f(0)}\Sigma \}.
\end{align*}
We argue by induction on $k$. The case $k=0$ is a special case of Proposition 3.4.2 in \cite{McDuffSalamon}. We assume that the claim is true for $k-1$ and prove it for $k$. 

Take any $v\in T_{f(0)}\Sigma$. By induction, there is $(\xi_1,Y_1) \in 
T_{(f,J_X)}
\M^*_{\infty, k-1, (X, \Sigma)}(\J^r(\V))
$ such that $(d\ev_{X,\Sigma})_{(f,J_X)}(\xi_1,Y_1) = v$ and $d^{k-1} \xi_1(0) \in T_{f(0)}\Sigma$. Let now $\tilde \xi \in T_f W^{m,p}(\CP^1, X)$ be given by
$$
\tilde \xi(z) = -\frac{z^k}{k!}\beta(z) \,\pi_{\C} \left(\frac{\partial^{k}}{\partial s^k} \xi_1\right)(0)
$$
where $\beta \colon \C\to [0,1]$ is a smooth function that is identically 1 near 0 and has compact support contained in $\C \setminus f^{-1}(\V)$. Writing 
$$
(D_f \xi)(z)  = \xi_s(z) + i \xi_t(z) + (J_X(f(z)) - i) \xi_t(z) + A(z) \xi(z)
$$
we have $(D_f \tilde \xi)(0) = 0$ and $d^{k-1}(D_f \tilde \xi)(0) = 0$ (this follows the fact that $\tilde \xi_s + i \tilde \xi_t \equiv 0$ near 0). By Lemma 6.6 in \cite{CieliebakMohnkeTransversality}, there is $(\hat \xi,\hat Y) \in T_f W^{m,p}(\CP^1, X)\times T_{J_X}\J(\V)$ such that $\hat \xi(0) = 0$, $d^k(\hat \xi)(0) = 0$ and 
$$
D_f\hat\xi + \frac{1}{2}\hat Y(f) \circ df \circ j = - D_f \tilde \xi .
$$
Let now $\xi_2 = \xi_1 + \tilde \xi + \hat \xi$ and $Y_2 = Y_1 + \hat Y$. We have 
$$
D_f \xi_2 + \frac{1}{2} Y_2(f) \circ df \circ j = 0  
$$
as well as $\xi_2(0) = v$, $d^{k-1}(\xi_2)(0) \in T_{f(0)}\Sigma$ and $\pi_\C \left(\frac{\partial^{k}}{\partial s^k} \xi_2\right)(0) = 0$. Lemma \ref{bootstrap} implies that $d^{k}(\xi_2)(0) \in T_{f(0)}\Sigma$, hence $(\xi_2,Y_2)\in T_{(f,J_X)} \M^*_{\infty, l, (X, \Sigma)}(\J^r_W)$.
This completes the proof. 
\end{proof}

Observe now that by combining this with standard arguments (see, for
instance, \cite{McDuffSalamon}*{Proposition 3.4.2}, which is also used in the proof of Proposition \ref{necklaces are regular} above), 
we obtain the transversality for the evaluation at a point, taking values
in $X$. This finishes the proof of Proposition \ref{prop:jet_condition_evaluation}.

\subsection{Proof of Proposition \ref{cascades form manifold}}
\label{S:proof}

We are now ready to complete the proof of Proposition \ref{cascades form
manifold}. To this end, we will show that the transversality problem for a
cascade reduces to the already solved transversality problem for chains of
pearls. The two key ingredients of this are the splitting of the linearized
operator given by Lemma \ref{L:linearizationFloer} and a careful study of the
flow-diagonal in $Y \times Y$.

Recall from Definition \ref{def:J_W} that 
$\J_Y$ denotes the space of compatible, cylindrical, Reeb--invariant almost
complex structures on $\R \times Y$. These are obtained as lifts of the almost
complex structures in $\J_\Sigma$.  
Let $\J_Y^{reg}$
be the set of almost complex structures on $\R \times Y$ that are lifts of
the almost complex structures in $\J_\Sigma^{reg}$ (see Proposition \ref{ev Sigma submersion bis}).  

Recall from Definition \ref{def:J_W} and from Proposition \ref{prop:Xacs},
if $J_W \in \J_W$ is an almost
complex structure on $W$ that is of the type we consider, it induces an
almost complex structure $P(J_W) = J_\Sigma \in \J_\Sigma$. The restriction of
$J_W$ to the cylindrical end of $W$, $J_Y$, is then a translation and
Reeb-flow invariant almost complex structure on $\R \times Y$ that has
$d\pi_\Sigma J_Y = J_\Sigma d\pi_\Sigma$.

Recall that the biholomorphism $\psi \colon W \to X \setminus \Sigma$ given in 
Lemma \ref{planes = spheres} allows us to identify holomorphic planes in $W$ with 
holomorphic spheres in $X$. In the following, we will suppress the distinction when convenient.

Recall also that by the definition
of an admissible Hamiltonian (Definition \ref{def:J_shaped}), for
each non-negative integer $m$, there exists a unique $b_m$ so that
$h'(\e^{b_m}) = m$. Then $Y_m = \{ b_m \} \times Y \subset \R \times Y$
is the corresponding Morse--Bott family of 1-periodic Hamiltonian orbits
that wind $m$ times around the fibre of $Y \to \Sigma$.

We now define moduli spaces of Floer cylinders, from which we will extract
the moduli spaces of cascades by imposing the gradient flow-line conditions.
First, we define the moduli spaces relevant for the differential connecting two
generators in $\R \times Y$. Then, we will define the moduli spaces relevant for
the differential connecting to a critical point in $W$.

\begin{definition}
Let $N \ge 1$, let $A_1, \dots, A_N \in H_2(\Sigma; \Z)$ be spherical homology
classes. Let $J_Y \in \J_Y$.

Define $\MM_{H, k, \R \times Y; k_-, k_+}^*((A_1, \dots, A_N); J_Y )$ 
to be a set of tuples of 
punctured cylinders $ (\tilde v_1, \ldots, \tilde v_N)$ with the
following properties:

\begin{enumerate}
\item There is a partition of $\Gamma = \Gamma_1 \cup \dots \cup \Gamma_N$ of
    $k$ augmentation marked points 
    with \[
        \tilde v_i \colon \R \times S^1 \setminus \Gamma_i \to \R \times Y
    \]
    so that $\tilde v_i$ is a finite hybrid energy punctured Floer cylinder.
    For each $z_j \in \Gamma$, there is a positive integer
    multiplicity $k(z_j)$. Let $v_i$ denote the projection to $Y$.
\item There is an increasing list of $N+1$ multiplicities from $k_-$ to $k_+$:
    \[ k_- = k_0 < k_1 < k_2 < \dots < k_N = k_+ \]
    such that, for each $i$, the cylinder $\tilde v_i$ has multiplicities
    $k_{i}$ and $k_{i-1}$ at $\pm \infty$, i.e.~ $\tilde v_i(+\infty, \cdot) \in Y_{k_i}, \tilde v_i(-\infty, \cdot) \in Y_{k_{i -1}}$.
\item The Floer cylinders $\tilde v_i$ are simple in the sense that their
    projections to $\Sigma$ are either somewhere injective or constant, if
    constant, they have at least one augmentation puncture, and their images are
    not contained one in the other.
\item For each $i$, and for every puncture $z_j \in \Gamma_i$, 
the augmentation puncture has a limit whose multiplicity is given by $k(z_j)$; 
i.e.~$\lim_{\rho \to -\infty} v_i(z_j + \e^{2\pi(\rho + i\theta)})$ is a Reeb orbit of multiplicity $k(z_j)$.
\item The projections of the Floer cylinders to $\Sigma$ represent the homology
    classes $A_i, i=1, \dots, N$; i.e.~$(\pi_\Sigma( \tilde v_i ) )_{i=1}^N \in \M^*_k( (A_1,
                  \dots, A_N), J_\Sigma)$.
\end{enumerate}

\end{definition}

Let $B \in H_2(X; \Z)$ be a spherical homology class, $B \ne 0$. 
Let $J_W$ be an almost complex structure on $W$ as given by Lemma \ref{planes =
spheres}, matching $J_Y$ on the cylindrical end.
\begin{definition} 
    Define the moduli space
    \[
    \MM_{H, k, W; k_+}^*\left((B; A_1,\ldots,A_N);J_W\right) 
    \]
    to consist of
    tuples
    \[
    (\tilde v_0, \tilde v_1, \ldots, \tilde v_N) 
    \]
    with the properties
    \begin{enumerate}
        \item The map $\tilde v_0 \colon \R \times S^1 \to W$ is a finite energy
            holomorphic cylinder with removable singularity at $-\infty$.
\item There is a partition of $\Gamma = \Gamma_1 \cup \dots \cup \Gamma_N$ of
    $k$ augmentation marked points 
    with \[
        \tilde v_i \colon \R \times S^1 \setminus \Gamma_i \to \R \times Y, \qquad i\geq 1,
    \]
    so that each $\tilde v_i$ is a finite hybrid energy punctured Floer
    cylinder. For each $z_j \in \Gamma$, there is a positive integer
    multiplicity $k(z_j)$. Denote by $v_i$ the projection of $\tilde v_i$
    to $Y$.
\item There is an increasing list of $N+1$ multiplicities:
    \[ k_0 < k_1 < k_2 < \dots < k_N = k_+ \]
\item For each $i\geq 1$, and for every puncture $z_j \in \Gamma_i$, 
the augmentation puncture has a limit whose multiplicity is given by $k(z_j)$; 
i.e.~$\lim_{\rho \to -\infty} v_i(z_j + \e^{2\pi(\rho + i\theta)})$ is a Reeb orbit of multiplicity $k(z_j)$.
\item The Floer cylinders $\tilde v_i$ for $i\geq 1$ are simple, in the strong sense that the projections
    to $\Sigma$ are somewhere injective or constant, and have images not contained one in the
    other. The cylinder $\tilde v_0$ is somewhere injective in $W$.
\item The projections of the Floer cylinders to $\Sigma$ represent the homology
    classes $A_i, i=1, \dots, N$; i.e.~$\pi_\Sigma( \tilde v_i ) )_{i=1}^N \in
    \M^*_k( (B; A_1, \dots, A_N), J_W)$.
\item After identifying $\tilde v_0$ with a holomorphic sphere in $X$, $\tilde v_0$ represents the
    homology class $B \in H_2(X; \Z)$.
\item 
 The cylinder $\tilde v_1$ has multiplicity $k_1$ at $+\infty$ and $\tilde
 v_1(+\infty, \cdot) \in Y_{k_1}$. At $-\infty$, $\tilde v_1$ converges to a
 Reeb orbit in $\{ -\infty \} \times Y$. This Reeb orbit has multiplicity $k_0$.
\item For each $i \ge 2$, the cylinder 
    $\tilde v_i$ has multiplicities $k_{i}$ and $k_{i-1}$ at $\pm \infty$:
 $\tilde v_i(+\infty, \cdot) \in Y_{k_i}, \tilde v_i(-\infty, \cdot) \in Y_{k_{i -1}}$.

 \item The cylinder $\tilde v_0$ converges at $+\infty$ to a Reeb orbit of multiplicity
     $k_0$.
\end{enumerate}

\end{definition}

Observe that these moduli spaces are non-empty only if for each $i =1, \dots, N$, 
\[ K\omega(A_i) = k_{i} - k_{i-1} - \sum_{z \in \Gamma_i} k(z).\]
Furthermore, for $\MM_{H, k, W}^*$, we must also have \[
    B \bullet \Sigma = K \omega(B) = k_0.
\]
Note also that these moduli spaces have a large number of connected components,
where different components have different partitions of $\Gamma$ or 
different intermediate multiplicities.

Identifying holomorphic spheres in $X$ with finite
energy $J_W$-planes in $W$, we consider also the moduli space of
holomorphic planes $\M^*_X( (B_1, \dots, B_k); J_W)$ as in Definition
\ref{D:moduliSpacesPearls}. 

The space $\MM_{H, k, \R \times Y}^*\left((A_1,\ldots,A_N);J_Y\right)$ consists of
$N$-tuples of somewhere injective punctured Floer cylinders in $\R\times Y$.
Similarly, $\MM_{H, k, W}^*$ consist of $N$-tuples of punctured Floer cylinders
in $\R \times Y$ together with a holomorphic plane in $W$ (which we can therefore also 
interpret as a holomorphic sphere in $X$).
The cylinders and the eventual plane have asymptotics with matching multiplicities, 
but are otherwise unconstrained.  
These two moduli spaces, $\MM_{H, k, \R \times Y}^*$ and $\MM_{H, k, W}^*$
fail to be simple split Floer cylinders with cascades (as in Definition \ref{D:simple cascade}) 
in two ways: they are missing the gradient trajectory constraints on their
asymptotic evaluation maps, and they are missing their augmentation planes. In
order to impose these conditions, we will need to study these evaluation maps
and establish their transversality.

\begin{proposition}
For $J_Y \in \J_Y^{reg}$,  
$\MM_{H, k, \R \times Y}^*\left((A_1,\ldots,A_N);J_Y\right)$
is a manifold of dimension 
\begin{equation*}
N(2n-1) + \sum_{i=i}^N 2 \, \langle c_1(T\Sigma), A_i\rangle + 2k.
\end{equation*}

For $J_W \in \J_W^{reg}$, 
    $\MM_{H, k, W}^*\left((B; A_1,\ldots,A_N);J_W\right) $
    is a manifold of dimension
\begin{equation*}
      N(2n-1) + 2n + 1 
     + \sum_{i=i}^N 2 \, \langle c_1(T\Sigma), A_i \rangle
     + 2 ( \langle c_1(TX), B \rangle - B \bullet \Sigma)
     + 2k. 
\end{equation*}
\label{cylinders transverse}
\end{proposition}

\begin{proof}
    
    Consider first the case of cylinders in $\R \times Y$. 
    Let \[
    (\tilde v_1, \dots, \tilde v_N) \in \MM_{H, k, \R \times Y}^*\left((A_1,\ldots,A_N);J_Y\right).
    \]

    Recall from Proposition \ref{necklaces are regular}
    that for $J_\Sigma \in \J_\Sigma^{reg}$, we have transversality for $D_{w_i}^\Sigma$ for 
    each sphere $w_i = \pi_\Sigma( v_i)$.
    
    Let $\delta > 0$ be sufficiently small. For each $i=1, \dots, N$, by Lemma \ref{L:verticalOperatorTransverse}, 
    $D_{\tilde v_i}^\C$ is surjective when considered on $W^{1,p,-\delta}$ (with exponential growth),
    and has Fredholm index $1$. The operator considered instead on the space $W^{1,p,\delta}_{\mathbf V}$, with $V_{-\infty} = V_{+\infty} = i\R $ and $V_P=\C$ for any puncture $P$ on the domain of $\tilde v_i$, 
    has the same 
    kernel and cokernel by Lemma \ref{L:exponentialGrowth}. 
    Thus, the operator, acting on sections free to move in the Morse--Bott family of orbits, is surjective and has index $1$.

    Since the operator $D_{\tilde v_i}$ is upper triangular from Lemma \ref{L:linearizationFloer}, and its diagonal
    components are both surjective, the operator is surjective. 
    Since the Fredholm index is the sum of these, each component $\tilde v_i$ contributes an index of 
    $1 + 2n-2 + 2 \, \langle c_1(T\Sigma), A_i \rangle + 2k_i = 2n-1 + 2 \, \langle c_1(T\Sigma), A_i \rangle + 2k_i$, where $k_i$ is the number of punctures.
  
    We now consider the case of a collection 
    \[
        (\tilde v_0, \tilde v_1, \dots, \tilde v_N) \in 
    \MM_{H, k, W}^*\left((B; A_1,\ldots,A_N);J_W\right).\]
     The same consideration as previously gives that $\tilde v_2, \dots, \tilde v_N$ are transverse and each contributes
     an index of $2n-1+ 2 \, \langle c_1(T\Sigma), A_i \rangle + 2k_i$, where $k_i$ is the number of punctures.
     For the component $\tilde v_1$, again applying Lemma \ref{L:linearizationFloer} and applying Lemma \ref{L:verticalOperatorTransverse} in the 
     case where the $-\infty$ end of the cylinder converges to a Reeb orbit at $\{ -\infty \} \times Y$, we obtain that the vertical Fredholm operator is surjective and has index $2$. The linearized Floer operator at $\tilde v_1$ is then surjective and 
     has index $2n + 2 \, \langle c_1(T\Sigma), A_1 \rangle + 2 k_1$.
          By Lemma \ref{planes = spheres}, the plane $\tilde v_0$ can be identified with a sphere in $X$ with an order of contact $l = B \bullet \Sigma$  
     with $\Sigma$.
     Its Fredholm index is $2n + 2 ( \langle c_1(TX), B \rangle - l)$.     
     The total Fredholm index is therefore
     \[
     (N-1)(2n-1) + 2n + 2n
     + \sum_{i=i}^N 2 \, \langle c_1(T\Sigma), A_i \rangle
     + 2 ( \langle c_1(TX), B \rangle - B \bullet \Sigma) + 2k.
     \]
     
For both cases, the result now follows from the implicit function theorem.
\end{proof}

It now suffices to prove the transversality of evaluation maps to the products of stable/unstable manifolds and flow diagonals,
and also transversality of the augmentation evaluation maps, in order to obtain the constraints coming
from pseudo-gradient flow lines. Indeed, let $(\tilde v_1, \dots, \tilde v_N)$ be a 
collection of $N$ cylinders in $\MM_{H, k, \R \times Y; k_-, k_+}^*\left((A_1,\ldots,A_N);J_Y\right)$.
Write each of the $\tilde v_i \colon \R \times S^1 \to \R \times Y$ as a pair
$\tilde v_i = (b_i, v_i)$. 
We then have asymptotic evaluation maps 
\begin{equation}
\begin{aligned}
\widetilde{\ev}_{Y} &\colon \MM_{H, k, \R \times Y; k_-, k_+}^*\left((A_1,\ldots,A_N);J_Y\right) \to Y^{2N} \\
&(\tilde v_1, \ldots, \tilde v_N)  \mapsto \left(\lim_{s\to - \infty} v_1(s, 1), \lim_{s\to +\infty} v_1(s, 1), 
		\ldots, \lim_{s\to - \infty} v_N(s, 1), \lim_{s\to +\infty}
            v_N(s, 1)\right) .
\end{aligned}
\label{ev_Y}
\end{equation}

If $(\tilde v_0, \tilde v_1, \ldots, \tilde v_N) \in \MM_{H, k, W}^*\left((B;
A_1,\ldots,A_N);J_W\right)$,
we have 
\begin{equation}
\begin{aligned}
\widetilde{\ev}_{W, Y} &\colon
	\MM_{H, k, W; k_+}^*\left((B; A_1,\ldots,A_N);J_W\right) 
\to W \times Y^{2N+1} \\
&(\tilde v_0, \tilde v_1, \ldots, \tilde v_N)  \mapsto \left(\tilde v_0(0), \lim_{r \to +\infty}
\pi_Y \tilde v_0(r + i 0), 
	\lim_{s\to - \infty} v_1(s, 1),  \ldots,  \lim_{s\to +\infty} v_N(s,
    1)\right) .
\end{aligned}
\label{ev_W,Y}
\end{equation}

These maps are $C^1$ smooth, which follows from exploiting the asymptotic expansion of a 
Floer cylinder near its asymptotic limit, as described by \cite{SiefringAsymptotics}. 
Details for this are given in \cite{Fish_Siefring_Connected_sums_2}.

We also have augmentation evaluation maps. For each puncture $z_0 \in \Gamma$,
there exists an index $i \in \{ 1, \dots, N\}$ so that the augmentation 
puncture $z_0$ is a puncture in the domain of $v_i$. For this augmentation puncture, we have
the asymptotic evaluation map $v_i \mapsto \lim_{z \to z_0} \pi_\Sigma( v_i( z)) \in \Sigma$.
Combining all of these evaluation maps over all punctures in $\Gamma$, we obtain 
\begin{align*}
\widetilde{\ev}^a_{\Sigma} &\colon \MM_{H, k, \R \times Y; k_-, k_+}^*\left((A_1,\ldots,A_N);J_Y\right) \to \Sigma^k \\
\widetilde{\ev}^a_{\Sigma} &\colon \MM_{H, k, W; k_+}^*\left((B; A_1,\ldots,A_N);J_W\right) \to \Sigma^k.
\end{align*}
Note that these maps are $C^1$ smooth, either by \cite{Fish_Siefring_Connected_sums_2}
or by combining \cite{WendlSuperRigid}*{Proposition 3.15} with the smoothness
for the evaluation map for closed spheres.

Define the {\em flow diagonal} in $Y \times Y$ to be 
\begin{equation*}
\widetilde{\Delta}_{f_Y} \coloneqq 
\big\{ (x, y) \in
        (Y\setminus \Crit(f_Y))^{2} \, : \,  \exists {t> 0} \text{ s.t. } \varphi^{t}_{Z_Y}(x) = y \big\}
\end{equation*}
where $\Crit(f_Y)$ is the set of critical points of $f_Y$. 

 Let $\tilde p, \tilde q \in Y$ be critical points of $f_Y$ and let $W_Y^u(\tilde p)$, $W_Y^s(\tilde q)$ be 
 the unstable/stable manifolds of $\tilde p, \tilde q$, as in \eqref{(un)stable}.

We may now describe the moduli space of simple split Floer cylinders from $\tilde q_{k_-}$ to $\tilde p_{k_+}$ as 
the unions of the fibre products of these moduli spaces under the asymptotic evaluation maps 
and augmentation evaluation maps. For notational convenience, we write
\begin{equation}
\begin{aligned}
\widetilde{\ev}  \colon 
	\MM_{H, k, \R \times Y; k_-, k_+}^*\left( (A_1,\ldots,A_N);J_Y\right) 
	&\times 
	\M^*_X( (B_1, \dots, B_k); J_W) 
	\to 
	Y^{2N} \times \Sigma^{k} \times \Sigma^k \\
&(\mathbf{\tilde v}, \mathbf{v}) \mapsto \left ( \widetilde \ev_Y( \mathbf{\tilde v}), \widetilde{\ev}_\Sigma^a(\mathbf{\tilde v}), \ev_\Sigma^a(\mathbf{v}) \right ).
\end{aligned}
\label{ev_Y aug}
\end{equation}

Write $\Delta_{\Sigma^k} \subset \Sigma^k \times \Sigma^k$ to denote the diagonal $\Sigma^k$.
Then, define
\begin{equation*}
\MM_{H}^*(\tilde q_{k_-},\tilde p_{k_+};  (A_1,\ldots,A_N), (B_1, \dots, B_k)  ; J_W) 
	= \widetilde{\ev}^{-1} \left ( 
W_Y^s(\tilde q) \times \left ( \tilde \Delta_{f_Y} \right )^{N-1} \times W^u_{Y}(\tilde p)
\times
\Delta_{\Sigma^k}		
	\right ). 
\end{equation*}

From this, we have 
\begin{equation}
\begin{split}
\MM_{H, N}^*(\tilde q_{k_-},\tilde p_{k_+};J_W)& = \\
	 \bigcup_{(A_1,\ldots,A_N)} \bigcup_{k \ge 0} \bigcup_{(B_1, \dots, B_k)} 
&\MM_{H}^*(\tilde q_{k_-},\tilde p_{k_+};  (A_1,\ldots,A_N), (B_1, \dots, B_k)
         ; J_W) .
\end{split}
\label{union} 
\end{equation}

Similarly, if $x \in W$ is a critical point of $f_W$, 
and letting $W_W^u(x)$ be the descending manifold of $x$ in $W$ 
for the gradient-like vector field $-Z_W$, we define
\begin{align*}
\widetilde{\ev}  \colon 
	\MM_{H, k, W;  k_+}^*\left( (B; A_1,\ldots,A_N);J_W \right) 
	&\times 
	\M^*_X( (B_1, \dots, B_k); J_W) 
	\to 
	W \times Y^{2N+1} \times \Sigma^{k} \times \Sigma^k \\
&((\tilde v_0, \mathbf{\tilde v}), \mathbf{v}) \mapsto \left (\widetilde \ev_{W,Y}(\tilde v_0, \mathbf{\tilde v}), \widetilde{\ev}_\Sigma^a(\mathbf{\tilde v}), \ev^a_\Sigma(\mathbf{v}) \right ).
\end{align*}
Then, define
\begin{equation*}
\begin{split}
\MM_{H}^*(x,\tilde p_{k_+};  (B; A_1,\ldots,A_N),& (B_1, \dots, B_k)  ; J_W) 
	=\\
	 \widetilde{\ev}^{-1} &\left (
W_W^u(x) \times \widetilde{\Delta} \times \left ( \widetilde \Delta_{f_Y} \right )^{N-1} \times W^u_{Y}(\tilde p)
\times
\Delta_{\Sigma^k}
	\right ). 
\end{split}
\end{equation*}

Finally, we obtain 
\begin{equation}
\begin{split}
\MM_{H,N}^*(x,\tilde p_{k_+};J_W) &= \\
\bigcup_{(B; A_1,\ldots,A_N)} \bigcup_{k \ge 0} &\bigcup_{(B_1, \dots, B_k)}
\MM_{H}^*(x,\tilde p_{k_+};  (B; A_1,\ldots,A_N), (B_1, \dots, B_k)  ; J_W) .
\end{split}
\label{union2} 
\end{equation}

In order to establish transversality for our moduli spaces, it then becomes
necessary to show transversality of the evaluation maps to these products of descending/ascending manifolds, diagonals
and flow diagonals.
Recall the space of almost complex structures $\J_W^{reg}$ given in Proposition \ref{ev Sigma submersion bis}. 
We denoted by $\J_Y^{reg}$ the space of cylindrical almost complex structures on $\R\times Y$ obtained from 
restrictions of elements in $\J_W^{reg}$. The following result will provide the final step in the proof of Proposition \ref{cascades form manifold}. 

\begin{proposition}
\label{prop:TransverseLiftedEvaluation}
Let $J_W \in \J_W^{reg}$ and let $J_Y \in \J_Y^{reg}$ be the induced almost complex structure on $\R \times Y$. 

Let $\tilde q, \tilde p$ denote critical points of $f_Y$, and let $x$ be a critical point of $f_W$ in $W$.
Let $k_+$ and $k_-$ be non-negative multiplicities, $k_+ > k_-$.

Let $A_1, \dots, A_N$ be spherical homology classes in $\Sigma$,
 let $B, B_1, \dots, B_k$ be spherical homology classes in $X$, $k \ge 0$.

Let $\tilde \Delta \subset Y \times Y$ and
$\Delta_{\Sigma^k} \subset \Sigma^k \times \Sigma^k$ be the diagonals.

Then, 
\begin{enumerate}
 \item the evaluation map
 \[
\widetilde{\ev}_{Y}
\times
\widetilde{\ev}_\Sigma^a \times \ev^a_\Sigma 
\colon 
\MM_{H, k, \R \times Y; k_-, k_+}^*\left((A_1,\ldots,A_N);J_Y\right) 
\times 
\M^*_X( (B_1, \dots, B_k); J_W) 
\to 
Y^{2N} 
\times
\Sigma^k \times \Sigma^k
\]
is transverse to the submanifold
\[
    W_Y^s(\tilde q) \times \left ( \tilde \Delta_{f_Y} \right )^{N-1} \times W^u_{Y}(\tilde p) 
    \times \Delta_{\Sigma^k}.
\]
 \item the evaluation map
     \[
 \widetilde{\ev}_{W,Y} 
 \times 
\widetilde{\ev}_\Sigma^a \times \ev^a_\Sigma 
 \colon \MM_{H, k, W; k_+}^*\left((B; A_1,\ldots,A_N);J_W\right) 
\times 
\M^*_X( (B_1, \dots, B_k); J_W) 
 \to W \times Y^{2N+1}
 \times
\Sigma^{k} \times \Sigma^k 
 \]
is transverse to the submanifold
\[
    W_W^u(x) \times \tilde \Delta \times \left ( \tilde \Delta_{f_Y} \right )^{N-1} \times W^u_{Y}(\tilde p)
    \times \Delta_{\Sigma^k}.
\]

\end{enumerate}

\end{proposition}

In order to prove this proposition, we will need a better description of the relationship
between the moduli spaces of spheres in $\Sigma$,
and the moduli spaces of Floer cylinders in $\R \times Y$ (or in $W$).
\begin{lemma} \label{L:proj_to_Sigma_submersion}
The maps
\begin{align*}
\pi_\Sigma^\M &\colon \MM_{H, k, \R \times Y; k_-,
k_+}^*\left((A_1,\ldots,A_N);J_Y\right) \to \M_{k, \Sigma}^*((A_1, \dots, A_N);
J_\Sigma)\\
\pi_\Sigma^\M &\colon \MM_{H, k, W; k_+}^*\left((B; A_1,\ldots,A_N);J_W\right)
\to \M_{k, X, \Sigma}^*((B; A_1, \dots, A_N); J_W)
\end{align*}
induced by $\pi_\Sigma \colon \R \times Y \to \Sigma$ are submersions. The
fibres have a locally free $(S^1)^N$ torus action by constant rotation by the
action of the Reeb vector field.
\end{lemma}
\begin{proof}
We will study the case of 
\[
\pi_\Sigma^\M \colon \MM_{H, k, \R \times Y}^*\left((A_1,\ldots,A_N);J_Y\right) \to \M_{k, \Sigma}^*((A_1, \dots, A_N); J_\Sigma)
\]
in detail. The case with a sphere in $X$ follows by the same argument with a small notational change.
It also suffices to consider the case with $N=1$, since moduli spaces with more
spheres are open subsets of products of these.

Suppose $\pi_\Sigma( \tilde v) = w$ with $\tilde v \in \MM_{H, k, \R\times Y; k_-, k_+}^*(A; J_Y)$ and $w \in \M_{k, \Sigma}^*(A; J_\Sigma)$. 
Recall the splitting of the linearized Floer operator at $\tilde v$, given in Lemma \ref{L:linearizationFloer} as 
\[
D_{\tilde v} = \begin{pmatrix} D_{\tilde v}^\C & M \\ 0 & D_{w}^\Sigma
\end{pmatrix}.
\]

By definition, $T_{w} \M_{k, \Sigma}^*(A) = \ker D_{w}^\Sigma$ and 
$T_{\tilde v} \MM_{H, k, \R\times Y; k_-, k_+}^*(A; J_Y) = \ker D_{\tilde v}$.
By Lemma \ref{L:verticalOperatorTransverse}, $D_{\tilde v}^\C$ is surjective. It follows then that any section $\zeta_0$ of $w^*T\Sigma$ that is in the kernel of $D_w^\Sigma$ can be lifted to a section $(\zeta_1, \zeta_0)$ of $\tilde v^*T\Sigma \cong (\R \oplus \R R) \oplus w^*T\Sigma$ that is in the kernel of $D_{\tilde v}$. 

Notice now that $d\pi_\Sigma(\zeta_1, \zeta_0) = \zeta_0$, establishing that the
evaluation map is a submersion.

Also observe that $S^1$ acts on the curve $\tilde v$ by the Reeb flow. 
By the Reeb invariance of $J_Y$ and of the admissible Hamiltonian $H$,
the rotated curve is in the same fibre of $\pi^\M_\Sigma$. Furthermore, for
small rotation parameter, the curve will be distinct (as a parametrized curve)
from $\tilde v$.
\end{proof}

The next result justifies why it was reasonable to assume $k_+ > k_-$ in Proposition \ref{prop:TransverseLiftedEvaluation}. The fact that $k_+\neq k_-$ will also be used below.

\begin{lemma} \label{different multiplicities}
Let $A \coloneqq [w] \in H_2(\Sigma;\Z)$, where $w\colon \CP^1 \to \Sigma$ is
the continuous extension of $\pi_\Sigma \circ \tilde v$. 
Assume that either $A\neq 0$ or $\Gamma \neq \emptyset$.  
Then, $k_+ > k_-$. 

\end{lemma}

\begin{proof}
Denote by $w^*Y$ the pullback under $w$
of the $S^1$-bundle $Y\to \Sigma$.
 The map $\tilde v$ gives a section $s$ of $w^*Y$, defined in
 the complement of $\Gamma\cup\{0,\infty\}$. By 
 \cite{BottTu}*{Theorem 11.16}, the Euler number $\int_{\CP^1}e(w^*Y)$ (where $e$ is
 the Euler class) is the sum of the local degrees of the section $s$
 at the points in $\Gamma\cup\{0,\infty\}$.

Denote the multiplicities of the periodic $X_H$-orbits $x_\pm (t) =
\lim_{s\to \pm\infty} v(s,t)$ by $k_\pm$, respectively, and denote
the multiplicities of the asymptotic Reeb orbits at the punctures
$z_1,\ldots,z_m \in \Gamma$ by $k_1,\ldots,k_m$, respectively.
The positive integers  $k_\pm$ and $k_i$ are the absolute values of the
degrees of $s$ at the respective points. Taking signs into account, we get 
$$
\int_{\CP^1}e(w^*Y) =  k_+ - k_-  - k_1 - \ldots - k_m.  $$ 
We will show that this quantity is non-negative. We have 
$$ 
\int_{\C
P^1}e(w^*Y) = \int_{\CP^1}w^*e(Y\to \Sigma) = \int_{\CP^1}w^*e(N\Sigma)
$$ 
where
$N\Sigma$ is the normal bundle to $\Sigma$ in $Y$. Now, $e(N\Sigma) =
s^* \Th(N\Sigma)$, where $s \colon \Sigma \to N\Sigma$ is the zero section and
$\Th(N\Sigma)$ is the Thom class of $N\Sigma$ \cite{BottTu}*{Proposition 6.41}. If $j\colon N\Sigma \to X$ is a tubular neighborhood, then $j_*\Th(N\Sigma) = PD([\Sigma]) = [K\omega] \in H^2(X;\R)$ \cite{BottTu}*{Equation (6.23)}. 
If $\iota\colon \Sigma \hookrightarrow X$ is the inclusion, then 
\begin{align*} 
\int_{\CP^1}w^*e(N\Sigma) &= \int_{\CP^1}w^*s^*\Th(N\Sigma) = \int_{\CP^1}w^*\iota^* j_* \Th(N\Sigma) = \\
&= \int_{\CP^1}w^*\iota^* K \omega = K \omega(A) \geq 0 
\end{align*}
since $K>0$ and $w$ is a $J_\Sigma$-holomorphic sphere. We
conclude that $$ k_+ - k_- -k_1 - \ldots - k_m = K \omega(A) \geq 0.$$
If $A\neq 0$, we get a strict inequality. If $A=0$, we get an equality, but the assumptions of the Lemma imply that $\sum_{i=1}^m
 k_i > 0$. In either case, we get $k_+ > k_-$, as wanted.
\end{proof}

Recall that the gradient-like vector field $Z_Y$ has the property that $d\pi_\Sigma Z_Y = Z_\Sigma$.
Also recall that we may use the contact form $\alpha$ as a connection to lift vector fields from $\Sigma$ to 
vector fields on $Y$, tangent to $\xi$. If $V$ is a vector field on $\Sigma$, we write 
$\pi_\Sigma^*V \coloneqq \widetilde V $ to be the vector field on $Y$ uniquely determined 
by the conditions $\alpha(V) = 0$, $d\pi_\Sigma \widetilde V = V$.
This extends as well to lifting vector fields on $\Sigma \times \Sigma$ to vector fields on $Y \times Y$.
\begin{lemma}
The flow diagonal in $Y$ satisfies
\[
    \pi_\Sigma( \widetilde \Delta_{f_Y} ) \subset \Delta_{f_\Sigma} \cup  \{
    (p, p)\, | \, p \in \Crit(f_\Sigma) \}.
\]

Let $(\tilde x, \tilde y) \in \widetilde \Delta_{f_Y}$ and $x = \pi_\Sigma(\tilde x)$, $y=\pi_\Sigma(\tilde y)$. Let $t>0$ be so that $\tilde y = \varphi^{t}_{Z_Y}(\tilde x)$.
Then, if $x=y$, we have $x \in \Crit(f_\Sigma)$ and
\begin{equation} \label{E:tangent_to_flow_diagonal_1}
\begin{aligned}
    T_{(\tilde x, \tilde y)} \widetilde \Delta_{f_Y} &= \{ (a R + v, bR+ 
    \pi_\Sigma^*{ d\varphi^t_{Z_\Sigma} d\pi_\Sigma v})  \in TY \oplus TY \, |
\, a, b \in \R \text{ and } \alpha(v) = 0 \}.
\end{aligned}
\end{equation}

If $x \ne y$, then $(x, y) \in \Delta_{f_\Sigma}$.
Then, there exists a positive $g = g(\tilde x, \tilde y) > 0$ so that 
\begin{equation} \label{E:tangent_to_flow_diagonal_2}
\begin{aligned}
T_{(\tilde x, \tilde y)} \widetilde \Delta_{f_Y} 
&= \R (R, gR) \oplus H 
\end{aligned}
\end{equation}
where the subspace $H$ is such that $d\pi_\Sigma|_H \colon H \to  T\Delta_{f_\Sigma}$ induces a linear
isomorphism.
\end{lemma}
\begin{proof}
Observe first that if $x = \pi_\Sigma \tilde x$, we have
\[
    \pi_\Sigma \varphi^t_{Z_Y}( \tilde x) = \varphi^t_{Z_\Sigma}( x ).
\]
This gives $d\pi_\Sigma \d \varphi^t_{Z_Y}(\tilde x) = d \varphi_{Z_\Sigma}^t
d\pi_\Sigma(\tilde x)$. From this, it follows that $d\varphi_{Z_Y}^t(\tilde x)R$
is a multiple of the Reeb vector field.
Observe also that $\varphi^t_{Z_\Sigma}$ and $\varphi^t_{Z_Y}$ are both
orientation-preserving diffeomorphisms for all $t$. We therefore obtain 
that if $y = \varphi^t_{Z_\Sigma}(x)$, 
$\varphi^t_{Z_Y}$ induces a diffeomorphism 
between the fibres $\pi_\Sigma^{-1}(x) \to \pi_\Sigma^{-1}(y)$.
Additionally, we must have then that $d\varphi_{Z_Y}^t(\tilde x)R$ is a positive
multiple of the Reeb vector field. Let $g(\tilde x, \tilde y) > 0$ such that 
$d\varphi_{Z_Y}^t(\tilde x)R = g(\tilde x, \tilde y) R$. 

In general, if $\tilde y = \varphi^t_{Z_Y}(\tilde x)$, we have
\begin{equation} \label{E:tangent_space_tilde_Delta}
T_{(\tilde x, \tilde y)} \widetilde \Delta_{f_Y} 
= 
\{ (v, d\varphi^t_{Z_Y}(\tilde x)v + c Z_Y(\tilde y) ) \, | \, v \in T_x Y, c \in \R \}.
\end{equation}

Consider first the case of $x=y$. 
	Then, both $\tilde x$ and $\tilde y$ are in the same fibre of $Y \to \Sigma$. By definition of the flow diagonal, there exists $t > 0$ so that $\varphi^t_{Z_Y}(\tilde x) = \tilde y$, and hence $Z_Y$ is vertical, $Z_\Sigma(x) = 0$. It follows 
that $x \in \Crit_{f_\Sigma}$.  
From this, it now follows that $\pi_\Sigma( \widetilde \Delta_{f_Y}) \subset
\Delta_{f_\Sigma} \cup \{ (p, p)\, | \, p \in \Crit(f_\Sigma) \}$.

We now consider the consequences of Equation
\eqref{E:tangent_space_tilde_Delta} in this case of $x=y$. Any $v \in T_xY$ may be written as $v_0 +
aR$ where $\alpha(v_0) = 0$. Furthermore, since $x=y \in \Crit(f_\Sigma)$, and
by definition, neither $\tilde x$ nor $\tilde y$ are critical points of $f_Y$, we
obtain that $Z_Y(\tilde y)$ is a non-zero multiple of the Reeb vector field. 
Equation \eqref{E:tangent_to_flow_diagonal_1} now follows from the fact that
$d\pi_\Sigma \varphi^t_{Z_Y}(\tilde x) = d\varphi^t_{Z_\Sigma}(x) d\pi_\Sigma$.

We now consider when $x \ne y$. Let $H = \{ (v, d\varphi^t_{Z_Y}(\tilde x)v + c
    Z_Y(\tilde y)\, | \, \alpha(v) = 0 \}$. Then, \[
    d\pi_\Sigma(H) = \{ (v, d\varphi^t_{Z_\Sigma}(x) v + c Z_\Sigma ) \, | \, v
    \in T_x \Sigma\} = T \Delta_{f_\Sigma}.
    \]
    By assumption, $y$ is not a critical point of $f_\Sigma$, so $d \pi_\Sigma$
    induces an isomorphism.
The decomposition of $T \widetilde \Delta_{f_Y}$ now follows immediately from
the definition of $g$ and from Equation \eqref{E:tangent_space_tilde_Delta}.
\end{proof}

\begin{proof}[Proof of Proposition \ref{prop:TransverseLiftedEvaluation}]

    We consider first the case of 
$$\widetilde{\ev}_{Y} \colon \MM_{H, k, \R \times Y; k_-,
k_+}^*\left((A_1,\ldots,A_N);J_Y\right) \to Y^{2N}. $$

    Suppose that $\tilde{\mathbf{v}} = (\tilde v_1, \dots, \tilde v_N) \in 
\MM_{H, k, \R \times Y; k_-, k_+}^*\left((A_1,\ldots,A_N);J_Y\right)$.
For each $i=1, \dots N$,
let $\tilde y_{i} = \tilde v_i(-\infty, 0) \in Y$ and $\tilde x_i = \tilde
    v_i(+\infty, 0) \in Y$, with 
    \begin{align*}
        &\tilde y_1 \in W^s_Y(\tilde q), \tilde x_N \in W^u_Y(\tilde p)\\
        &(\tilde x_i, \tilde y_{i+1}) \in \widetilde \Delta_{f_Y} \quad \text{ for } 1 \le i \le
        N-1.
    \end{align*}
    Let $w_i = \pi_\Sigma(\tilde v_i)$ and $x_i = \pi_\Sigma(\tilde x_i)$, $y_i
    = \pi_\Sigma(\tilde y_i)$. Then, it follows that
    \begin{align*}
        &y_1 \in W^s_\Sigma(q), x_N \in W^u_\Sigma(p)\\
        &(x_i, y_{i+1}) \in \Delta_{f_\Sigma} \cup \{ (p, p) \, | \, p \in
    \Crit(f_\Sigma) \} \quad \text{ for } 1 \le i \le N-1.
    \end{align*}
Let $S \subset \Sigma^{2N-2}$ be the appropriate product of a number of copies
of $\Delta_{f_\Sigma}$ and of $\{ (p,p) \, | \, p \in \Crit(f_\Sigma) \}$. By
Proposition \ref{ev Sigma submersion bis}, the evaluation map on
$\M_\Sigma((A_1, \dots, A_N); J_Y)$ is transverse to $S$.

Then, by the previous Lemma,
    \[
    TS \subset d\pi_\Sigma \left ( 
    T_{\tilde y_0} W^s_Y(\tilde q) 
        \times 
        T_{(\tilde x_1, \tilde y_2)} \widetilde{\Delta}_{f_Y} \times 
        \dots 
        \times T_{(\tilde x_{N-1}, \tilde y_{N})} \widetilde{\Delta}_{f_Y} \times 
        T_{\tilde x_N} W^u_Y(\tilde p)
        \right ).
    \]

It suffices therefore to obtain transversality in the vertical direction. Notice
that by rotating by the action of the Reeb vector field on $\tilde v_i$, we
obtain that the image of $d\tilde \ev$ contains the subspace 
\[
    \{ (a_1R, a_1R, a_2R, a_2R, \dots, a_{N}R, a_{N}R )\, | \, (a_1, \dots, a_N)
    \in \R^{N} \} \subset (TY)^{2N}.
\]

In the case of the chain of pearls in $\Sigma$, each of the spheres $w_i, i=1,
\dots, N$ must either be non-constant or have a non-trivial collection of
augmentation punctures. Then, 
by Lemma \ref{different multiplicities}, each punctured cylinder $\tilde v_i$
has different multiplicities $k_i^+, k_i^-$ at $\pm \infty$, and thus the action
of rotating the domain marker gives that the image of $d\tilde \ev_Y(\tilde
v_i)$ contains $(k_-R, k_+R) \in T_{\tilde y_{i}} Y \oplus T_{\tilde x_i} Y$.
While this holds for each $i=1, \dots, N$, we only require such a vector for one
cylinder. Then, by taking this in the case of $i=1$, we see that the following
$N+1$ vertical vectors in $(\R R)^{2N} \subset TY^{2N}$ 
are in the image of the linearized evaluation map (the first two obtained by
combining the two Reeb actions on $\tilde v_1$, the remainder by the Reeb action
on $\tilde v_i$, $i\ge 2$):
\begin{align*}
    &(R, 0, 0, \dots, 0),\\
    &(0, R, 0, \dots, 0),\\
    &(0, 0, R, R, 0, 0, \dots, 0), \\
    &(0, 0, 0, 0, R, R, 0, \dots, 0),\\
    & \dots \\
    &(0, 0, \dots, 0, R, R).
\end{align*}

By the previous Lemma, the tangent space $T_{(\tilde x_i, \tilde y_{i+1})}
\widetilde \Delta_{f_Y}$ contains at least the vertical vector $(R, g_iR)$, where
$g_i \coloneqq g(\tilde x_i, \tilde y_{i+1}) > 0$, for each $1 \le i \le N-1$. In the vertical
direction, this then contains the following $N-1$ vectors:
\begin{align*}
    &(0, R, g_1R, 0, \dots, 0),\\
    &(0, 0, 0, R, g_2R, 0, \dots, 0) \\
    &\dots \\
    &(0, \dots, 0, R, g_{N-1}R, 0).
\end{align*}

We now observe that this collection of $2N$ vectors spans 
$(\R R)^{2N}$. This establishes that $\tilde\ev_Y$ defined on 
$\MM_{H, k, \R \times Y; k_-, k_+}^*\left((A_1,\ldots,A_N);J_Y\right)$ is transverse to 
$W^s_Y(\tilde q) \times \widetilde{\Delta}_{f_Y}^N \times W^u_Y(\tilde p)$.

We now consider 
the case of 
\[
   \widetilde{\ev}_{W,Y} \colon \MM_{H, k, W; k_+}^*\left((B;
   A_1,\ldots,A_N);J_W\right) \to W \times Y^{2N+1}.
   \]
We will show this evaluation map is transverse to 
$$\tilde S \coloneqq W_W^u(x) \times \tilde \Delta \times \left ( \tilde \Delta_{f_Y} \right )^{N-1} \times W^u_{Y}(\tilde p).$$

As before, it suffices to show transversality in a vertical direction, since, by
Proposition \ref{ev Sigma submersion bis}, the projections to $X$, $\Sigma$ are
transverse. More precisely, let $S \subset W \times \Sigma \times \Sigma^{2N}$ be of the form
$S = W_W^u(x) \times \Delta \times S' \times W^u_\Sigma(p)$, where 
$S' \subset \Sigma^{2N-2}$ is a product of some number of
$\Delta_{f_\Sigma}$ and of $\{ (p, p) \, | \, p \in \Crit(f_\Sigma) \}$ so that 
$TS \subset T d\pi_\Sigma( \tilde S)$.
Proposition \ref{ev Sigma submersion bis} gives transversality of $\widetilde{\ev}_{W,Y}$ to $S$.

Notice that the tangent space $T \tilde S$ contains at least the following
vertical vectors (we put $0$ in the first component since $TW$ has no vertical
direction):
\begin{align*}
    &(0, R, R, 0, 0, \dots, 0) \\
    &(0, 0, 0, R, g_1R, 0, \dots, 0) \\
    &\dots \\
    &(0, \dots, 0, R, g_{N-1}R, 0).
\end{align*}

Let $(\tilde v_1, \tilde v_1, \dots, \tilde v_N) \in \MM_{H, k, W; k_+}^*$. 
The plane $\tilde v_1$ converges to a Reeb orbit of multiplicity $l = B \bullet \Sigma$.
Observe that domain rotation on the plane $\tilde v_1$ then gives that $(0, lR, 0, \dots,
0) \in TW \oplus TY \oplus TY^{2N}$ is in the image of $d\widetilde{\ev}_{W,Y}$. 

As before, the Reeb rotation on each of the punctured cylinders $\tilde v_1,
\dots, \tilde v_N$ gives that the following vertical vectors are in the image of
$d\widetilde{\ev}_{W,Y}$:
\begin{align*}
   &(0,0, R, R, 0, 0, \dots, 0, 0) \\
   &(0,0, 0, 0, R, R, \dots, 0, 0) \\
   &(0,0, 0, 0, \dots, 0, R, R).
\end{align*}
We notice then that these vectors span $0 \oplus (\R R)^{2N-1}$, so it 
follows that the evaluation map is transverse to $\tilde S$.

Finally, the transversality of the evaluation maps at augmentation punctures
comes from the fact that 
the augmentation evaluation maps
\begin{align*} 
\widetilde{\ev}_\Sigma^a \times \ev^a_\Sigma \colon
\MM_{H, k, \R \times Y; k_-, k_+}^*\left( (A_1,\ldots,A_N);J_Y\right) 
\times 
\M^*_X( (B_1, \dots, B_k); J_W) 
&\to
\Sigma^{k} \times \Sigma^k \\
\widetilde{\ev}_\Sigma^a \times \ev^a_\Sigma \colon
\MM_{H, k, W; k_+}^*\left( (B; A_1,\ldots,A_N);J_W\right) 
\times 
\M^*_X( (B_1, \dots, B_k); J_W) 
&\to
\Sigma^{k} \times \Sigma^k 
\end{align*}
factor through the evaluation maps 
\begin{align*}
&\ev^a\times \ev^a  \colon \M^*_{k, \Sigma}((A_1, \dots, A_N); J_\Sigma) \times \M^*_X((B_1, \dots, B_k); J_W) \to \Sigma^{2k} \quad \text{ and }\\
&\ev^a \times \ev^a \colon \M^*_{k, (X, \Sigma)}((B; A_1, \dots, A_N); J_W) \times
\M^*_X((B_1, \dots, B_k); J_W) \to \Sigma^{2k}.
\end{align*}
The required transversality for these maps is given by Proposition \ref{ev Sigma
submersion bis}.
Furthermore, these
evaluation maps are invariant under the domain and Reeb rotations used
to obtain transversality for $\widetilde{\ev}_Y$ and for
$\widetilde{\ev}_{W,Y}$
in the vertical directions, so the transversality follows immediately.  
\end{proof}

\section{Monotonicity and the differential}

The results of the previous section show that the moduli spaces of
Floer cylinders with cascades that project to \textit{simple} chains of pearls are
transverse. 

We now impose monotonicity conditions on $(X, \omega)$ and on $(\Sigma,
\omega_\Sigma)$ in order to show that these moduli spaces are sufficient for
the purposes of defining the split Floer differential.
We suppose $(X, \omega)$ is spherically monotone, so there exists a constant
$\tau_X > 0$ with $\langle c_1(TX), A \rangle = \tau_X \omega(A)$ for every
spherical homology class $A$.
Also, if we let $K > 0$ such that $A \bullet \Sigma = K\omega(A)$, then we
require $\tau_X - K > 0$. Observe that $\Sigma$ must be spherically monotone
if it has a spherical homology class, with 
monotonicity constant $\tau_\Sigma = \tau_X - K$.

\subsection{Index inequalities from monotonicity and transversality}
\label{sec:index_inequalities}

First, we consider the Fredholm index contributions of a plane in $W$ that could
appear as an augmentation plane, to obtain some bounds on the possible indices.

\begin{lemma} \label{L:indexAugmentationNonNeg}
If $v \colon \C \to W$ is a $J_W$ holomorphic plane asymptotic to a 
given closed Reeb orbit $\gamma$ in $Y$,
the Fredholm index for the deformations of $v$ 
(as an unparameterized curve) keeping $\gamma$ fixed is $|\gamma|_0$ and it is non-negative.
Furthermore, if $v$ is multiply covered, this Fredholm index is at least $2$.
\end{lemma}

\begin{proof}
The fact that the Fredholm index $\Ind(v)$ in the statement is given by $|\gamma|_0$ as in 
\eqref{E:grading_Reeb_orbits} can be seen using Theorem \ref{T:RiemannRochMB}. On the other hand,
thinking of $v$ as giving a $J_X$-holomorphic sphere in homology class $B\in H_2(X;\Z)$, with an order 
of contact $B\bullet \Sigma$ with $\Sigma$, we see that
\[
\Ind(v) = 2( \langle c_1(TX),B\rangle - B\bullet \Sigma -1 ) = 2( \tau_X \omega(B) - K \omega(B) - 1).
\]
Since the plane is holomorphic, the class $B$ has $\omega(B) > 0$. 
By our monotonicity assumptions, we have
$$\tau_X \omega(B) - K \omega(B) = (\tau_X - K) \, \omega(B) > 0.$$
Finally, since $\tau_X\omega(B)=\langle c_1(TX),B\rangle\in \Z$ and $K\omega(B)= B\bullet \Sigma \in \Z$, 
we have that $\tau_X \omega(B) - K \omega(B)>0$ is an integer, and is thus at least 1.

It therefore follows that $\Ind(v) \ge 0$.

Suppose now that $v$ is a $k$-fold cover of an underlying simple holomorphic
plane $v_0$, representing classes $B = kB_0$ and $B_0$, respectively. 
Then, 
\[
    \Ind(v)+2 = 2(\tau_X -K) \omega(kB_0) = k ( \Ind(v_0) + 2).
\]
Hence, $\Ind(v) \ge 2(k-1)$.  
\end{proof}

\begin{figure}
  \begin{center}
    \def\svgwidth{0.3\textwidth}

\begingroup
  \makeatletter
  \providecommand\color[2][]{%
    \errmessage{(Inkscape) Color is used for the text in Inkscape, but the package 'color.sty' is not loaded}
    \renewcommand\color[2][]{}%
  }
  \providecommand\transparent[1]{%
    \errmessage{(Inkscape) Transparency is used (non-zero) for the text in Inkscape, but the package 'transparent.sty' is not loaded}
    \renewcommand\transparent[1]{}%
  }
  \providecommand\rotatebox[2]{#2}
  \ifx\svgwidth\undefined
    \setlength{\unitlength}{153.65167524pt}
  \else
    \setlength{\unitlength}{\svgwidth}
  \fi
  \global\let\svgwidth\undefined
  \makeatother
  \begin{picture}(1,0.52879344)%
    \put(0,0){\includegraphics[width=\unitlength]{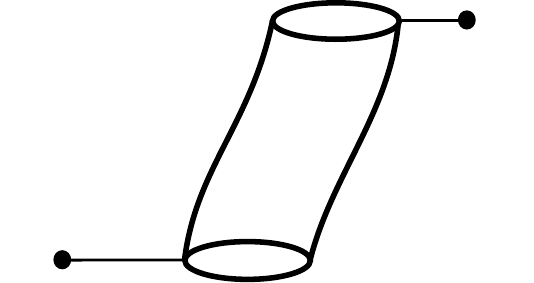}}%
    \put(-0.0634575,0.02310476){\color[rgb]{0,0,0}\makebox(0,0)[lb]{\smash{$\widehat q_{k_-}$}}}%
    \put(0.95149337,0.47648333){\color[rgb]{0,0,0}\makebox(0,0)[lb]{\smash{$\widecheck p_{k_+}$}}}%
    \put(0.52131919,0.24979405){\color[rgb]{0,0,0}\makebox(0,0)[lb]{\smash{$\tilde v$}}}%
    \put(0.902131919,0.164979405){\color[rgb]{0,0,0}\makebox(0,0)[lb]{\smash{$\R\times Y$}}}%
  \end{picture}%
\endgroup

   \end{center}
  \caption[Case 1]{Option (1) in Proposition \protect{\ref{P:Floer_in_R_times_Y}}, call it Case 1}
  \label{case_1_fig}
\end{figure}

\begin{figure}
  \begin{center}
    \def\svgwidth{0.25\textwidth}

\begingroup
  \makeatletter
  \providecommand\color[2][]{%
    \errmessage{(Inkscape) Color is used for the text in Inkscape, but the package 'color.sty' is not loaded}
    \renewcommand\color[2][]{}%
  }
  \providecommand\transparent[1]{%
    \errmessage{(Inkscape) Transparency is used (non-zero) for the text in Inkscape, but the package 'transparent.sty' is not loaded}
    \renewcommand\transparent[1]{}%
  }
  \providecommand\rotatebox[2]{#2}
  \ifx\svgwidth\undefined
    \setlength{\unitlength}{155.47384928pt}
  \else
    \setlength{\unitlength}{\svgwidth}
  \fi
  \global\let\svgwidth\undefined
  \makeatother
  \begin{picture}(1,1.34900323)%
    \put(0,0){\includegraphics[width=\unitlength]{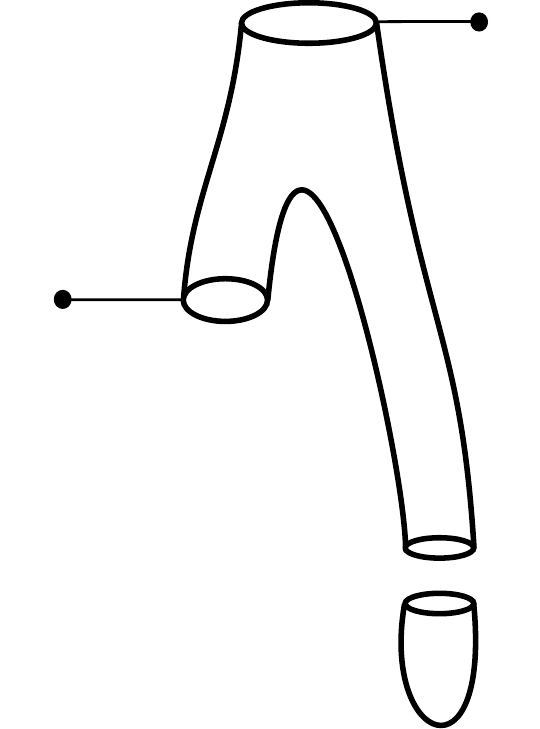}}%
    \put(-0.08341697,0.77649313){\color[rgb]{0,0,0}\makebox(0,0)[lb]{\smash{$\widehat p_{k_-}$}}}%
    \put(0.95206187,1.29539538){\color[rgb]{0,0,0}\makebox(0,0)[lb]{\smash{$\widecheck p_{k_+}$}}}%
    \put(0.54043248,1.11827007){\color[rgb]{0,0,0}\makebox(0,0)[lb]{\smash{$\tilde v$}}}%
    \put(0.6052953,0.09792872){\color[rgb]{0,0,0}\makebox(0,0)[lb]{\smash{$U$}}}%
    \put(0.95206187,0.77649313){\color[rgb]{0,0,0}\makebox(0,0)[lb]{\smash{$\R\times Y$}}}%
    \put(1.04206187,0.097649313){\color[rgb]{0,0,0}\makebox(0,0)[lb]{\smash{$W$}}}%
  \end{picture}%
\endgroup

   \end{center}
  \caption[Case 2]{Option (2) in Proposition \protect{\ref{P:Floer_in_R_times_Y}}, call it Case 2}
  \label{case_2_fig}
\end{figure}
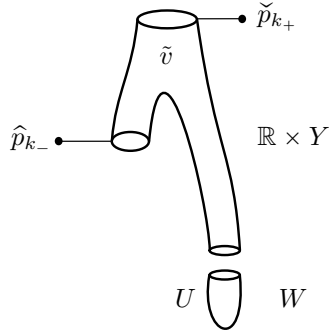

\begin{proposition} \label{P:Floer_in_R_times_Y}
Any Floer cascade appearing in the differential, connecting two periodic orbits in $\R \times Y$,  
must be one of the following configurations:
\begin{enumerate}
\item[(0)] An index $1$ gradient trajectory in $Y$ without any (non-constant) holomorphic components and without any augmentation punctures.
\item A smooth cylinder in $\R \times Y$ without any augmentation punctures 
and a non-trivial projection to $\Sigma$. 
The positive puncture converges to an orbit $\widecheck p_{k_+}$ and the negative puncture 
converges to an orbit $\widehat q_{k_-}$. The difference in multiplicites of the orbits
is given by $k_+-k_-= K \omega(A)$, where $A \in H_2(\Sigma;\Z)$ is the homology class represented 
by the projection of the cylinder to $\Sigma$. See Figure \ref{case_1_fig}. 
\item A cylinder with one augmentation puncture and whose projection 
to $\Sigma$ is trivial. The positive puncture converges to an orbit $\widecheck p_{k_+}$ 
and the negative puncture converges to an orbit $\widehat q_{k_-}$. The augmentation
plane has index $0$.
If $B \in H_2(X;\Z)$ is the class represented by the augmentation plane, then the
difference in multiplicities is given by $k_+-k_-=K \omega(B)$. Furthermore, $\widecheck
p$ and $\widehat q$ are critical points of $f_Y$ contained in the same fibre of
$Y \to \Sigma$, which we can write as $q=p$. See Figure \ref{case_2_fig}.
\end{enumerate}

\end{proposition}

\begin{proof}

Consider a cascade with $N$ levels and $k$ augmentation planes appearing in the
differential $d \widetilde p_{k_+} = \cdots + \widetilde q_{k_-} + \cdots$.
Let $A_1, \dots, A_N \in H_2(\Sigma)$ denote the homology classes of the projections to
$\Sigma$, let $B_1, \dots, B_k \in H_2(X)$ denote the homology classes
corresponding to the augmentation planes. Let $\gamma_i, i=1, \dots, k$ denote
the limits at the augmentation punctures, and let $k_i$ denote their
multiplicites. Let $A = \sum_{i=1}^N A_i$.

We therefore have $k_+ - k_- - \sum_{j=1}^k k_j = K \omega(A) = K \frac{\langle c_1(T\Sigma), A \rangle}{\tau_X - K}$. 
We also have $k_j = B_j \bullet \Sigma = K\omega(B_j)$. 
Notice then that $|\gamma_i|_0 = 
                            2 \langle c_1(TX), B_j \rangle 
                            - 2 B_j \bullet \Sigma 
                            -2$.

We therefore have
\begin{equation} \label{E:indexDifferential}
\begin{aligned}
    1 &= | \widetilde p_{k_+} | - | \widetilde q_{k_-} | \\
      &= i(\widetilde p) + M(p) - i(\widetilde q) - M(q) 
                    + 2\frac{\tau_X -K}{K} (k_+ - k_- - \sum_{j=1}^k k_j) 
                    + 2\frac{\tau_X -K}{K} \sum_{j=1}^k k_j \\
      &=  i(\widetilde p) + M(p) - i(\widetilde q) - M(q) 
                    + 2 \langle c_1(T\Sigma), A \rangle 
                    + 2k + \sum_{j=1}^k |\gamma_j|_0.
\end{aligned}
\end{equation}

By Lemma \ref{L:indexAugmentationNonNeg}, we have that for each $j=1, \dots, k$, 
$| \gamma_j|_0 \ge 0$.

Consider the chain of pearls in $\Sigma$ obtained by projecting the upper level
of this split Floer trajectory to $\Sigma$.
By Proposition \ref{necklaces are regular}, if this is a simple chain of pearls, it has Fredholm index 
\[
    I_\Sigma \coloneqq M(p) + 2 \langle c_1(T\Sigma), A \rangle - M(q) + N - 1 + 2k.
\]
If the chain of pearls is not simple, by monotonicity, we have that the index is at
least as large as the index of the underlying simple chain of pearls. 

Now let $N_0$ be the number of sub-levels that project to constant curves in $\Sigma$ and let $N_1$ be the 
number of sub-levels that project to non-constant curves in $\Sigma$, $N = N_0 + N_1$. 
Note that by the stability condition,
each cylinder that projects to a constant curve in $\Sigma$ must have at least one augmentation puncture, 
so $N_0 \le k$.

By transversality for simple chains of pearls (Proposition \ref{necklaces are
regular}), we obtain the inequality 
\[
    I_\Sigma \ge 2 N_1 + 2k
\]
by considering the 2-dimensional automorphism group for the $N_1$ non-constant
spheres and by considering the $2k$-parameter family of moving augmentation
marked points on the domains.

Combining with Equation \eqref{E:indexDifferential}, we obtain
\begin{align*}
    1 &= i(\widetilde p) - i(\widetilde q) + (I_\Sigma -N + 1)+ 
                     \sum_{j=1}^k |\gamma_j|_0 \\ 
    1 &\ge (i(\widetilde p) - i(\widetilde q)+1) + 2N_1  + 2k - N +
                     \sum_{j=1}^k |\gamma_j|_0 \\
    1 &\ge (i(\widetilde p) - i(\widetilde q)+1) + N_1  + k + (k - N_0) +
                     \sum_{j=1}^k |\gamma_j|_0.
\end{align*}
Observe now that each term on the right-hand-side of the inequality is
non-negative. In particular, there is at most one augmentation plane ($k\leq 1$)
and if there is one, it 
must have $ |\gamma_1|_0 = 2 \langle c_1(TX), B_1 \rangle - 2 B_1 \bullet \Sigma -2  = 0$
(so the augmentation plane cannot be multiply covered, by Lemma \ref{L:indexAugmentationNonNeg}).

We can further write 
\[
    1 \ge  (i(\widetilde p) - i(\widetilde q)+1) + N_1  + k + (k - N_0).
\]
Notice that $N_1 + 2k - N_0 \ge N$.

This inequality can be satisfied in one of the following ways:
\begin{enumerate}
    \item[(0)] $N=0$. Then, either $i(\widetilde p) = i(\widetilde q)$ 
        or  $\widetilde p = \check p$ and $\widetilde q = \hat q$. Since $N =0$,
        this is a pure Morse differential term.
    \item $N_1 = 1$, $N_0 = k = 0$ and $\widetilde p = \widecheck p$,
        $\widetilde q = \widehat q$.
        This case corresponds to a non-constant sphere in $\Sigma$ without any
        augmentation punctures.
    \item $N_1 = 0$, $k=1$, $N_0 = 1$, and $\widetilde p = \widecheck p$, $\widetilde q = \widehat q$.
        In this case, the Floer cylinder has one augmentation puncture, and
        projects to a constant in $\Sigma$, so $q=p \in \Sigma$.
\end{enumerate}
\end{proof}

\begin{figure}
  \begin{center}
    \def\svgwidth{0.2\textwidth}

\begingroup
  \makeatletter
  \providecommand\color[2][]{%
    \errmessage{(Inkscape) Color is used for the text in Inkscape, but the package 'color.sty' is not loaded}
    \renewcommand\color[2][]{}%
  }
  \providecommand\transparent[1]{%
    \errmessage{(Inkscape) Transparency is used (non-zero) for the text in Inkscape, but the package 'transparent.sty' is not loaded}
    \renewcommand\transparent[1]{}%
  }
  \providecommand\rotatebox[2]{#2}
  \ifx\svgwidth\undefined
    \setlength{\unitlength}{146.34692076pt}
  \else
    \setlength{\unitlength}{\svgwidth}
  \fi
  \global\let\svgwidth\undefined
  \makeatother
  \begin{picture}(1,1.81416906)%
    \put(0,0){\includegraphics[width=\unitlength]{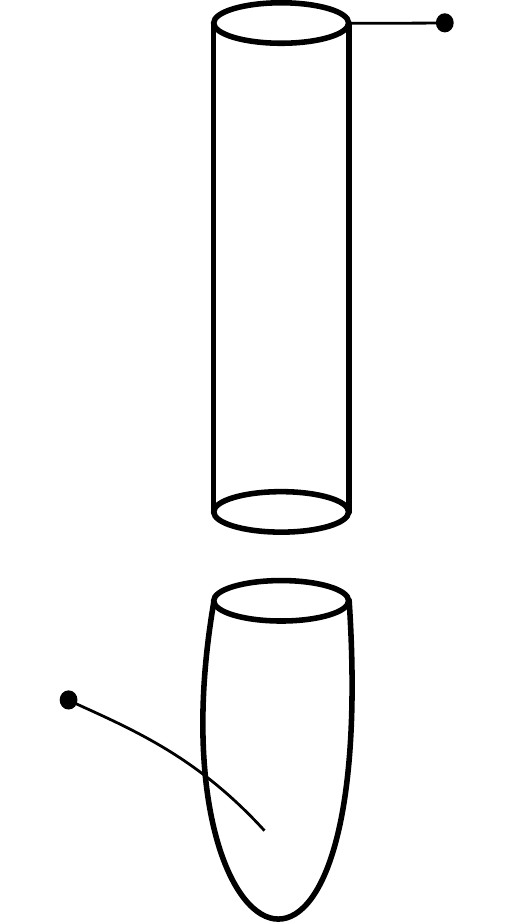}}%
    \put(-0.00144135,0.41693098){\color[rgb]{0,0,0}\makebox(0,0)[lb]{\smash{$x$}}}%
    \put(0.94907221,1.75491512){\color[rgb]{0,0,0}\makebox(0,0)[lb]{\smash{$\widecheck p_{k_+}$}}}%
    \put(0.5192222,1.27663133){\color[rgb]{0,0,0}\makebox(0,0)[lb]{\smash{$\tilde v_1$}}}%
    \put(0.5192222,0.3684972){\color[rgb]{0,0,0}\makebox(0,0)[lb]{\smash{$\tilde v_0$}}}%
    \put(0.94907221,1.12663133){\color[rgb]{0,0,0}\makebox(0,0)[lb]{\smash{$\R\times Y$}}}%
    \put(1.04907221,0.3284972){\color[rgb]{0,0,0}\makebox(0,0)[lb]{\smash{$W$}}}%
  \end{picture}%
\endgroup

   \end{center}
  \caption[Case 3]{Configuration as in Proposition \protect{\ref{P:Floer_also_in_W}}, call it Case 3} 
  \label{case_3_fig}
\end{figure}
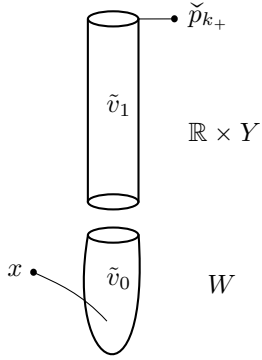

We now consider the possible terms in the differential that connect non-constant Hamiltonian 
trajectories in $\R \times Y$ to Morse critical points in $X$.

\begin{proposition} \label{P:Floer_also_in_W}
Any Floer cascade appearing in the differential, 
connecting a non-constant Hamiltonian orbit $\tilde p_{k_+}$ in $ \R \times Y$ to a Morse critical
point $x$ in $W$, consists of two levels. 
The upper level, in $ \R \times Y $, projects to a point in $\Sigma$ and 
is a cylinder asymptotic at $+\infty$ to an orbit $\widecheck p_{k_+}$ and at $-\infty$ to a Reeb orbit 
$\gamma$ in $\{ -\infty \} \times Y$. This $\gamma$ is the parametrized Reeb orbit associated to $\widecheck p_{k_+}$.

The lower level is a holomorphic plane in $W$ converging to the parametrized orbit $\gamma$ at $\infty$
and with $0$ mapping to the descending manifold of the critical point $x$. As a parametrized curve,
this has Fredholm index $1$. See Figure \ref{case_3_fig}.
\end{proposition}

\begin{proof}
Suppose such a cascade occurs in the differential, connecting the
non-constant orbit $\widetilde p_{k_+}$ to the critical point $x$ in the
filling $W$.

Let $N$ be the number of cylinders in $\R \times Y$ that appear in the split
Floer cylinder. Let $A_i \in H_2(\Sigma), i=1, \dots, N$, denote the spherical classes
represented by the projections of these cylinders to $\Sigma$. Let $A =
\sum_{i=1}^N A_i$.

Let $k$ be the number of augmentation planes, and let $B_j \in H_2(X), j=1, \dots, k$ be
the corresponding spherical homology classes in $X$. Let $\gamma_j, j=1, \dots,
k$, be the corresponding Reeb orbits with multiplicities $k_j = B_j \bullet \Sigma
= K \omega(B_j)$.

Let $B \in H_2(X)$ be the spherical homology class in $X$ represented by the
lower level $v_0$ in $W$, connecting to the critical point $x$.
Let $k_- = B \bullet \Sigma$ be the multiplicity of the orbit
to which the plane $v$ converges.
As before, we have 
\[
    k_+ - k_- - \sum_{j=1}^k k_j = K \omega(A).
\]

We then have 
\begin{equation} \label{E:differentialWithW}
\begin{aligned}
    1 &= | \widetilde p_{k_+}| - |x|  \\
      &= i(\widetilde p) + M(p) + 1-2n + M(x) 
            + 2\frac{\tau_X-K}{K} \left ( k_+ - k_- - \sum_{j=1}^k k_j 
            + k_- +  \sum_{j=1}^k k_j  \right ) \\ 
      &= i(\widetilde p) + M(p) + 1 - 2n + M(x) 
            + 2\langle c_1(T\Sigma), A \rangle 
            + 2\langle c_1(TX), B \rangle - 2 B \bullet \Sigma  
            + 2k + \sum_{j=1}^k |\gamma_j|_0 
\end{aligned}
\end{equation}

Projecting to $\Sigma$, we obtain a chain of pearls with a sphere in $X$.
Let $N_0$ be the number of constant spheres in $\Sigma$ and let $N_1$ be the
number of non-constant spheres in $\Sigma$, $N = N_0 + N_1$.
Notice that each non-constant sphere in $\Sigma$ has a 2-parameter family of
automorphisms, and each augmentation marked point can be moved in a 2-parameter
family. Furthermore, the holomorphic sphere $v_0$ also has a 2-parameter family of
automorphisms. 
By passing to a simple underlying chain of pearls as necessary, and applying
monotonicity and Proposition \ref{necklaces are regular} (to $\M^*_{k,(X,\Sigma)}((B;A_1,\ldots,A_N);x,p,J_W)$), we obtain
\begin{align*}
    I_X &\coloneqq M(p) + 2 \langle c_1(T\Sigma), A \rangle + 2 \left ( \langle
    c_1(TX), B \rangle - B \bullet \Sigma \right ) + M(x) - 2n +1 + N + 2k  \\
    &\ge 2N_1 + 2k + 2.
\end{align*}

We now combine the inequality with Equation \eqref{E:differentialWithW}:
\begin{align*}
    1 &= i(\widetilde p) + I_X - N +  \sum_{j=1}^k |\gamma_j|_0 \\
    1 &\ge i(\widetilde p) + 2N_1 + 2k + 2 - N_0 - N_1 + \sum_{j=1}^k
        |\gamma_j|_0\\
    0 &\ge i(\widetilde p) + N_1 + k + (k+1-N_0) + \sum_{j=1}^k |\gamma_j|_0.  
\end{align*}
Notice that we have $N_0 \le k+1$ since the first sphere in the chain of pearls
with a sphere in $X$ is allowed to be constant without any marked points. 
This observation together with Lemma \ref{L:indexAugmentationNonNeg} gives that
each term on the right-hand-side of the inequality is non-negative. 
It follows therefore that each term must vanish: 
$N_1 = 0$, $N_0 =1$, $k=0$ and $\widetilde p = \widecheck p$.
Notice that the Floer cylinder in $\R \times Y$ is contained in a single fibre
of $\R \times Y \to \Sigma$, so the marker condition coming from $\widecheck p$
can be interpreted as a marker condition on the holomorphic plane $v_0$ (via the 
parametrized Reeb orbit $\gamma$ in the statement). Without the marker condition, 
$v_0$ has Fredholm index 2, and thus with the marker constraint, it has index $1$.  

\end{proof}

\begin{remark}
    Similar analysis applied to continuation maps gives that our construction doesn't depend on the choices
    of almost complex structure $J_Y$, $J_W$ or of the auxiliary Morse functions and gradient-like vector fields. 

    In general, $\partial^2 = 0$ is obtained through analyzing gluing and
    considering the boundary of $1$--dimensional moduli spaces. 
        In our situation, if additionally $f_\Sigma$ and $f_W$ are assumed to
        be lacunary
        (i.e.~have no critical points of consecutive indices), 
     all contributions to the differential of an orbit $\widecheck p$ 
    are either of the form $\widehat q$ or constant orbits. 
    This automatically gives that $\partial^2 = 0$ for split symplectic homology.
 \end{remark}

Case (2) in Proposition \ref{P:Floer_in_R_times_Y} allows for the
existence of augmented configurations contributing to the symplectic
homology differential.  We will now adapt an argument originally due
to Biran and Khanevsky \cite{BiranKhanevsky} to show that
if $\overline W$ is a Weinstein domain (or equivalently, if $W$ is a Weinstein
manifold of finite-type), and $\Sigma$ has minimal Chern number at least $2$, 
then there can only be rigid augmentation planes
 if the isotropic skeleton has codimension at most 2 (in particular,
 $\dim_\R X = 2n\leq 4$).

\begin{lemma} \label{L:BiranKhanevskySkeleton}
If $W$ is a Weinstein domain with isotropic skeleton of real codimension
at least 3, then $X$ is symplectically aspherical if and only if $\Sigma$ is.

Furthermore, any symplectic sphere in $X$ is in the image of the inclusion
\[
\imath_* \colon \pi_2(\Sigma) \to \pi_2(X).
\]
\end{lemma}
\begin{proof}
The trivial direction is that if there exists a spherical class $A \in \pi_2(\Sigma)$ with $\omega(A) > 0$, then
$\imath_* A \in \pi_2( X)$ and still has positive area. 

We will now prove that any symplectic sphere in $X$ is in the image
of the inclusion. Let $C \subset W$ be the isotropic skeleton of $W$.
Notice that by following the flow of the Liouville vector field on $W$,
we obtain that $W \setminus C$ is symplectomorphic to a piece of the
symplectization $(-\infty, a) \times Y$. Thus, we have that $X \setminus
C$ is an open subset of a symplectic disk bundle over $\Sigma$ (the normal bundle to $\Sigma$ in $X$). We denote this
bundle's projection map by $\pi \colon X \setminus C \to \Sigma$.

Suppose $A \in \pi_2(X)$ is a spherical class with $\omega(A) > 0$.
By hypothesis, the skeleton $C$ is of codimension at least 3.  We may
therefore perturb $A$ in a neighbourhood of the skeleton so that it
does not intersect the skeleton $C$.  If $\iota\colon \Sigma \to X$ and
$j\colon X\setminus C \to X$ are the inclusion maps, then $\omega_\Sigma =
\iota^*\omega$ and $\iota\circ \pi$ is homotopic to $j$.  This implies
that $\omega_X(A) = \omega_\Sigma( \pi_* A)$, and the result follows.

\end{proof}

\begin{lemma} \label{L:BiranKhanevsky c1>1}
Suppose $W$ is a Weinstein domain with isotropic skeleton of real codimension at least 3 and $\Sigma$ has minimal Chern number at least $2$. Then, there do not exist any augmentation planes.
\end{lemma}
\begin{proof}
Recall from Proposition \ref{P:Floer_in_R_times_Y} that an augmentation 
plane in the class $B$ must have index $0$, 
so $0 = 2(\langle c_1(TX),B\rangle - B\bullet \Sigma - 1)$. Now, $\langle c_1(TX),B\rangle - B\bullet \Sigma = (\tau_X - K)\,\omega(B) \ge 1$.
Thus, the augmentation plane can only exist if there is a spherical class $B$ with $(\tau_X - K ) \,\omega(B) = 1$. 

By applying Lemma \ref{L:BiranKhanevskySkeleton}, 
we have $B = \imath_*A$, where $A \in \pi_2(\Sigma)$ is a spherical class in $\Sigma$. 

Now observe that $\langle c_1(T\Sigma), A \rangle + \langle c_1(N\Sigma), A
\rangle = \langle c_1(TX), A \rangle$, so we have $\langle c_1(T\Sigma), A
\rangle = (\tau_X - K) \omega_\Sigma(A).$
Hence, 
$1 = (\tau_X-K)\,\omega(A) = \langle c_1(T\Sigma),A\rangle$.
This contradicts the assumption that the minimal Chern number of $\Sigma$ is at least 2, so the augmentation plane cannot exist.
\end{proof}

\begin{remark}\label{rem:augmentation planes not near Sigma}
    Observe that this lemma applies more generally: if $\Sigma$ has minimal
    Chern number at least 2, then an augmentation plane cannot represent a
    spherical class in the image of $\imath_* \colon \pi_2(\Sigma) \to \pi_2(X)$. 

    Additionally, we have that an augmentation plane cannot have image entirely contained in $\varphi(\overline{\U})$. 
    Indeed, any holomorphic sphere contained in $\varphi(\overline{\U})$ will have index
    too high to be an augmentation plane: the $J_X$-holomorphic sphere
    with image in $\varphi(\overline{\U})$ automatically comes in a 2-parameter family
    (corresponding to the $\C^*$ action on the normal bundle to $\Sigma$).
    To make this argument more precise, we use our index computations. Suppose
    a sphere in $\varphi(\overline{\U})$ is an augmentation plane. 
    It then represents a class $\imath_*A$ with $A \in H_2(\Sigma)$.
    By the same index argument as in Lemma \ref{L:BiranKhanevsky c1>1}, 
    $1 = \langle c_1(T\Sigma), A \rangle$. 
    Since the image is assumed to be in $\varphi(\overline {\U})$, 
    the projection of the curve to $\Sigma$ is $J_\Sigma$-holomorphic. The index of this
    projection is given by $-4 + 2 \langle c_1(T\Sigma), A \rangle = -2$. This must
    be non-negative, however, since the projection is $J_\Sigma$-holomorphic, 
    and represents an indecomposable homology class. This contradiction then rules this
    possibility out.  
\end{remark}

\begin{remark}
The dichotomy between $\Sigma$ with minimal Chern number equal to 1 and bigger than 1 
is also explored in 
upcoming joint work of the first named author
with D.~Tonkonog, R.~Vianna and W.~Wu, studying the effect of the
Biran circle bundle construction on superpotentials of monotone
Lagrangian submanifolds \cite{DTVW}.
\end{remark}

\section{Orientations}
\label{sec:orientations}

In order to orient our moduli spaces, we will take the
point of view of coherent orientations, which is implemented in the Morse--Bott setting in 
\citelist{\cite{BourgeoisThesis},\cite{BOSymplecticHomology}}.
Some authors
\citelist{\cite{ZapolskyOrientations},\cite{SchmaeschkeOrientations}}
have used the alternative approach of canonical orientations. We
find it more straightforward to use coherent orientations in our
computations, especially since there are very few
choices involved. Notice also that if one has a canonical orientation scheme,
it is possible to extract a coherent orientation from this by making choices
of preferred orientations of certain capping operators. 

The geometry of our specific situation allows us to avoid some of the technical
difficulties present in the general Morse--Bott situation. In particular, we
have two key features that make our analysis more straightforward. First of all,
we do not have any ``bad'' orbits appearing in our setting
(recall from Section \ref{sec:gradings}
that, if we take a ``constant'' trivialization, the Conley--Zehnder index does not depend on covering
multiplicity).
For another, the
manifolds of orbits are all orientable, and are even oriented quite naturally by
the symplectic/contact structures that exist on them.

We now recall the general method for obtaining signs in Floer homology, as
first introduced in \cite{FloerHofer} and since generalized. First of all,
over the space of all Fredholm Cauchy--Riemann operators, there is a
determinant bundle. A choice of a section of this bundle then induces an
orientation on moduli spaces of holomorphic curves. This (together with some
additional choices in the Morse--Bott situation) allows us to orient 
all moduli spaces that occur in Floer
homology. On the other hand, configurations that are counted in the
differential have a natural $\R$--action on them by reparametrization, which
also induces an orientation on these moduli spaces. The sign of such a term in
the differential is positive if they agree and negative if they disagree.

\subsection{Orienting the moduli spaces of curves}

We now explain the first part of this method: how to orient the moduli spaces
of Floer punctured cylinders, but without considering their constraints coming from evaluation maps. 
We begin by sketching the situation for the non-degenerate case and then
discuss the modifications needed for the Morse--Bott situation.

First, consider all
Cauchy--Riemann operators on Hermitian vector bundles over  punctured spheres
with fixed trivializations near the punctures $E \to \dot S$,  
as described in Section \ref{S:sobolev_morse_bott}.
For a given Hermitian vector bundle with fixed
trivializations near the punctures and fixed, non-degenerate, asymptotic
operators, the space of all Cauchy--Riemann operators with these asymptotic
operators is contractible.
Each such operator induces a Fredholm operator $D \colon W^{1,p}(\dot S, E)
\to L^p(\dot S, \Lambda^{0,1}T^*\dot S \otimes E)$.
There exists a line bundle over this space of Fredholm operators called the 
\defin{determinant line bundle} and its fibre over an operator $D$ is given by
\[
\det D = (\Lambda^\text{max} \ker D ) \otimes_\R ( \Lambda^\text{max} \coker D)^* .
\]
(See for instance \cite{Zinger_determinant}.)
An orientation corresponds to a nowhere vanishing continuous section of this
determinant bundle over the space of Cauchy--Riemann operators (topologized in
a way compatible with the discrete topology on the space of asymptotic
operators). 
We then extend this to the case of Cauchy--Riemann operators with possibly
degenerate asymptotic operators that are Fredholm when acting on spaces of sections with exponential
weights $D \colon W^{1,p,\delta}(\dot S, E) \to L^{p,\delta}(\dot S,
\Lambda^{0,1}T^*\dot S\otimes E)$ by means of the conjugation
to an operator $D^\delta \colon W^{1,p}(\dot S, E) \to L^p(\dot S,
\Lambda^{0,1}T^*\dot S \otimes E)$ given in Definition \ref{def:delta_perturbed}.

In the case of non-degenerate operators, 
an orientation is \textit{coherent} if it respects the gluing operation on
Cauchy--Riemann operators, considered as operators $D \colon W^{1,p}(\dot S, E) \to
L^p(\dot S, \Lambda^{0,1}T^*\dot S \otimes E)$. 
Indeed, given two such operators 
\begin{align*}
    &D \colon W^{1,p}(\dot S, E) \to L^{p}(\dot S, \Lambda^{0,1}T^*\dot S
    \otimes E) \\
    \intertext{and}
    &D' \colon W^{1,p}(\dot S', E') \to L^{p}(\dot S', \Lambda^{0,1}T^*\dot S'
\otimes E')
\end{align*}
that have a matching  asymptotic operator at a positive puncture for $D$ and a
negative puncture for $D'$, we may form a glued surface $\dot S \hash
\dot S'$, a glued bundle $E \hash E' \to \dot S\hash \dot S'$, 
and a glued operator 
\[
    D \hash D' \colon W^{1,p}(\dot S \hash \dot S', E\hash E') \to L^{p}(\dot S\hash \dot S', \Lambda^{0,1}T^*(\dot S \hash \dot S') \otimes (E \hash E')).
\]
This
operator is not unique, but depends on a  contractible family of choices, in
particular on a gluing parameter. 
If the operators $D$ and $D'$ are both surjective, this is explicitly
constructed by a map $\ker D \oplus \ker D' \to \ker (D \hash D')$, which we
take to be orientation preserving.
After stabilizing operators that are not surjective, we obtain a map 
$\det D \otimes \det D' \to \det (D \hash D')$, which we require to be
orientation preserving in a coherent orientation scheme.
(See, for instance,
\cite{FloerHofer}*{Section 3} and \cite{BourgeoisMohnke}.) 
Thus, an orientation of $D$ and an orientation of $D'$ induce an orientation
of $D \hash D'$.

In our setting, we also require the coherent orientation to have the following two properties:
\begin{itemize}
\item the orientation of the direct sum of two operators is the tensor product of their orientations,  
\item the orientation of a complex linear operator is its canonical orientation.  
\end{itemize}

Finally, we extend this coherent orientation to operators with degenerate
asymptotics acting on weighted spaces 
$D \colon W^{1,p,\delta}(\dot S, E) \to L^{p,\delta}(\dot S, \Lambda^{0,1}T^*\dot S \otimes E)$ 
by means of the
conjugation to 
$D^\delta \colon W^{1,p}(\dot S, E) \to L^{p}(\dot S, \Lambda^{0,1}T^*\dot S
\otimes E)$ as in Definition  \ref{def:delta_perturbed}. Recall that for
fixed $\delta$, this conjugation is not unique, but depends on a contractible family of choices (of cut-off functions), so
the orientation of the determinant bundle does not depend on the choices.

From \citelist{\cite{FloerHofer} \cite{BourgeoisMohnke}
\cite{EGH}*{Section 1.8}}, a coherent orientation of the determinant
bundle over non-degenerate Cauchy--Riemann operators exists and may be
specified by choosing a preferred section of the determinant bundle over
certain \defin{capping operators}. These are operators whose domain is
the once punctured sphere $\C$ (where the puncture is positive), 
with a trivial Hermitian vector bundle
over them and specified asymptotic operator. In order to achieve the two
properties listed above, it suffices to enforce them on these capping
operators since the linear gluing operation described in \cite{FloerHofer}
respects direct sums and complex linearity.

We now describe how we orient capping operators for the relevant asymptotic
operators.
By Lemma \ref{L:linearizationFloer}, the linearized operator associated to 
a Floer cylinder $\tilde v$ is a compact perturbation of a split
operator $D^\C_{\tilde v} \oplus \dot D^\Sigma_w$, where $w = \pi_\Sigma \circ \tilde v$. There is also a
corresponding splitting of the asymptotic operators at the asymptotic limits. 
In particular, $\dot D^\Sigma_{ w}$ has complex linear asymptotic operators, and thus is a
compact perturbation of a complex linear Cauchy--Riemann operator. Hence,
its orientation is induced by the canonical one, and is independent of
choice of trivialization or of capping operator (which may always be taken to be
complex linear).

We are left with the task of orienting operators with the same asymptotic
operators as $D^\C_{\tilde v}$. 
By Lemma \ref{L:verticalOperatorTransverse}, if $\tilde v$ converges at both
$\pm \infty$ to a closed Hamiltonian orbit, with $\delta > 0$ sufficiently
small (see Remark \ref{delta small}), the operator
\[
    D^\C_{\tilde v} \colon W^{1,p, \delta}_{\mathbf V_0}(\R \times S^1, \C) \to
    L^{p, \delta}(\Hom^{0,1}(T(\R \times S^1), \C))
\]
has Fredholm index 1, is surjective and its kernel 
contains an element that can be identified with the Reeb vector field. 
We may identify the kernel (and cokernel) of this operator with those of
\[
    D^\C_{\tilde v} \colon W^{1,p, -\delta}(\R \times S^1, \C) \to
    L^{p, -\delta}(\Hom^{0,1}(T(\R \times S^1), \C)).
\]

At $\pm \infty$, the $-\delta$-perturbed asymptotic operators (see Definition \ref{def:delta_perturbed})
associated to $D^\C_{\tilde v}$ are  
\begin{equation}
\mathbf{A}_\pm := - \left ( J \frac{d}{dt} 
	+ \begin{pmatrix} h''(\e^{b_\pm}) \e^{b_\pm} \pm \delta &0 \\ 0 & \pm \delta \end{pmatrix} \right ) 
\label{asymptop}
\end{equation}
(The asymptotic operator at a Reeb orbit at $-\infty$ is just $-J\frac{d}{dt}$
and is $-\delta$ perturbed to give $-(J\frac{d}{dt} - \delta)$.)

We now choose capping operators for the $\mathbf{A}_\pm$, which determines an orientation of 
$D^\C_{\tilde v}$ by the coherent orientation scheme.  

\begin{lemma}
    Let $\delta > 0$ be sufficiently small.
 There is a choice of capping operators with orientations for the asymptotic operators
 $\mathbf{A}_\pm$ above, such that the orientation induced on 
 \[
     D^\C_{\tilde v} \colon W^{1,p, -\delta}(\R \times S^1, \C) 
     \to L^{p, -\delta}(\Hom^{0,1}(T(\R \times S^1), \C))
 \]
 identifies the Reeb vector field as positively oriented. (Recall that we have
 identified $\R \partial_r \oplus \R R$ with $\C$.)
\label{Dv orient Reeb}
\end{lemma}

\begin{proof}
Recall that for each $b_k > 0$ satisfying $h'(\e^{b_k}) = k \in \Z_+$, we have a $Y$-parametric family of 1-periodic Hamiltonian orbits.
We can associate to each of these orbits two operators $A_\pm$, as in \eqref{asymptop}. We will define capping operators 
$$
\Phi^\pm_k \colon W^{1,p}(\C,\C) \to L^{p}(\Hom^{0,1}(T(\C),\C))
$$
with these asymptotic operators.

We first define two families of auxiliary Fredholm operators. For each $k>0$,
$$
\Psi_k \colon W^{1,p}(\R\times S^1,\C) \to L^{p}(\Hom^{0,1}(T(\R \times S^1),\C))
$$
is an operator given by 
$$
\Psi_k (F)(\partial_s) = F_s + i F_t + 
             \begin{pmatrix} a(s) - \delta  &0 \\ 0 & -\delta \end{pmatrix} F
$$
where the function $a\colon \R \to \R$ is such that $\lim_{s\to -\infty} a(s) = h''(e^{b_1}) e^{b_1}$ and $\lim_{s\to +\infty} a(s) = h''(e^{b_k}) e^{b_k}$. 
Let now 
$$
\Xi_k \colon W^{1,p}(\R\times S^1,\C) \to L^{p}(\Hom^{0,1}(T(\R \times S^1),\C))
$$
be an operator given by 
$$
\Xi_k (F)(\partial_s) = F_s + i F_t + 
             \begin{pmatrix} h''(e^{b_k}) e^{b_k} + \delta(s)  &0 \\ 0 & \delta(s) \end{pmatrix} F
$$
where $\delta\colon \R \to \R$ is such that $\lim_{s\to -\infty} \delta(s) = -\delta <0$ and $\lim_{s\to +\infty} \delta(s) = \delta$.

The operators $\Psi_k$ are isomorphisms (in particular, they are canonically oriented). This follows from an argument analogous to the proof of Lemma \ref{L:verticalOperatorTransverse}. 
A version of the same argument implies that the operators $\Xi_k$ are Fredholm of index 1 and surjective, and that their kernels contain elements that can be identified with the Reeb vector field. 

Now, pick an arbitrary capping operator $\Phi_1^-$. Define $\Phi_k^-$ for $k>1$ by gluing $\Phi_1^- \# \Psi_k$. Define $\Phi_k^+$ for all $k>0$ by gluing $\Phi_k^- \# \Xi_k$. For these choices of capping operators, $D^\C_{\tilde v}$ 
are oriented in the direction of the Reeb flow, as wanted.
\end{proof}

We now analyze how, for $\delta > 0$ and small, a coherent orientation scheme relates
the orientations of $D \colon W^{1,p,-\delta}(\dot S, \C) \to
L^{1,p,-\delta}(\dot S, T^*\dot S \otimes \C)$
and of $D \colon W^{1,p,\delta}_{\mathbf{V}}(\dot S,
\C) \to 
L^{1,p,-\delta}(\dot S, T^*\dot S \otimes \C)$ (where $\mathbf{V}$ is the collection of 
kernels of the asymptotic operators at the punctures). Recall from Lemma \ref{L:exponentialGrowth}
that their determinant bundles are isomorphic. 
\begin{lemma} \label{L:kernel asymptotic operator}
 Let $\mathbf{A}$ be a degenerate asymptotic operator, let $V$ be its kernel and let
 $\delta > 0$ be chosen small enough that $[-\delta, \delta] \cap \sigma(\mathbf{A}) = \{
 0 \}$. 

 Then, the kernel of $\frac{\partial}{\partial s} - \mathbf{A} \colon
 W^{1,p,-\delta}(\R \times S^1, \C^n) \to L^{p,-\delta}(\R \times S^1, \C^n)$
 consists of constant maps with values in $V$.
\end{lemma}
\begin{proof}
    The proof follows from expanding $L^2(S^1, \C^n)$ in a Hilbert
    basis given by eigenvectors of the asymptotic operator $A$ seen as an
    elliptic self-adjoint unbounded operator on $L^2(S^1, \C^n)$. Then, the
    kernel of $\frac{\partial}{\partial s} - \mathbf{A}$ is spanned by solutions of the
    form $\e^{\lambda s}v(t)$, where $v(t)$ is an eigenfunction for the
    eigenvalue $\lambda$. Since we require exponential growth of rate $\delta$,
    this forces $-\delta < \lambda < \delta$. The result now follows since
    $0$ is the only such eigenvalue.

    We thank Chris Wendl and Richard Siefring for suggesting this argument. See
    also \cite{Nicolaescu_Lectures_geometry_manifolds}*{Theorem 10.4.19}.
\end{proof}

From this, we are able to revisit Lemma \ref{L:exponentialGrowth}. A more
general result is true, but we only need this case of the following lemma:
\begin{lemma}
Let $D$ be a Cauchy--Riemann operator on a punctured cylinder $\dot S = \R \times
S^1 \setminus \Gamma$. Fix a puncture $z_0 \in \{ \pm \infty \}$.  Assume the
asymptotic operators at the punctures in $\Gamma$ are complex linear.

Let $\mathbf{\delta}$ and $\mathbf{\delta'}$ be vectors of sufficiently small 
weights so that the differential operator induces a Fredholm operator on $W^{1,p,\mathbf{\delta}}$
and on $W^{1,p,\mathbf{\delta'}}$, and $\delta_{z_0} > 0 $ and $\delta_{z_0}' < 0$,
the interval $[\delta_{z_0}', \delta_{z_0}] \cap \sigma(\mathbf{A}_{z_0}) = \{ 0 \}$,
and for each $z \in \Gamma \cup \{ \pm \infty \}$ with $z \ne z_0$, the weights $\delta_z = \delta'_z$.

Let $\mathbf{V}$ be the trivial vector space at each puncture other than $z_0$ and let $V_{z_0}$ be the kernel of the asymptotic operator $\mathbf{A}_{z_0}$ at $z_0$.

Then, we may identify the determinant bundles of the operators
\begin{align*}
D_\delta &\colon W^{1,p, \mathbf \delta}_{\mathbf{V}}( \dot S, E)  \to L^{p, \mathbf{\delta}}(\dot S, \Lambda^{0,1} T^*\dot S \otimes E)\\
D_{\delta'} &\colon W^{1,p, \mathbf \delta'}( \dot S, E)  \to L^{p, \mathbf{\delta'}}(\dot S, \Lambda^{0,1} T^*\dot S \otimes E)
\end{align*}
by Lemma \ref{L:exponentialGrowth}. 
This choice of orientation of the determinant bundle of $D_{\delta'}$ then
induces an orientation on 
\[ 
    D_\delta \colon W^{1,p, \mathbf \delta}_{\mathbf{V}}( \dot S, E)  \to L^{p,
    \mathbf{\delta}}(\dot S, \Lambda^{0,1} T^*\dot S \otimes E)  \ .
\]
\end{lemma}
\begin{proof}

    By Lemma \ref{L:kernel asymptotic operator}, $V_{z_0}$ may be identified
    with the kernel of the Cauchy--Riemann operator on the cylinder given by
    $\partial_s - \mathbf{A}_{z_0} \colon W^{1,p,-\delta}(\R \times S^1, \C^n)
    \to L^{p,-\delta}(\R \times S^1, \C^n)$. 
    By the coherent orientation, this operator is oriented and hence induces an
    orientation on the space of sections $V_{z_0}$. 

    The result now follows by the gluing property of coherent orientations. 
    
    In particular, if $z_0$ is a positive puncture, this identifies 
the determinant bundle of 
\[ 
    D_\delta \colon W^{1,p, \mathbf \delta}_{\mathbf{V}}( \dot S, E)  \to L^{p, \mathbf{\delta}}(\dot S, \Lambda^{0,1} T^*\dot S \otimes E) 
\]
with the tensor product of the determinant bundle of 
\[
    D_\delta \colon W^{1,p, \mathbf \delta}(\dot S, E)  \to L^{p, \mathbf{\delta}}(\dot S, \Lambda^{0,1} T^*\dot S \otimes E)  
\]
with the determinant bundle of the kernel of $\mathbf{A}_{z_0}$.
The order is reversed if the puncture is negative. 
\end{proof}

We remark here that our asymptotic operators are either complex linear or have a
kernel that is naturally identified with the Reeb vector field. We have already
verified that our capping operators are oriented in a way that is consistent
with this.

From this, 
we have now oriented the operators $D^\C_{\tilde v}$ and $\dot D^\Sigma_w$
acting on spaces of sections free to move in their Morse--Bott families. 
Since the linearized Floer operator is a compact perturbation of their direct sum, 
we get induced orientations on the transverse moduli spaces of Floer cylinders 
with punctures. 

\subsection{Orientations with constraints}

We have now explained how to orient all of the moduli spaces of punctured
cylinders with ends free to move in the corresponding Morse--Bott families of
orbits. This is not yet sufficient to orient our moduli spaces of cascades.  
The additional ingredient
necessary is to orient moduli spaces of holomorphic curves with
\textit{constraints} on their asymptotic evaluation maps, with the asymptotic
evaluation map constrained to lie in stable/unstable manifolds of critical
points of the auxiliary Morse functions, or, in the case of multilevel cascades,
constrained to lie in [flow] diagonals in manifolds of orbits.

Let us begin by stating the convention in \cite{BOSymplecticHomology} for how to orient a fibre sum (which agrees with 
\cite{JoyceCorners}*{Convention 7.2.(b)}). 

\begin{definition}
Given linear maps between oriented vector spaces $f_i\colon V_i \to W$, $i = 1,2$, such that $f_1 - f_2 \colon V_1\oplus V_2 \to W$ is surjective, the {\em fibre sum orientation} on $V_1 \oplus_{f_i} V_2 = \ker (f_1 - f_2)$ is such that 
\begin{enumerate}
 \item $f_1 - f_2$ induces an isomorphism $(V_1 \oplus V_2) / \ker (f_1 - f_2) \to W$ which changes orientations by $(-1)^{\dim V_2 . \dim W}$,
 \item where a quotient $U/V$ of oriented vector spaces is oriented in such a way that the isomorphism $V \oplus (U/V) \to U$ (associated to a section of the quotient short exact sequence) preserves orientations. 
\end{enumerate}
\label{fibre sum orient}
\end{definition}

One key property of this orientation convention for fibre sums is that it is associative (this property 
specifies the orientation convention almost uniquely, as explained in \cite{JoyceCorners}*{Remark 7.6.iii} 
and \cite{RamshawBasch}; this was pointed out to us by Maksim Maydanskiy). 

To orient our constrained moduli spaces, we follow the point of view
in the literature \citelist{\cite{BourgeoisThesis}
    \cite{BOSymplecticHomology}\cite{BiranCornea} \cite{FOOO}
\cite{SchmaeschkeOrientations} }. Specifically, we begin by orienting  
moduli spaces of unconstrained Floer cylinders as in the previous section, 
by the chosen coherent orientation of
Cauchy--Riemann operators with free asymptotics. We also fix orientations on all stable and unstable 
manifolds of the relevant manifolds of orbits (see the next section for more details), as well as 
on the relevant diagonals and flow-diagonals. Then, we orient the moduli
spaces of Floer cylinders with cascades by the rule that the asymptotic
constraints are obtained as fibre products over descending and ascending
manifolds of the Morse functions in the manifolds of orbits $Y_k$ and $W$ and
as fibre products over flow diagonals and diagonals in $Y_k \times Y_k$ and in $Y \times Y$. 
The fiber products are oriented using Definition \ref{fibre sum orient}. 
For this scheme to induce a differential, we then need the orientations of the various
boundary components of these moduli spaces of cascades to be consistent.

    Observe that in a general Morse--Bott situation, there are additional orientation
    difficulties that are not present in our problem. Specifically, 
    \citelist{\cite{BourgeoisThesis} \cite{BOSymplecticHomology}\cite{SchmaeschkeOrientations}} have
    to deal with parametric families of asymptotic operators that move in the
    space of asymptotic operators of fixed degeneracy. In our problem, the
    asymptotic operators of $D^\C_{\tilde v}$ are constant on each Morse--Bott
    family of orbits, dramatically simplifying the problem to consider. 

    We also notice another key feature regarding cascades: suppose that 
    $\tilde v_-$ and $\tilde v_+$ are two (punctured) cylinders so the
    asymptotic limit of $\tilde v_-$ at $+\infty$ matches the asymptotic limit
    of $\tilde v_+$ at $-\infty$, as in the center of Figure \ref{fig:cascades orientation}. Let $x$ denote the limit of $\tilde v_-$ at
    $-\infty$ and let $y$ denote the limit of $\tilde v_+$ at $+\infty$. 
    This configuration can arise in two ways. It appears in the compactification of the
    space of cylinders with negative end in the Morse--Bott family of orbits
    that includes $x$ and positive end in the Morse--Bott family of orbits that
    includes $y$ (right in the figure). It also appears as the limit of a two level cascade as the
    length of the finite length flow line goes to $0$ (left in the figure). The key point of a
    Morse--Bott orientation scheme is that the orientations should be such that
    the broken configuration of $(\tilde v_-, \tilde v_+)$ should appear as an
    interior point of the moduli space.

    We now sketch the key point that is developed in greater detail in
    \cite{SchmaeschkeOrientations}*{Section 8.4} (and in greater generality,
    and with totally real boundary conditions). 
    The linearized operator that describes the tangent space to the moduli space
    of pairs $(\tilde v_-, \tilde v_+)$ with matching asymptotic $\tilde
    v_-(+\infty) = \tilde v_+(-\infty) \in S_0$ is naturally given by:
    \begin{equation}\label{eqn:matching asymptotics linearization}
    \begin{split} 
        D_{\tilde v_-} \oplus D_{\tilde v_+} \colon 
        W^{1,p,\delta}&(\dot S_-, E_-)
        \oplus \Delta \oplus 
        W^{1,p,\delta}(\dot S_+, E_+)\\
        &\longrightarrow
        L^{p, \delta}(\dot S_-, \Lambda^{0,1}T^*\dot S_- \otimes E_-)
        \oplus 
        L^{p, \delta}(\dot S_+, \Lambda^{0,1}T^*\dot S_+ \otimes E_+)
    \end{split}
\end{equation}
    where $\delta > 0$ imposes exponential decay (for a weight chosen 
    smaller than the spectral gap, as in Remark \ref{delta small}) 
    and where $\Delta \subset TS_0 \oplus TS_0$ is the diagonal. 
    Notice that $\Delta$ is naturally isomorphic to $TS_0$ and can be oriented
    as the image of the map $x \to (x,x)$. Let $\mathbf{A}$ be the (degenerate)
    asymptotic operator at the shared orbit\footnote{Technically, to define
        this requires a trivialization of $\xi$ along this orbit. In our setting 
       the asymptotic operator is complex linear in the $\xi$ direction, so
    this choice does not matter.}.
    After the conjugation described in Definition \ref{def:delta_perturbed}, 
    we obtain non-degenerate operators $\hat D_{\tilde v_-}$ and $\hat D_{\tilde v_+}$ that have
    asymptotic operators $\mathbf{A}+\delta$ and $\mathbf{A}-\delta$ respectively. If we consider
    now a $\delta$--perturbed Cauchy--Riemann operator coming from the
    linearization at a trivial orbit cylinder, $T \coloneqq \partial_s -
    \mathbf{A} + f(s)$, with $f(s) = +\delta$ near $-\infty$ and $f(s) = -\delta$ near
    $+\infty$, we obtain an operator with trivial cokernel and whose kernel is
    identified with $TS_0$ (by Lemma \ref{L:kernel asymptotic operator}). This allows us to identify 
    the $D_{\tilde v_-} \oplus D_{\tilde v_+}$ in \eqref{eqn:matching asymptotics linearization} 
    with the triple $(\hat D_{\tilde v_-}, T, \hat D_{\tilde v_+})$. This
    triple can be glued to obtain the linearization of the space of cylinders
    with one fewer cascade, so this is naturally oriented as the boundary.

\begin{figure}
  \begin{center}
    \def\svgwidth{.6\textwidth}
\begingroup%
  \makeatletter%
  \providecommand\color[2][]{%
    \errmessage{(Inkscape) Color is used for the text in Inkscape, but the package 'color.sty' is not loaded}%
    \renewcommand\color[2][]{}%
  }%
  \providecommand\transparent[1]{%
    \errmessage{(Inkscape) Transparency is used (non-zero) for the text in Inkscape, but the package 'transparent.sty' is not loaded}%
    \renewcommand\transparent[1]{}%
  }%
  \providecommand\rotatebox[2]{#2}%
  \ifx\svgwidth\undefined%
    \setlength{\unitlength}{426.53707394bp}%
    \ifx\svgscale\undefined%
      \relax%
    \else%
      \setlength{\unitlength}{\unitlength * \real{\svgscale}}%
    \fi%
  \else%
    \setlength{\unitlength}{\svgwidth}%
  \fi%
  \global\let\svgwidth\undefined%
  \global\let\svgscale\undefined%
  \makeatother%
  \begin{picture}(1,0.45628681)%
    \put(0,0){\includegraphics[width=\unitlength,page=1]{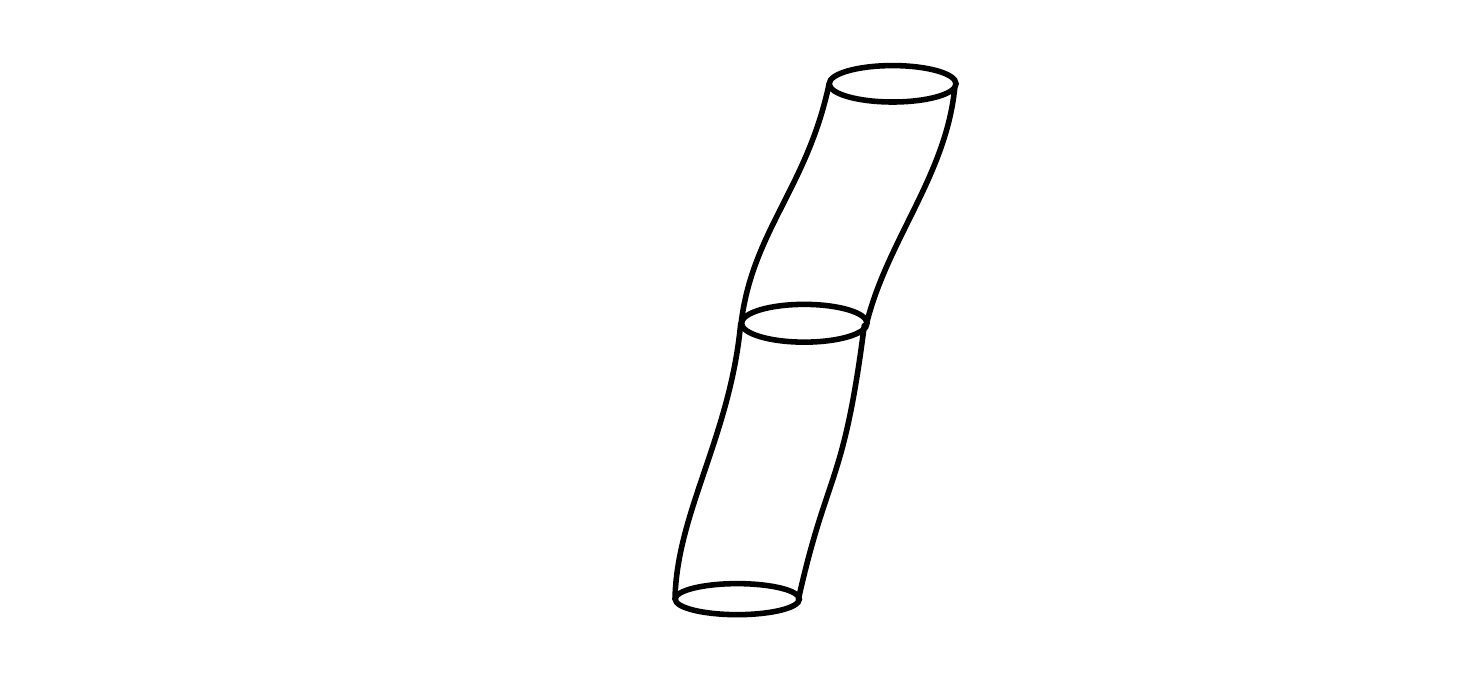}}%
    \put(0.4980138,0.11899484){\color[rgb]{0,0,0}\makebox(0,0)[lb]{\smash{$\tilde v_-$}}}%
    \put(0.55046118,0.30521519){\color[rgb]{0,0,0}\makebox(0,0)[lb]{\smash{$\tilde v_+$}}}%
    \put(0.3488769,0.22302299){\color[rgb]{0,0,0}\makebox(0,0)[lb]{\smash{$\rightsquigarrow$}}}%
    \put(0.47799011,0.00285731){\color[rgb]{0,0,0}\makebox(0,0)[lb]{\smash{$x$}}}%
    \put(0.5885062,0.42936213){\color[rgb]{0,0,0}\makebox(0,0)[lb]{\smash{$y$}}}%
    \put(0,0){\includegraphics[width=\unitlength,page=2]{cascade_orientation.pdf}}%
    \put(0.70148403,0.22302299){\color[rgb]{0,0,0}\makebox(0,0)[lb]{\smash{$\leftsquigarrow$}}}%
  \end{picture}%
\endgroup%
   \end{center}
  \caption{An interior point of the moduli space of cascades.} 
   \label{fig:cascades orientation}
\end{figure}

    Similarly, by taking the fibre product over the diagonal, we see this as
    oriented with the opposite orientation of the boundary of the flow diagonal,
    which is oriented as $[0, \infty) \times \Delta$. This allows us to conclude
    that our orientation scheme is coherent with respect to the additional
    breakings that appear in the Morse--Bott setting.

\subsection{A calculation of signs}

Having now explained the general framework of our orientations, let us now give
an explicit description of the signs associated to a Floer cylinder with
cascades contributing to the differential. 
By Propositions
\ref{P:Floer_in_R_times_Y} and \ref{P:Floer_also_in_W}, there are
four types of contributions to the differential, referred to as
Cases 0 through 3.  We will explain how to determine the signs in each
case.

In Case 0, we have gradient flow lines of Morse--Smale pairs $(f,Z)$ on manifolds of orbits, which can either be 
$(f_Y,Z_Y)$ on $Y$ or $(-f_W,-Z_W)$ on $W$ (see Definition \ref{def:FloerCylinderWithCascades}). Let us stipulate the 
orientation conventions for Morse homology that we will use. The Morse complex of a Morse--Smale pair $(f,Z)$ on a manifold $S$ is 
generated by critical points of $f$ and the differential $\partial_{f}$ is such that, given $p\in \Crit(f)$, 
\begin{equation}
\partial_{f}(p) = \sum_{\substack{q \in \Crit(f)\\ \ind_f(p) - \ind_f(q) = 1}} \#\left(\big(W^s_S(q) \cap W^u_S(p) \big) / \R \right) \, q.
\label{Morse diff}
\end{equation}
In this formula, we use the notation of \eqref{(un)stable} for critical manifolds of $Z$. Note that they intersect transversely, 
by the Morse--Smale assumption. 

We need to make sense of the signed count in the formula. 
We will be interested in the cases where $S$ is $\Sigma$, $Y$ or $W$,
all of which are oriented (by their chosen symplectic, contact and symplectic forms, respectively). If we fix an orientation on a 
critical manifold at a critical point $p$, then we get an induced orientation on the other critical manifold, by imposing that
the splitting 
\begin{equation}
 T_p W^s_S(p) \oplus T_p W^u_S(p) \cong T_p S
 \label{orient stable and unstable}
\end{equation}
preserves orientations. 
Pick orientations on all unstable manifolds of $\Sigma$ and $W$. For all $p\in \Crit(f_\Sigma)$, we will assume that the 
orientations on critical submanifolds of $\Sigma$ and $Y$
are such that the 
restrictions of $\pi_\Sigma\colon Y \to \Sigma$ to
\begin{equation}
 W^u_Y(\widecheck p) \to W^u_\Sigma(p) \qquad \text{and} \qquad W^s_Y(\widehat p) \to W^s_\Sigma(p)
 \label{orient point lift}
\end{equation}
are orientation-preserving diffeomorphisms. 

If $\gamma\colon \R\to S$ is a rigid flow line from $q$ to $p$ (critical
points of consecutive indices), then it induces a diffeomorphism onto
its image 
$$ \gamma(\R) \subset W^s_S(q) \cap W^u_S(p).  $$ 
The source $\R$ has its usual orientation, corresponding to increasing values
of time in the orbit $\gamma$, and the image 
$W^s_S(q) \cap W^u_S(p)$
is a transverse intersection, and can be oriented in such a way that
the splitting 
\begin{equation} 
    T W^u_S(p) \cong T( W^s_S(q) \cap W^u_S(p)) \oplus T W^u_S(q) 
\end{equation} 
is orientation-preserving
(see \cite{HutchingsMorse}*{Equation (2)} for a similar convention).
The flow line $\gamma$ contributes to 
$\#\left(\big(W^s_S(q) \cap W^u_S(p) \big) / \R \right)$ 
in \eqref{Morse diff} positively iff it
preserves orientations. The Case 0 contribution of a flow line to the
symplectic homology differential is the same as its contribution to the
Morse differential.

\ 

Let us now consider Case 1, which is more interesting. Recall that such configurations consist of a 
Floer cylinder without augmentation punctures, together with two flow lines of $Z_Y$ at the ends. 
Suppose that the Floer cylinder converges to orbits of multiplicities $k_\pm$ at $\pm\infty$. 
Such a cylinder is an element of the space 
$\MM^*_{H,0,\R\times Y;k_-,k_+}(A;J_Y)$, for some $A\in H_2(\Sigma;\Z)$. 

Cylinders with cascades that contribute to the differential in Case 1 are elements of spaces 
$\MM^*_{H,1}(\widehat q_{k_-},\widecheck p_{k_+};J_W)$, for $p\neq q\in \Crit(f_\Sigma)$ (recall the notation in \eqref{union}). These spaces are unions of fibre products
\begin{equation}
 W^s_{Y}(\widehat q) \times_{\ev} \MM^*_{H,0,\R\times Y;k_-,k_+}(A;J_Y) \times_{\ev} W_{Y}^u(\widecheck p)
\label{fib prod}
\end{equation}
defined with respect to the inclusion maps
\begin{align*}
W^s_{Y}(\widehat q), W_{Y}^u(\widecheck p) &\to Y
\end{align*}
and the evaluation maps from \eqref{ev_Y}
\begin{align*}
\tilde{\ev}_Y \colon \MM^*_{H,0,\R\times Y;k_-,k_+}(A;J_Y) &\to Y\times Y, %
\end{align*}
for appropriate $A\in H_2(\Sigma;\Z)$. 
Recall the discussion of Case 0 above, which included a specification of orientations on all 
critical submanifolds of $Y$.
We use the fibre sum convention (in Definition \ref{fibre sum orient}) to 
orient the fibre product \eqref{fib prod}.

Observe now that if $\MM^*_{H,1}(\widehat q_{k_-},\widecheck p_{k_+};J_W)$ is one-dimensional, 
then its tangent space at every point is generated by the infinitesimal translation in the 
$s$-direction on the domain of the Floer cylinder. 
This induces an orientation on $\MM^*_{H,1}(\widehat q_{k_-},\widecheck p_{k_+};J_W)$. Comparing this 
orientation with the one defined above with the fibre sum rule, we get the sign of such a contribution 
to the split symplectic homology differential. 

\

We adapt the argument above to associate a sign to a contribution to the differential in Case 2. 
Such cascades are elements of spaces 
$\MM^*_{H,1}(\widehat p_{k_-},\widecheck p_{k_+};J_W)$, for $p\in \Crit(f_\Sigma)$.
The analogue of \eqref{fib prod} is now (using the notation of \eqref{ev_Y aug}):
\begin{equation}
 W^s_{Y}(\wh p) \times_{\ev} \big(\M_X^*(B;J_W) 
    \times_{\tilde \ev} \MM^*_{H,1,\R\times Y;k_-,k_+}(0;J_Y) \big) 
    \times_{\ev} W_{Y}^u(\wc p).
\label{fib prod2}
\end{equation}
Notice that in this case we also need capping operators for augmentation punctures. 
The asymptotic operators at such punctures are
$$
\mathbf{A} = - J \frac{d}{dt},
$$
which are complex linear. We may therefore take (canonically oriented) complex linear capping
operators at these punctures. 
We can now orient the fibre product \eqref{fib prod2} using the fibre sum convention. 
The sign of such a contribution to the differential
is obtained by comparing this orientation with the one induced 
by $s$-translation on the domain of the punctured Floer
cylinder in $\MM^*_{H,1,\R\times Y;k_-,k_+}(0;J_Y)$.

\ 

Finally, Case 3 Floer cylinders with cascades that contribute to the differential are elements of 
$\MM^*_{H,1}(x,\widecheck p_{k_+};J_W)$, 
which are unions of fibre products 
\begin{equation*}
 W^u_{W}(x) \times_{\ev} \left(\MM^*_{H,0,W;k_+}((B;0);J_W) \times_{\ev} \widetilde \Delta \right) \times_{\ev} W_{Y}^u(\wc p).
\end{equation*}
The relevant evaluation maps are given by factors of $\tilde \ev_{W,Y}$ in \eqref{ev_W,Y}.
It will be useful to write an alternative description of this fibre product. 
Recall that $\MM^*_{H,0,W;k_+}((B,0);J_W)$ is a space of pairs 
$(\tilde v_0,\tilde v_1)$ with the following properties. The simple $J_W$-holomorphic cylinder $\tilde v_0\colon \R \times S^1 \to W$ has a 
removable singularity at $-\infty$, defining a pseudoholomorphic sphere in $X$ in class $B\in H_2(X;\Z)$,
with order of contact $k_+=B\bullet \Sigma$ at $\infty$. Denote the space of such cylinders by $\M^*_{H}(B;J_W)$. 
The Floer cylinder $\tilde v_1\colon \R \times S^1 \to \R \times Y$ converges at $+\infty$ to a Hamiltonian 
orbit of multiplicity $k_+$ and at $-\infty$ to a Reeb orbit of the same multiplicity in $\{-\infty\} \times Y$. 
It projects to a constant in $\Sigma$. Denote the space of such cylinders by $\M^*_{H,k_+}(0;J_Y)$.
We have evaluation maps 
$$
(\ev^1_-,\ev^1_+) \colon \M^*_{H}(B;J_W) \to W \times Y \qquad \text{and} \qquad  (\ev^2_-,\ev^2_+) \colon \M^*_{H,k_+}(0;J_Y) \to Y \times Y
$$
and can write 
\begin{equation*} 
\MM^*_{H,0,W;k_+}((B,0);J_W) = \M^*_{H}(B;J_W) \times \M^*_{H,k_+}(0;J_Y)
\end{equation*}
and 
$$
\MM^*_{H,0,W;k_+}((B,0);J_W) \times_{\ev} \widetilde \Delta = \M^*_{H}(B;J_W)  \prescript{}{\ev^1_+\,}\times\prescript{}{\ev^2_-}{}  \M^*_{H,k_+}(0;J_Y).
$$
We now rewrite the Case 3 contributions to the differential as 
\begin{equation}
 W^u_{W}(x) \times_{\ev^1_-} \left(\M^*_{H}(B;J_W)  \prescript{}{\ev^1_+\,}\times\prescript{}{\ev^2_-}{}  \M^*_{H,k_+}(0;J_Y) \right) \times_{\ev^2_+} W_{Y}^u(\wc p),
\label{fib prod3}
\end{equation}
which is oriented using coherent orientations on the spaces of cylinders and the fibre sum orientation convention.

The space $\M^*_{H}(B;J_W)  \prescript{}{\ev^1_+\,}\times\prescript{}{\ev^2_-}{}  \M^*_{H,k_+}(0;J_Y)$ has an action of $\R_1 \times \R_2$, where the 1-dimensional real 
vector space $\R_1$ acts by $s$-translation on the domain of the cylinder in $W$ and $\R_2$ acts by $s$-translation on the domain of the cylinder in $\R\times Y$.  
The sign of a Case 3 contribution to the differential is obtained by comparing the coherent/fibre product orientation on 
\eqref{fib prod3} with the usual orientation on $\R_1 \times \R_2$, corresponding to $s$-translation on the domain of $\tilde v_0$
followed by $s$-translation on the domain of $\tilde v_1$.


\def\cprime{$'$}

\end{document}